\pdfoutput=1
\RequirePackage[l2tabu, orthodox]{nag} 
\documentclass[12pt,a4paper,reqno]{amsart} 
\allowdisplaybreaks
\usepackage{amsmath,amsthm,amssymb,mathrsfs,mathtools,bm,eucal,tensor} 
\usepackage{microtype,fixltx2e,lmodern} 
\usepackage{enumerate,comment,braket,xspace,tikz-cd} 
\usepackage[utf8]{inputenc} 
\usepackage[T1]{fontenc} 
\definecolor{linkcolor}{HTML}{005050}
\usepackage[centering,vscale=0.7,hscale=0.7]{geometry}
\usepackage[hidelinks]{hyperref}
\usepackage[capitalize]{cleveref}

\theoremstyle{plain}
\newtheorem{thm}{Theorem}[section]
\newtheorem{lem}[thm]{Lemma}
\newtheorem{prop}[thm]{Proposition}

\newtheorem{cor}[thm]{Corollary}

\theoremstyle{definition}
\newtheorem{defin}[thm]{Definition}
\newtheorem{notation}[thm]{Notation}
\theoremstyle{remark}
\newtheorem*{rem*}{Remark}

\newtheorem{rem}[thm]{Remark}
\numberwithin{equation}{section}

\newcommand{\C}{\mathbb C}

\newcommand{\Z}{\mathbb Z}

\newcommand{\rB}{\mathrm B}
\newcommand{\rH}{\mathrm H}

\newcommand{\rL}{\mathrm L}
\newcommand{\rR}{\mathrm R}

\newcommand{\cA}{\mathcal A}

\newcommand{\cC}{\mathcal C}
\newcommand{\cD}{\mathcal D}
\newcommand{\cE}{\mathcal E}
\newcommand{\cF}{\mathcal F}

\newcommand{\cG}{\mathcal G}

\newcommand{\cJ}{\mathcal J}
\newcommand{\cK}{\mathcal K}

\newcommand{\cO}{\mathcal O}

\newcommand{\cS}{\mathcal S}
\newcommand{\cT}{\mathcal T}
\newcommand{\cU}{\mathcal U}
\newcommand{\cV}{\mathcal V}
\newcommand{\cW}{\mathcal W}
\newcommand{\cX}{\mathcal X}
\newcommand{\cY}{\mathcal Y}
\newcommand{\cZ}{\mathcal Z}
\DeclareFontFamily{U}{BOONDOX-calo}{\skewchar\font=45 }
\DeclareFontShape{U}{BOONDOX-calo}{m}{n}{<-> s*[1.05] BOONDOX-r-calo}{}
\DeclareFontShape{U}{BOONDOX-calo}{b}{n}{<-> s*[1.05] BOONDOX-b-calo}{}
\DeclareMathAlphabet{\mathcalboondox}{U}{BOONDOX-calo}{m}{n}

\newcommand{\bD}{\mathbf D}
\newcommand{\bP}{\mathbf P}
\newcommand{\bQ}{\mathbf Q}

\newcommand{\bDelta}{\bm{\Delta}}

\makeatletter
\let\save@mathaccent\mathaccent
\newcommand*\if@single[3]{%
	\setbox0\hbox{${\mathaccent"0362{#1}}^H$}%
	\setbox2\hbox{${\mathaccent"0362{\kern0pt#1}}^H$}%
	\ifdim\ht0=\ht2 #3\else #2\fi
}
\newcommand*\rel@kern[1]{\kern#1\dimexpr\macc@kerna}
\newcommand*\widebar[1]{\@ifnextchar^{{\wide@bar{#1}{0}}}{\wide@bar{#1}{1}}}
\newcommand*\wide@bar[2]{\if@single{#1}{\wide@bar@{#1}{#2}{1}}{\wide@bar@{#1}{#2}{2}}}
\newcommand*\wide@bar@[3]{%
	\begingroup
	\def\mathaccent##1##2{%
		\let\mathaccent\save@mathaccent
		\if#32 \let\macc@nucleus\first@char \fi
		\setbox\z@\hbox{$\macc@style{\macc@nucleus}_{}$}%
		\setbox\tw@\hbox{$\macc@style{\macc@nucleus}{}_{}$}%
		\dimen@\wd\tw@
		\advance\dimen@-\wd\z@
		\divide\dimen@ 3
		\@tempdima\wd\tw@
		\advance\@tempdima-\scriptspace
		\divide\@tempdima 10
		\advance\dimen@-\@tempdima
		\ifdim\dimen@>\z@ \dimen@0pt\fi
		\rel@kern{0.6}\kern-\dimen@
		\if#31
		\overline{\rel@kern{-0.6}\kern\dimen@\macc@nucleus\rel@kern{0.4}\kern\dimen@}%
		\advance\dimen@0.4\dimexpr\macc@kerna
		\let\final@kern#2%
		\ifdim\dimen@<\z@ \let\final@kern1\fi
		\if\final@kern1 \kern-\dimen@\fi
		\else
		\overline{\rel@kern{-0.6}\kern\dimen@#1}%
		\fi
	}%
	\macc@depth\@ne
	\let\math@bgroup\@empty \let\math@egroup\macc@set@skewchar
	\mathsurround\z@ \frozen@everymath{\mathgroup\macc@group\relax}%
	\macc@set@skewchar\relax
	\let\mathaccentV\macc@nested@a
	\if#31
	\macc@nested@a\relax111{#1}%
	\else
	\def\gobble@till@marker##1\endmarker{}%
	\futurelet\first@char\gobble@till@marker#1\endmarker
	\ifcat\noexpand\first@char A\else
	\def\first@char{}%
	\fi
	\macc@nested@a\relax111{\first@char}%
	\fi
	\endgroup
}
\makeatother

\newcommand{\PSh}{\mathrm{PSh}}
\newcommand{\Sh}{\mathrm{Sh}}

\newcommand{\Tuupperp}{\tensor*[^\cT]{u}{^p}}
\newcommand{\Tulowerp}{\tensor*[^\cT]{u}{_p}}
\newcommand{\Tuuppers}{\tensor*[^\cT]{u}{^s}}
\newcommand{\Tulowers}{\tensor*[^\cT]{u}{_s}}
\newcommand{\Tulowersw}{\tensor*[^\cT]{u}{_s^\wedge}}
\newcommand{\pu}{\tensor*[_p]{u}{}}
\newcommand{\su}{\tensor*[_s]{u}{}}
\newcommand{\Tpu}{\tensor*[^\cT_p]{u}{}}
\newcommand{\Tsu}{\tensor*[^\cT_s]{u}{}}
\newcommand{\Tsuw}{\tensor*[^\cT_s]{u}{^\wedge}}
\newcommand{\Dfpull}{\tensor*[^\cD]{f}{^{-1}}}
\newcommand{\Dfpush}{\tensor*[^\cD]{f}{_*}}
\newcommand{\Duuppers}{\tensor*[^\cD]{u}{^s}}

\newcommand{\Geom}{\mathrm{Geom}}
\newcommand{\LPr}{\mathcal{P}\mathrm{r}^\rL}
\newcommand{\RPr}{\mathcal{P}\mathrm{r}^\rR}
\newcommand{\CX}{\cC_{/X}}

\newcommand{\CXP}{(\cC_{/X})_{\bP}}
\newcommand{\GeomXP}{(\mathrm{Geom}_{/X})_\bP}
\newcommand{\GeomYP}{(\mathrm{Geom}_{/Y})_\bP}
\newcommand{\infcat}{$\infty$-category\xspace}
\newcommand{\infcats}{$\infty$-categories\xspace}
\newcommand{\infsite}{$\infty$-site\xspace}
\newcommand{\infsites}{$\infty$-sites\xspace}

\newcommand{\Grpd}{\mathrm{Grpd}}
\newcommand{\sSet}{\mathrm{sSet}}

\newcommand{\Ab}{\mathrm{Ab}}
\newcommand{\DAb}{\cD(\Ab)}
\newcommand{\tauan}{\tau_\mathrm{an}}

\newcommand{\tauet}{\tau_\mathrm{\acute{e}t}}
\newcommand{\tauqet}{\tau_\mathrm{q\acute{e}t}}
\newcommand{\bPsm}{\bP_\mathrm{sm}}
\newcommand{\bPqsm}{\bP_\mathrm{qsm}}
\newcommand{\Modh}{\textrm{-}\mathrm{Mod}^\heartsuit}
\newcommand{\Mod}{\textrm{-}\mathrm{Mod}}
\newcommand{\Coh}{\mathrm{Coh}}
\newcommand{\Cohb}{\mathrm{Coh}^\mathrm{b}}
\newcommand{\Cohh}{\mathrm{Coh}^\heartsuit}
\newcommand{\RcHom}{\rR\!\mathcal H\!\mathit{om}}

\newcommand{\Stn}{\mathrm{Stn}}

\newcommand{\Aff}{\mathrm{Aff}}
\newcommand{\Afflfp}{\mathrm{Aff}^{\mathrm{lfp}}}
\newcommand{\An}{\mathrm{An}}
\newcommand{\Afd}{\mathrm{Afd}}

\newcommand{\wiotaC}{\iota_\cC^\wedge}
\newcommand{\wiotaD}{\iota_\cD^\wedge}
\newcommand{\wLC}{\rL_\cC^\wedge}
\newcommand{\wLD}{\rL_\cD^\wedge}

\newcommand{\an}{^\mathrm{an}}
\newcommand{\alg}{^\mathrm{alg}}

\newcommand{\inv}{^{-1}}

\newcommand{\canal}{$\mathbb C$-analytic\xspace}

\newcommand{\kanal}{$k$-analytic\xspace}

\newcommand{\op}{^\mathrm{op}}
\newcommand{\Cech}{\check{\mathcal C}}

\providecommand{\abs}[1]{\lvert#1\rvert}

\usetikzlibrary{decorations.markings} 
\tikzset{
  closed/.style = {decoration = {markings, mark = at position 0.5 with { \node[transform shape, xscale = .8, yscale=.4] {/}; } }, postaction = {decorate} },
  open/.style = {decoration = {markings, mark = at position 0.5 with { \node[transform shape, scale = .7] {$\circ$}; } }, postaction = {decorate} }
}

\DeclareMathOperator{\cosk}{cosk}

\DeclareMathOperator{\Fun}{Fun}
\DeclareMathOperator{\FunR}{Fun^R}

\DeclareMathOperator{\Int}{Int}

\DeclareMathOperator{\Map}{Map}

\DeclareMathOperator{\Sp}{Sp}

\DeclareMathOperator{\Spec}{Spec}

\DeclareMathOperator*{\colim}{colim}

\begin{document}
\title{Higher analytic stacks and GAGA theorems}

\author{Mauro PORTA}
\address{Mauro PORTA, Institut de Math\'ematiques de Jussieu - Paris Rive Gauche, CNRS-UMR 7586, Case 7012, Universit\'e Paris Diderot - Paris 7, B\^atiment Sophie Germain 75205 Paris Cedex 13 France}
\email{mauro.porta@imj-prg.fr}

\author{Tony Yue YU}
\address{Tony Yue YU, Institut de Math\'ematiques de Jussieu - Paris Rive Gauche, CNRS-UMR 7586, Case 7012, Universit\'e Paris Diderot - Paris 7, B\^atiment Sophie Germain 75205 Paris Cedex 13 France}
\email{yuyuetony@gmail.com}

\date{December 16, 2014 (Revised on July 6, 2016)}

\begin{abstract}
We develop the foundations of higher geometric stacks in complex analytic geometry and in non-archimedean analytic geometry.
We study coherent sheaves and prove the analog of Grauert's theorem for derived direct images under proper morphisms.
We define analytification functors and prove the analog of Serre's GAGA theorems for higher stacks.
We use the language of infinity category to simplify the theory.
In particular, it enables us to circumvent the functoriality problem of the lisse-étale sites for sheaves on stacks.
Our constructions and theorems cover the classical 1-stacks as a special case.
\end{abstract}

\maketitle

\tableofcontents

\section{Introduction}\label{sec:intro_stacks}

A moduli space is a space which classifies certain objects.
Due to non-trivial automorphisms of the objects, a moduli space often carries the structure of a stack.
Moreover, a moduli space often has a geometric structure.
For example, the theory of algebraic stacks was developed in order to study moduli spaces in algebraic geometry \cite{Deligne_Irreducibility_1969,Artin_Versal_1974}.
Likewise, for moduli spaces in analytic geometry, we need the theory of analytic stacks.

The definition of algebraic stacks carries over easily to analytic geometry, and so do many constructions and theorems.
However, there are aspects which do not have immediate translations.
In this paper, we begin by the study of proper morphisms, coherent sheaves and their derived direct images in the setting of analytic stacks.
They are treated in a different way from their algebraic counterparts.
After that, we study the analytification of algebraic stacks and compare algebraic coherent sheaves with analytic coherent sheaves by proving analogs of Serre's GAGA theorems \cite{Serre_GAGA}.

Before stating the theorems, let us make precise what we mean by ``analytic'' and what we mean by ``stacks''.

By ``analytic geometry'', we mean both complex analytic geometry and non-archimedean analytic geometry.
We use the theory of Berkovich spaces \cite{Berkovich_Spectral_1990,Berkovich_Etale_1993} for non-archimedean analytic geometry.
We will write \emph{\canal} to mean complex analytic; we will write \emph{\kanal} to mean non-archimedean analytic for a non-archimedean ground field $k$.
We will simply say \emph{analytic} when the statements apply to both \canal and \kanal situations simultaneously.

By ``stacks'', we mean higher stacks in the sense of Simpson \cite{Simpson_Algebraic_1996}.
It is a vast generalization of the classical notion of stacks usually defined as categories fibered in groupoids satisfying certain conditions \cite{Deligne_Irreducibility_1969,Artin_Versal_1974,Laumon_Champs_2000,stacks-project}.
We choose to work with higher stacks not only because it is more general, but also because it makes many constructions and proofs clearer both technically and conceptually.

Our approach is based on the theory of \infcat (cf.\ Joyal \cite{Joyal_Quasi-categories_2002} and Lurie \cite{HTT}).
It has many advantages.
In particular, it enables us to circumvent the classical functoriality problem of lisse-étale sites in a natural way (see Olsson \cite{Olsson_Sheaves_2007} for a description of the problem and a different solution).

We begin in \cref{sec:higher_geometric_stacks} by introducing the notion of higher geometric stacks for a general geometric context.
They are specialized to algebraic stacks, complex analytic stacks and non-archimedean analytic stacks in \cref{sec:examples}.

In \cref{sec:proper_morphisms}, we introduce the notion of weakly proper pairs of analytic stacks, which is then used to define proper morphisms of analytic stacks.
In \cref{sec:direct_images}, we define coherent sheaves and their derived direct images using an analog of the classical lisse-étale site.

For a higher analytic stack (or a higher algebraic stack) $X$, we denote by $\Coh(X)$ (resp.\ $\Coh^+(X)$) the unbounded (resp.\ bounded below) derived \infcat of coherent sheaves on $X$ (cf.\ \cref{def:coherent_sheaf_on_stacks}).

The following theorem is the analog of Grauert's direct image theorem \cite{Grauert_Theorem_der_analytischen_1960} for higher stacks.

\begin{thm}[Theorems \ref{thm:proper_direct_image_alg}, \ref{thm:proper_direct_image_anal}]
Let $f\colon X\to Y$ be a proper morphism of higher analytic stacks (or locally noetherian higher algebraic stacks).
Then the derived pushforward functor $\rR f_*$ sends objects in $\Coh^+(X)$ to $\Coh^+(Y)$.
\end{thm}

Algebraic stacks and analytic stacks are related via the analytification functor in \cref{sec:analytification}.
We prove the analogs of Serre's GAGA theorems in \cref{sec:GAGA}.

Let $A$ be either the field of complex numbers or a $k$-affinoid algebra.

\begin{thm}[GAGA-1, \cref{thm:GAGA1}]
Let $f\colon X\to Y$ be a proper morphism of higher algebraic stacks locally finitely presented over $\Spec A$.
The canonical comparison morphism
\[ (\rR f_*\cF)\an\longrightarrow \rR f\an_*\cF\an\]
in $\Coh^+(Y\an)$ is an equivalence for all $\cF \in\Coh^+(X)$.
\end{thm}

\begin{thm}[GAGA-2, \cref{cor:unbounded_GAGA2}]
Let $X$ be a higher algebraic stack proper over $\Spec A$.
The analytification functor on coherent sheaves induces an equivalence of categories
\[\Coh(X)\xrightarrow{\ \sim\ } \Coh(X\an).\]
\end{thm}

\begin{rem}
The two theorems above are stated for the absolute case in complex geometry, while for the relative case in non-archimedean geometry.
The proof for the relative case in complex geometry would be more involved because Stein algebras are not noetherian in general.
\end{rem}

\medskip

\paragraph{\textbf{Related works}}

In the classical sense, complex analytic stacks were considered in \cite{Behrend_Uniformization_2006}, and non-archimedean analytic stacks were considered in \cite{Yu_Gromov_2014,Ulirsch_Geometric_2014} to the best of our knowledge. 

The general theory of higher stacks was studied extensively by Simpson \cite{Simpson_Algebraic_1996}, Lurie \cite{DAG-V} and Toën-Vezzosi \cite{HAG-I,HAG-II}.
Our \cref{sec:higher_geometric_stacks} follows mainly \cite{HAG-II}.
However, we do not borrow directly the HAG context of \cite{HAG-II}, because the latter is based on symmetric monoidal model categories which is not suitable for analytic geometry.

Our definition of properness for analytic stacks follows an idea of Kiehl in rigid analytic geometry \cite{Kiehl_Endlichkeitssatz_1967}.
The coherence of derived direct images under proper morphisms (i.e.\ Grauert's theorem) was proved in \cite{Grauert_Theorem_der_analytischen_1960,Kiehl_Verdier_1971,Forster_Knorr_1971,Houzel_Espaces_analytiques_relatifs_1973,Levy_Grauert_1987}  for complex analytic spaces and in \cite{Kiehl_Endlichkeitssatz_1967} for rigid analytic spaces.

Analytification of algebraic spaces and classical algebraic stacks was studied in \cite{Artin_Algebraization_II_1970,Lurie_Tannaka_duality,Toen_Algebrisation_2008,Conrad_Non-archimedean_analytification_2009}.
Analogs and generalizations of Serre's GAGA theorems are found in \cite{SGA1,Kopf_Eigentliche_1974,Lutkebohmert_Formal_1990,Berkovich_Spectral_1990,Conrad_Relative_2006,Conrad_Formal_GAGA,Hall_GAGA_2011}.
Our proofs use induction on the geometric level of higher stacks. We are very much inspired by the strategies of Brian Conrad in \cite{Conrad_Formal_GAGA}.

In \cite{Porta_DCAGI}, Mauro Porta deduced from this paper a GAGA theorem for derived complex analytic stacks.
Several results in this paper are also used in the work \cite{Porta_Yu_DNAnG_I}.

In \cite{Yu_Enumeration_cylinders_2015}, Tony Yue Yu applied the GAGA theorem for non-archimedean analytic stacks to the enumerative geometry of log Calabi-Yau surfaces.

\medskip
\paragraph{\textbf{Acknowledgements}}
We are grateful to Antoine Chambert-Loir, Antoine Ducros, Maxim Kontsevich, Yves Laszlo, Valerio Melani, François Petit, Marco Robalo, Matthieu Romagny, Pierre Schapira, Michael Temkin and Gabriele Vezzosi for very useful discussions.
The authors would also like to thank each other for the joint effort.

\section{Higher geometric stacks}\label{sec:higher_geometric_stacks}

\subsection{Notations}

We refer to Lurie \cite{HTT,Lurie_Higher_algebra} for the theory of \infcat.
The symbol $\mathcal S$ denotes the $\infty$-category of spaces.

\begin{defin}[cf.\ {\cite[00VH]{stacks-project}}]
An \emph{\infsite} $(\cC,\tau)$ consists of a small \infcat $\cC$ and a set $\tau$ of families of morphisms with fixed target $\{U_i\to U\}_{i\in I}$, called \emph{$\tau$-coverings} of $\cC$, satisfying the following axioms:
\begin{enumerate}[(i)]
\item If $V\to U$ is an equivalence then $\{V\to U\}\in\tau$.
\item If $\{U_i\to U\}_{i\in I}\in\tau$ and for each $i$ we have $\{V_{ij}\to U_i\}_{j\in J_i}\in\tau$, then $\{V_{ij}\to U\}_{i\in I,j\in J_i}\in\tau$.
\item If $\{U_i\to U\}_{i\in I}\in\tau$ and $V\to U$ is a morphism of $\cC$, then $U_i\times_U V$ exists for all $i$ and $\{U_i\times_U V\to V\}_{i\in I}\in\tau$.
\end{enumerate}
\end{defin}

Let $(\cC,\tau)$ be an \infsite.
Let $\cT$ be a presentable \infcat.
Let $\PSh_\cT(\cC)$ denote the \infcat of $\cT$-valued presheaves on $\cC$.
Let $\Sh_\cT(\cC,\tau)$ denote the \infcat of $\cT$-valued sheaves on the \infsite $(\cC,\tau)$ (cf.\ \cite[\S 1.1]{DAG-V}).
We will refer to $\cS$-valued presheaves (resp.\ sheaves) simply as presheaves (resp.\ sheaves).
We denote $\PSh(\cC) \coloneqq \PSh_\cS(\cC)$, $\Sh(\cC,\tau) \coloneqq \Sh_\cS(\cC,\tau)$.
We denote the Yoneda embedding by
\[ h\colon\cC\to\PSh(\cC),\qquad X\mapsto h_X.\]
We denote by $\Sh_\cT(\cC,\tau)^\wedge$ the full subcategory of $\Sh_\cT(\cC,\tau)$ spanned by hypercomplete $\cT$-valued sheaves (cf.\ \cite[\S 6.5]{HTT} and \cite[\S 5]{DAG-VII} for the notion of hypercompleteness and hypercoverings).

Let $\iota_\cC\colon\Sh_\cT(\cC,\tau)\to\PSh_\cT(\cC)$ and $\wiotaC \colon \Sh_\cT(\cC,\tau)^\wedge \to \PSh_\cT(\cC)$ denote the inclusion functors.
Let $\rL_\cC\colon \PSh_\cT(\cC)\to\Sh_\cT(\cC,\tau)$ denote the sheafification functor.
Let $\wLC \colon \PSh_\cT(\cC) \to \Sh(\cC, \tau)^\wedge$ denote the composition of sheafification and hypercompletion.

\subsection{Geometric contexts}\label{sec:general_context}

In this section, we introduce the notion of geometric context.
It can be regarded as the minimum requirement to work with geometric stacks in various situations (compare Toën-Vezzosi \cite[\S 1.3.2]{HAG-II}, Toën-Vaquié \cite[\S 2.2]{Toen_Algebrisation_2008}).

\begin{defin}\label{def:context}
A \emph{geometric context} $(\cC,\tau,\bP)$ consists of an \infsite $(\cC,\tau)$ and a class $\bP$ of morphisms in $\cC$ satisfying the following conditions:
\begin{enumerate}[(i)]
\item \label{context:subcanonical} Every representable presheaf is a hypercomplete sheaf.
\item \label{context:P} The class $\bP$ is closed under equivalence, composition and pullback.
\item \label{context:tau_in_P} Every $\tau$-covering consists of morphisms in $\bP$.
\item \label{context:P_tau_local} For any morphism $f\colon X\to Y$ in $\cC$, if there exists a $\tau$-covering $\{U_i\to X\}$ such that each composite morphism $U_i\to Y$ belongs to $\bP$, then $f$ belongs to $\bP$.
\end{enumerate}
\end{defin}

\begin{rem}
In all the examples that we will consider in this paper, the site $(\cC,\tau)$ is a classical Grothendieck site, i.e.\ the category $\cC$ is a 1-category.
We state \cref{def:context} for general \infsites because it will serve as a geometric context for derived stacks in our subsequent works (cf.\ \cite{Porta_DCAGI,Porta_Yu_DNAnG_I}).
\end{rem}

\begin{prop} \label{prop:descent_vs_hyperdescent}
	Let $(\cC, \tau)$ be an $\infty$-site and let $\cD$ be an $(n+1,1)$-category for $n\ge 0$ (cf.\ \cite[\S 2.3.4]{HTT}).
	Then a functor $F \colon \cC^{\mathrm{op}} \to \cD$ satisfies descent for coverings if and only if it satisfies descent for hypercoverings.
\end{prop}
\begin{proof}
	Let $D \in \cD$ be any object and let $c_D \colon \cD \to \cS$ be the functor $\Map_\cC(D,-)$ corepresented by $D$.
	Then $F$ satisfies descent for coverings (resp.\ hypercoverings) if and only if $c_D \circ F$ does for every $D$.
	Since $\cD$ is an $(n+1,1)$-category, we see that $c_D \circ F$ takes values in $\tau_{\le n} \cS$.
	Therefore, we may replace $\cD$ with $\cS$ and assume that $F$ takes values in $\tau_{\le n}\cS$.
	For every $U \in \cC$, we have
	\[ \Map_{\Sh_{\tau_{\le n}\cS}(\cC, \tau)}(\tau_{\le n} \rL_\cC(h_U), F) \simeq \Map_{\Sh(\cC, \tau)}(\rL_{\cC}(h_U), F) \simeq \Map_{\PSh(\cC)}(h_U,F) \simeq F(U) . \]
	Therefore, it suffices to show that for every hypercovering $U^\bullet \to U$ in $\cC$, the augmented simplicial diagram
	\[ \tau_{\le n} \rL_\cC(h_{U^\bullet}) \to \tau_{\le n} \rL_{\cC}(h_U) \]
	is a colimit diagram in $\Sh_{\tau_{\le n}\cS}(\cC, \tau)$.
	Since $\tau_{\le n}$ is a left adjoint, we see that in $\Sh_{\tau_{\le n}\cS}(\cC, \tau)$ the relation
	\[ |\tau_{\le n} \rL_{\cC}(h_{U^\bullet})| \simeq \tau_{\le n} | \rL_{\cC} (h_{U^\bullet})| \]
	holds, where $\abs{\cdot}$ denotes the geometric realization of a simplicial object.
	Moreover, since $U^\bullet \to U$ is a hypercovering, the morphism $|\rL_{\cC}(h_{U^\bullet})| \to \rL_{\cC}(h_U)$ is $\infty$-connected in virtue of \cite[6.5.3.11]{HTT}.
	Since $\tau_{\le n}$ commutes with $\infty$-connected morphisms,  we conclude that
	\[ \tau_{\le n} |\rL_{\cC}(h_{U^\bullet})| \to \tau_{\le n} \rL_{\cC}(h_U) \]
	is an $\infty$-connected morphism between $n$-truncated objects.
	Therefore it is an equivalence in $\Sh(\cC, \tau)$.
	In conclusion, the morphism $|\tau_{\le n} \rL_{\cC}(h_{U^\bullet})| \to \tau_{\le n} \rL_{\cC}(h_U)$ is an equivalence in $\Sh_{\tau_{\le n}\cS}(\cC, \tau)$, completing the proof.
\end{proof}

\begin{cor}\label{cor:hyperdescent}
Let $(\cC,\tau)$ be a classical Grothendieck site.
Then a representable presheaf is a sheaf if and only if it is a hypercomplete sheaf.
\end{cor}
\begin{proof}
Since $\cC$ is a 1-category, for every object $X\in\cC$, the representable presheaf $h_X$ takes values in the category of sets.
Therefore, by \cref{prop:descent_vs_hyperdescent}, the representable presheaf $h_X$ is a sheaf if and only it is a hypercomplete sheaf.
\end{proof}

\subsection{Higher geometric stacks}\label{sec:definition_geometric_stacks}

We fix a geometric context $(\cC,\tau,\bP)$ for this section.

\begin{defin}\label{def:stack}
A \emph{stack} is a hypercomplete sheaf on the site $(\cC,\tau)$.
\end{defin}

\begin{rem}\label{rem:comparison}
Let us explain briefly how \cref{def:stack} is related to the classical notion of stacks in terms of categories fibered in groupoids as in \cite{Deligne_Irreducibility_1969,Artin_Versal_1974,Laumon_Champs_2000,stacks-project}.
Firstly, a stack in terms of a category fibered in groupoids over a site is equivalent to a stack in terms of a sheaf of groupoids on the site (cf.\ \cite{Vistoli_Grothendieck_2005}).
Secondly, we have an adjunction
\[\Pi_1\colon\cS\to\Grpd \qquad \mathrm{N}\colon\Grpd\to \cS,\]
where $\Grpd$ denotes the \infcat of groupoids.
The nerve functor $\mathrm{N}$ induces an embedding of sheaves of groupoids on $\cC$ into sheaves of spaces on $\cC$, whose image consists of all 1-truncated objects.
In other words, classical stacks in groupoids are 1-truncated stacks in \cref{def:stack}.
We remark that the sheaf condition for a stack, usually presented as two separated descent conditions (see \cite[Definition 4.1(ii)(iii)]{Deligne_Irreducibility_1969}), 
is now combined into one single descent condition via $\infty$-categorical limits (i.e.\ homotopy limits) instead of 1-categorical limits.
We refer to \cite{Hollander_Homotopy_2008} and \cite[§2.1.2]{HAG-II} for more detailed discussions on the comparison.
\end{rem}

We introduce the notion of $n$-geometric stacks following \cite[\S 1.3.3]{HAG-II} (see also \cite{Simpson_Algebraic_1996}).

\begin{defin}\label{def:geometric_stack}
We define the following notions by induction on $n$.
We call $n$ the geometric level.
Base step:
\begin{enumerate}[(i)]
\item A stack is said to be \emph{$(-1)$-geometric} if it is representable.
\item A morphism of stacks $F\to G$ is said to be \emph{$(-1)$-representable} if for any representable stack $X$ and any morphism $X\to G$, the pullback $F\times_G X$ is representable.
\item A morphism of stacks $F\to G$ is said to be in $(-1)$-$\bP$ if it is $(-1)$-representable and if for any representable stack $X$ and any morphism $X\to G$, the morphism $F\times_G X\to X$ is in $\bP$.
\end{enumerate}
Now let $n\geq 0$:
\begin{enumerate}[(i)]
\item An \emph{$n$-atlas}  of a stack $F$ is a family of representable stacks $\{U_i\}_{i\in I}$ equipped with $(n\!-\!1)$-$\bP$-morphisms $U_i\to F$ such that the total morphism $\coprod_{i\in I} h_{U_i}\to F$ is an effective epimorphism of sheaves (cf.\ \cite[§6.2.3]{HTT} for the notion of effective epimorphism).
\item A stack $F$ is said to be \emph{$n$-geometric} if
\begin{itemize}
\item the diagonal morphism $F\to F\times F$ is $(n\!-\!1)$-representable, and
\item the stack $F$ admits an $n$-atlas.
\end{itemize}
\item A morphism of stacks $F\to G$ is said to be \emph{$n$-representable} if for any representable stack $X$ and any morphism $X\to G$, the pullback $F\times_G X$ is $n$-geometric.
\item \label{item:n-P} A morphism of stacks $F\to G$ is said to be in $n$-$\bP$ if it is $n$-representable and if for any representable stack $X$ and any morphism $X\to G$, there exists an $n$-atlas $\{U_i\}$ of $F\times_G X$ such that each composite morphism $U_i\to X$ is in $\bP$.
\end{enumerate}
A stack $F$ is said to be \emph{geometric} if it is $n$-geometric for some $n$.
\end{defin}

The following proposition summarizes the basic properties.
We refer to \cite[Proposition 1.3.3.3]{HAG-I} for the proof.

\begin{prop}\label{prop:basic_properties}
\begin{enumerate}[(i)]
\item A stack $F$ is $n$-geometric if and only if the morphism $F\to *$ is $n$-representable.
\item An $(n\!-\!1)$-representable morphism is $n$-representable.
\item An $(n\!-\!1)$-$\bP$ morphism is an $n$-$\bP$ morphism.
\item \label{item:geometric_stacks_basic_stabilities} The class of $n$-representable morphisms is closed under equivalences, compositions and pullbacks.
\item The class of $n$-$\bP$ morphisms is closed under equivalences, compositions and pullbacks.
\end{enumerate}
\end{prop}

\begin{rem} \label{rem:property_of_morphisms}
The collection of 0-atlases over representable stacks generates another Grothendieck topology on the category $\cC$, which we call the topology $\bP$.

Let $\bQ$ be a property of morphisms in $\cC$ which is stable under equivalence, composition and pullback, and which is local on the target with respect to the topology $\bP$.
We can use it to define a corresponding property for $(-1)$-representable morphisms of stacks.
Namely, a $(-1)$-representable morphism $f\colon X\to Y$ of stacks is said to have property $\bQ$ if there exists an atlas $\{U_i\}_{i\in I}$ of $Y$ (or equivalently for any atlas $\{U_i\}_{i\in I}$ of $Y$), such that the pullback morphisms $X\times_Y U_i\to U_i$ have property $\bQ$.
In this way, we will be able to speak of closed immersions, open immersions, dense open immersions, etc.\ later for algebraic stacks and analytic stacks.

If the property $\bQ$ is moreover local on the source with respect to the topology $\bP$, then we can use it to define a corresponding property for morphisms of geometric stacks.
Namely, a morphism $f\colon X\to Y$ of geometric stacks is said to have property $\bQ$ if there exists an atlas $\{U_i\}_{i\in I}$ of $Y$ (or equivalently for any atlas $\{U_i\}_{i\in I}$ of $Y$), and for every $i\in I$ there exists an atlas $\{V_{ij}\}_{j\in J_i}$ of $X\times_Y U_i$ (or equivalently for any atlas $\{V_{ij}\}_{j\in J_i}$ of $X\times_Y U_i$), such that the composite morphisms $V_{ij}\to X\times_Y U_i\to U_i$ have property $\bQ$.
\end{rem}

\subsection{Functorialities of the category of sheaves} \label{sec:functoriality}

In this section, we discuss the functorialities of the \infcat of sheaves.
We refer to \cite{SGA4,stacks-project} for the classical 1-categorical case.

Fix a presentable \infcat $\cT$ in which our (pre)sheaves will take values.

Let $u\colon\cC\to\cD$ be a functor between two small \infcats.
It induces a functor
\begin{align*}
\Tuupperp\colon\PSh_\cT(\cD)&\longrightarrow\PSh_\cT(\cC)\\
F&\longmapsto F\circ u.
\end{align*}
We denote by $\Tulowerp$ the left Kan extension along $u$,
and by $\Tpu$ the right Kan extension along $u$.
By definition, $\Tulowerp$ is left adjoint to $\Tuupperp$ and $\Tpu$ is right adjoint to $\Tuupperp$.
Their existences follow from \cite[\S 5]{HTT}.
When $\cT$ equals the \infcat of spaces $\cS$, we will omit the left superscripts~$\tensor*[^\cT]{}{}$.

\begin{lem}\label{lem:u_p_preserves_representables}
For any object $U\in\cC$, we have $u_p h_U \simeq h_{u(U)}$.
\end{lem}
\begin{proof}
Since $u_p$ is left adjoint to $u^p$, we have
\[\Map_{\PSh(\cD)}(u_p h_U, G) \simeq \Map_{\PSh(\cC)}(h_U,u^p G) \simeq u^p G(U) \simeq G(u(U))\]
for any $G\in\PSh(\cD)$.
Hence $u_p h_U \simeq h_{u(U)}$ by the Yoneda lemma.
\end{proof}

Let $(\cC,\tau)$, $(\cD,\sigma)$ be two \infsites.
For a functor $u\colon\cC\to\cD$, we define
\begin{align*}
\Tulowers &\coloneqq \rL_\cD\circ \Tulowerp\circ\iota_\cC\colon\Sh_\cT(\cC)\to\Sh_\cT(\cD),\\
\Tulowersw &\coloneqq \wLD \circ \Tulowerp \circ \wiotaC \colon \Sh_\cT(\cC)^\wedge \to \Sh_\cT(\cD)^\wedge , \\
\Tuuppers &\coloneqq \rL_\cC\circ \Tuupperp \circ\iota_\cD\colon \Sh_\cT(\cD)\to\Sh_\cT(\cC),\\
\big(\Tuuppers\big)^\wedge & \coloneqq \wLC \circ \Tuupperp \circ \wiotaD \colon \Sh_\cT(\cD)^\wedge \to \Sh_\cT(\cC)^\wedge , \\
\Tsu &\coloneqq \rL_\cD\circ \Tpu \circ\iota_\cC\colon\Sh_\cT(\cC)\to\Sh_\cT(\cD),\\
\Tsuw & \coloneqq \wLD \circ \Tpu \circ \wiotaC \colon \Sh_\cT(\cC)^\wedge \to \Sh_\cT(\cD)^\wedge.
\end{align*}
When $\cT$ equals the \infcat of spaces $\cS$, we will omit the left superscripts~$\tensor*[^\cT]{}{}$.

\begin{defin}[cf.\ {\cite[Tag 00WV]{stacks-project}}]\label{def:continuous}
A functor $u\colon\cC\to\cD$ is called \emph{continuous} if for every $\tau$-covering $\{U_i\to U\}$ we have
\begin{enumerate}[(i)]
\item $\{u(U_i)\to u(U)\}$ is a $\sigma$-covering, and
\item for any morphism $S\to U$ in $\cC$ the morphism $u(S\times_U U_i)\to u(S)\times_{u(U)} u(U_i)$ is an equivalence.
\end{enumerate}
\end{defin}

\begin{lem} \label{lem:u^p_preserves_sheaves}
Let $u\colon\cC\to\cD$ be a continuous functor.
The functor $\Tuupperp$ sends $\cT$-valued sheaves (resp.\ hypercomplete sheaves) on the site $(\cD,\sigma)$ to $\cT$-valued sheaves (resp.\ hypercomplete sheaves) on the site $(\cC,\tau)$.
In particular, the functor $\Tuuppers$ (resp.\ $\big(\Tuuppers\big)^\wedge$) equals the restriction of the functor $\Tuupperp$ to $\Sh_\cT(\cD,\sigma)$ (resp.\ to $\Sh_\cT(\cD,\sigma)^\wedge$).
\end{lem}
\begin{proof}
	For every $\tau$-covering $\{U_i \to X\}$ consider the total morphism $f \colon \coprod h_{U_i} \to h_X$.
	Since $u$ is continuous, \cref{lem:u_p_preserves_representables} shows that $u_p$ commutes with the \v{C}ech nerve of $f$, in the sense that the natural morphism
	\[ \Cech(u_p(f)) \to u_p \circ \Cech(f) \]
	is an equivalence of simplicial objects.
	Moreover, \cref{lem:u_p_preserves_representables} also shows that $u_p$ takes $\tau$-hypercoverings to $\sigma$-hypercoverings.
	
	Let $F \in \Sh_{\cT}(\cD, \sigma)$ (resp.\ $\Sh_{\cT}(\cD,\sigma)^\wedge$).
	The presheaf $\Tuupperp(F) \in \PSh_{\cT}(\cC)$ is a sheaf (resp.\ hypercomplete sheaf) for the topology $\tau$ if and only if for every $\tau$-covering $\{U_i \to X\}$ (resp.\ for every $\tau$-hypercovering $U^\bullet \to X$) the morphism
	\[ \Tuupperp(F)(X) \to \lim_{\mathbf \Delta} \Tuupperp(F) \circ \Cech(f) \qquad \text{(resp.\ } \Tuupperp(F)(X) \to \lim_{\mathbf \Delta} \Tuupperp(F)(U^\bullet) \text{)} \]
	is an equivalence in $\cT$.
	By the definition of $\Tuupperp$ and the previous paragraph, we see that this is equivalent to the condition that
	\[ F(u(X)) \to \lim_{\mathbf \Delta} F \circ \Cech(u_p(f)) \qquad \text{(resp.\ } F(u(X)) \to \lim_{\mathbf \Delta} F(u_p(U^\bullet)) \text{)}  \]
	is an equivalence.
	Since $u$ is continuous, we see that $u_p(f)$ is the total morphism of a $\sigma$-covering (resp.\ that $u_p(U^\bullet) \to u(X)$ is a $\sigma$-hypercovering).
	Since $F$ belongs to $\Sh_{\cT}(\cD, \sigma)$ (resp.\ to $\Sh_\cT(\cD, \sigma)^\wedge$) by assumption, the condition above holds.
\end{proof}

\begin{lem} \label{lem:adjunction_u_s_u^s}
Let $u\colon\cC\to\cD$ be a continuous functor.
The functor $\Tulowers$ (resp.\ $\Tulowersw$) is left adjoint to the functor $\Tuuppers$ (resp.\ $(\Tuuppers)^\wedge$).
Assume moreover that both $\cC$ and $\cD$ admit finite limits and that $u \colon \cC \to \cD$ preserves them.
Then $\Tulowers$ and $\Tulowersw$ are left exact.
In particular, the pairs $(u_s, u^s)$ and $(u_s^\wedge, (u^s)^\wedge)$ are geometric morphisms of $\infty$-topoi.
\end{lem}

\begin{proof}
Since $\Tulowerp$ is left adjoint to $\Tuupperp$ and $\rL_\cD$ is left adjoint to $\iota_\cD$,
the composition $\rL_\cD\circ \Tulowerp$ is left adjoint to $\Tuupperp\circ\iota_\cD$.
By \cref{lem:u^p_preserves_sheaves}, restricting to the categories of $\cT$-valued sheaves, we obtain that $\Tulowers$ is left adjoint to $\Tuuppers$.
This shows the first statement in the non-hypercomplete case.
The same proof works in the hypercomplete case.
The second statement is a consequence of \cite[6.2.2.7, 6.1.5.2]{HTT} in the non-hypercomplete case and of \cite[6.5.1.16, 6.2.1.1]{HTT} in the hypercomplete case.
\end{proof}

\begin{lem} \label{lem:u_s_applied_to_presheaf}
	Let $u\colon\cC\to\cD$ be a continuous functor.
	For any $\cT$-valued presheaf $F$ on $\cC$, we have $\rL_\cD \Tulowerp F \simeq \Tulowers \rL_\cC F$ and $\wLD \Tulowerp F \simeq \Tulowersw \wLC F$.
\end{lem}
\begin{proof}
For any $\cT$-valued sheaf $G$ on $(\cD,\sigma)$, we have
\begin{align*}
\Map_{\Sh_\cT(\cD,\sigma)}(\Tulowers \rL_\cC F,G) & \simeq \Map_{\Sh_\cT(\cC,\tau)} (\rL_\cC F, \Tuuppers G)\\
& \simeq \Map_{\PSh_\cT(\cC)}(F, \Tuupperp\iota_\cD G)\\
& \simeq \Map_{\PSh_\cT(\cD)}(\Tulowerp F, \iota_\cD G)\\
& \simeq \Map_{\Sh_\cT(\cD,\sigma)}(\rL_\cD \Tulowerp F, G).
\end{align*}
So the statement in the non-hypercomplete case follows from the Yoneda lemma.
The same proof works in the hypercomplete case.
\end{proof}

\begin{lem} \label{lem:u_s_preserves_representables}
	Let $u\colon\cC\to\cD$ be a continuous functor.
	For any object $U\in \cC$, we have $u_s \rL_\cC h_U \simeq \rL_\cD h_{u(U)}$ and $u_s^\wedge \wLC h_U \simeq \wLD h_{u(U)}$.
\end{lem}
\begin{proof}
It follows from \cref{lem:u_p_preserves_representables} and \cref{lem:u_s_applied_to_presheaf}.
\end{proof}

\begin{defin}[cf.\ {\cite[Tag 00XJ]{stacks-project}}]\label{def:cocontinuous}
A functor $u\colon\cC\to\cD$ is called \emph{cocontinuous} if for every $U\in \cC$ and every $\sigma$-covering $\{V_j\to u(U)\}_{j\in J}$, there exists a $\tau$-covering $\{U_i\to U\}_{i\in I}$ such that the family of maps $\{u(U_i)\to u(U)\}_{i\in I}$ refines the covering $\{V_j\to u(U)\}_{j\in J}$.
\end{defin}

\begin{lem} \label{lem:pu_preserves_sheaves}
Let $u\colon\cC\to\cD$ be a cocontinuous functor.
The functor $\pu\colon\PSh(\cC)\allowbreak\to\PSh(\cD)$ sends sheaves (resp.\ hypercomplete sheaves) on the site $(\cC,\tau)$ to sheaves (resp.\ hypercomplete sheaves) on the site $(\cD,\sigma)$.
In particular, the functor $\su$ (resp.\ $\su^\wedge$) equals the restriction of the functor $\pu$ to $\Sh(\cC,\tau)$ (resp.\ to $\Sh(\cC, \tau)^\wedge$).
\end{lem}
\begin{proof}
Let $F$ be a sheaf (resp.\ a hypercomplete sheaf) on $(\cC,\tau)$.
Let us prove that $\pu(F)$ is a sheaf (resp.\ a hypercomplete sheaf) on $(\cD,\sigma)$.
Let $\{V_i \to X\}$ be a covering in $(\cD, \sigma)$.
Consider the total morphism
\[ f \colon V \coloneqq \coprod h_{V_i} \to h_X \]
and form the \v{C}ech nerve $V^\bullet \coloneqq \Cech(p)$ in $\PSh(\cD)$, (respectively, let $V^\bullet \to h_X$ be a $\sigma$-hypercovering).
We have to show that the canonical map
\[ \Map(h_X, \pu(F)) \to \lim \Map(V^\bullet, \pu(F)) \]
is an equivalence.
By adjunction, this is equivalent to the fact that
\[ \Map(u^p(h_X), F) \to \lim \Map(u^p(V^\bullet), F) \]
is an equivalence.
For the latter, it suffices to show that the canonical map $u^p(h_V) \to \colim u^p(h_{V^\bullet})$ becomes an equivalence after sheafification.

We will first deal with the case where $V^\bullet = \Cech(p)$.
Since $u^p$ is both a left and a right adjoint, we can identify $u^p(V^\bullet)$ with the \v{C}ech nerve of the morphism
\[ u^p(f) \colon u^p(V) = \coprod u^p(h_{V_i}) \to u^p(h_X). \]
Then it suffices to show that this morphism is an effective epimorphism.
Given any object $U \in \cC$, using the adjunction $(u_p, u^p)$,
a morphism
\[ \alpha \colon h_U \to u^p(h_X) \]
is uniquely determined by a morphism
\[ \alpha' \colon u_p(h_U) \simeq h_{u(U)} \to h_X. \]
Since $u$ is cocontinuous, we can find a $\tau$-covering $\{U_i \to U\}$ such that each composite morphism $h_{u(U_i)} \to h_{u(U)} \to h_X$ factors through $V \to X$.
By adjunction, we see that each composite morphism $h_{U_i} \to h_{U} \to u^p(h_X)$ factors as
\[ \begin{tikzcd}
	{} & {} & u^p(V) \arrow{d} \\
	h_{U_i} \arrow{r} \arrow[dashed]{urr} & h_U \arrow{r} & u^p(h_X) .
\end{tikzcd} \]
Therefore $\pi_0(u^p(V)) \to \pi_0(u^p(h_X))$ is an effective epimorphism of sheaves of sets.
We conclude by \cite[7.2.1.14]{HTT} that $u^p(V) \to u^p(h_X)$ is an effective epimorphism.

We now turn to the case of a general hypercovering $V^\bullet$ of $X$.
By \cite[6.5.3.12]{HTT}, it suffices to show that
\[u^p(h_{V^n}) \to \big(\cosk_{n-1} \big(u^p(h_{V^\bullet})/u^p(h_V)\big)\big)^n\]
is an effective epimorphism.

Given any object in $U\in\cC$ and any morphism
\[\alpha\colon h_U\to\big(\cosk_{n-1} \big(u^p(h_{V^\bullet})/u^p(h_V)\big)\big)^n,\]
by adjunction, we obtain a morphism
\[\alpha'\colon u(U)\to(\cosk_{n-1} (h_{V^\bullet}/h_V))^n.\]
Since the morphism $V^n\to(\cosk_{n-1} (h_{V^\bullet}/h_V))^n$ is a covering in $(\cD,\sigma)$, the pullback $u(U)\times_{(\cosk_{n-1} (h_{V^\bullet}/h_V))^n} V^n \to u(U)$ is also a covering in $(\cD,\sigma)$.
Since the functor $u$ is cocontinuous, there exists a $\tau$-covering $\{U_i\to U\}_{i\in I}$ such that the family of maps $\{u(U_i)\to u(U)\}_{i\in I}$ refines the covering $u(U)\times_{(\cosk_{n-1} (h_{V^\bullet}/h_V))^n} V^n \to u(U)$.
By construction, every morphism $h_{U_i}\to \big(\cosk_{n-1} \big(u^p(h_{V^\bullet})/u^p(h_V)\big)\big)^n$ factors through $u^p(h_{V^n})$,
\[
\begin{tikzcd}
	{} & {} & u^p(h_{V^n}) \arrow{d}{\alpha} \\
	h_{U_i} \arrow{urr} \arrow{r} & h_U \arrow{r} & \big(\cosk_{n-1} \big(u^p(h_{V^\bullet})/u^p(h_V)\big)\big)^n.
	\end{tikzcd}
\]
Therefore
\[u^p(h_{V^n}) \to \big(\cosk_{n-1} \big(u^p(h_{V^\bullet})/u^p(h_V)\big)\big)^n\]
is an effective epimorphism.
\end{proof}

\begin{lem} \label{lem:adjunction_u^s_su}
Let $u\colon\cC\to\cD$ be a cocontinuous functor.
The functor $u^s$ (resp.\ $(\big(u^s\big)^\wedge$) is left adjoint to the functor $\su$ (resp.\ to $\su^\wedge$).
\end{lem}
\begin{proof}
Using \cref{lem:pu_preserves_sheaves}, the same argument in the proof of \cref{lem:adjunction_u_s_u^s} applies.
\end{proof}


\begin{lem} \label{lem:Tpu_preserves_sheaves}
Let $u\colon\cC\to\cD$ be a cocontinuous functor.
The functor $\Tpu\colon\PSh_\cT(\cC)\allowbreak\to\PSh_\cT(\cD)$ sends $\cT$-valued sheaves (resp.\ hypercomplete sheaves) on the site $(\cC,\tau)$ to $\cT$-valued sheaves (resp.\ hypercomplete sheaves) on the site $(\cD,\sigma)$.
In other words, the functor $\Tsu$ (resp.\ $\Tsu^\wedge$) equals the restriction of the functor $\Tpu$ to $\Sh_\cT(\cC,\tau)$ (resp.\ to $\Sh_\cT(\cC, \tau)^\wedge$).
Moreover, we have the following commutative diagrams
\[
\begin{tikzcd}
\Sh_{\cT}(\cD) & \Sh_{\cT}(\cC) \arrow{l}[swap]{\Tsu} \\
\FunR(\Sh(\cD)^\mathrm{op}, \cT) \arrow{u}[sloped,above]{\sim} & \FunR(\Sh(\cC)^\mathrm{op}, \cT), \arrow{u}[sloped,above]{\sim} \arrow{l}[swap]{-\circ u^s}
\end{tikzcd}
\]
and
\[ \begin{tikzcd}
	\Sh_{\cT}(\cD)^\wedge & \Sh_{\cT}(\cC)^\wedge \arrow{l}[swap]{\Tsu^\wedge} \\
	\FunR((\Sh(\cD)^\wedge)^\mathrm{op}, \cT) \arrow{u}[sloped,above]{\sim} & \FunR((\Sh(\cC)^\wedge)^\mathrm{op}, \cT). \arrow{u}[sloped,above]{\sim} \arrow{l}[swap]{-\circ (u^s)^\wedge}
\end{tikzcd} \]
\end{lem}
\begin{proof}
Let us explain the proof in the non-hypercomplete case.
The proof in the hypercomplete case is similar.
To simply notations, we set
\begin{gather*}
\cX\coloneqq\Sh(\cC,\tau),\ \cY\coloneqq\Sh(\cD,\sigma),\ \cX_\cT \coloneqq \FunR(\cX^\mathrm{op}, \cT),\ \cY_\cT \coloneqq \FunR(\cY^\mathrm{op}, \cT),\\
\cX'_\cT\coloneqq\FunR(\PSh(\cC)\op,\cT),\quad \cY'_\cT\coloneqq\FunR(\PSh(\cD)\op,\cT).
\end{gather*}

We claim the commutativity of the following diagram
\[
\begin{tikzcd}
\PSh_\cT(\cD) & \PSh_\cT(\cC) \arrow{l}[swap]{\Tpu} \\
\cY'_\cT \arrow{u}[sloped,above]{\sim} & \cX'_\cT \arrow{l}[swap]{- \circ u^p} \arrow{u}[sloped,above]{\sim} .
\end{tikzcd}
\]
Indeed, since $\Tpu$ and $- \circ u^p$ are right adjoint respectively to $\Tuupperp$ and to $- \circ u_p$, it suffices to show that the diagram
\[
\begin{tikzcd}
\PSh_\cT(\cD) \arrow{r}{\Tuupperp} & \PSh_\cT(\cC) \\
\cY'_\cT \arrow{u}[sloped,above]{\sim} \arrow{r}{- \circ u_p} & \cX'_\cT \arrow{u}[sloped,above]{\sim}
\end{tikzcd}
\]
commutes. Since the vertical morphisms are the compositions with the Yoneda embedding, and $u^p$ is the composition with $u$, it suffices to show that the diagram
\[
\begin{tikzcd}
\cC \arrow{r}{u} \arrow{d}{y_\cC} & \cD \arrow{d}{y_\cD} \\
\PSh(\cC) \arrow{r}{u_p} & \PSh(\cD)
\end{tikzcd}
\]
commutes.
This follows from \cref{lem:u_p_preserves_representables}.

Now we consider the diagram
\[
\begin{tikzcd}
\PSh_\cT(\cD) & \PSh_\cT(\cC) \arrow{l}[swap]{\Tpu} \\
\cY'_\cT \arrow{u} & \cX'_\cT \arrow{l}[swap]{- \circ u^p} \arrow{u} \\
\cY_\cT \arrow{u} & \cX_\cT \arrow{l}[swap]{- \circ u^s} \arrow{u} .
\end{tikzcd}
\]
The bottom square commutes.
We have just shown that the top square commutes as well.
We observe that the morphism $\cY_\cT \to \PSh_\cT(\cD)$ factors as
\[
\cY_\cT \to \Sh_\cT(\cD) \to \PSh_\cT(\cD)
\]
and the same goes for the morphism $\cX_\cT \to \PSh_\cT(\cC)$.
Since the morphism $\cY_\cT \to \Sh_\cT(\cD)$ is an equivalence, we obtain a morphism $\Sh_\cT(\cC) \to \Sh_\cT(\cD)$ which fits into the commutative diagram
\[
\begin{tikzcd}
\PSh_\cT(\cD) & \PSh_\cT(\cC) \arrow{l}[swap]{\Tpu} \\
\Sh_\cT(\cD) \arrow{u} & \Sh_\cT(\cC) \arrow{l} \arrow{u} \\
\cY_\cT \arrow{u}[sloped,above]{\sim} & \cX_\cT. \arrow{l}[swap]{- \circ u^s} \arrow{u}[sloped,above]{\sim}
\end{tikzcd}
\]
This shows that the functor $\Tpu$ preserves $\cT$-valued sheaves.
So the morphism $\Sh_\cT(\cC) \to \Sh_\cT(\cD)$ in the commutative diagram above coincides with $\Tsu$.
So we have proved the commutative diagram in the statement of the lemma.
\end{proof}

\begin{lem} \label{lem:adjunction_Tu^s_Tsu}
Let $u\colon\cC\to\cD$ be a cocontinuous functor.
Then the functor $\Tuuppers$ (resp.\ $\big(\Tuuppers\big)^\wedge$) is left adjoint to the functor $\Tsu$ (resp.\ $\Tsuw$).
\end{lem}
\begin{proof}
Using \cref{lem:Tpu_preserves_sheaves}, the same argument in the proof of \cref{lem:adjunction_u_s_u^s} applies.
\end{proof}

\begin{prop} 
	\label{prop:equivalence_of_topoi}
	Let $(\cC,\tau)$, $(\cD,\sigma)$ be two \infsites.
	Let $u\colon\cC\to\cD$ be a functor.
	Assume that
	\begin{enumerate}[(i)]
		\item \label{item:continuity} $u$ is continuous;
		\item \label{item:cocontinuity} $u$ is cocontinuous;
		\item $u$ is fully faithful;
		\item \label{item:cofinality} for every object $V \in \cD$ there exists a $\sigma$-covering of $V$ in $\mathcal{D}$ of the form $\{u(U_i) \to V\}_{i \in I}$.
		\item For every object $D \in \cD$, the representable presheaf $h_D$ is a hypercomplete sheaf.
	\end{enumerate}
	Then the induced functor
	\[ u_s^\wedge \colon \Sh(\cC, \tau)^\wedge \simeq \Sh(\cD, \sigma)^\wedge . \]
	is an equivalence of $\infty$-categories.
\end{prop}

\begin{proof}
	We start by proving that $u_s^\wedge$ is fully faithful.
	Consider the commutative square
	\begin{equation} \label{eq:right_adjointable_square}
		\begin{tikzcd}
			\PSh(\cC) \arrow{r}{u_p} \arrow{d}{\rL_\cC^\wedge} & \PSh(\cD) \arrow{d}{\rL_\cD^\wedge} \\
			\Sh(\cC, \tau)^\wedge \arrow{r}{u_s^\wedge} & \Sh(\cD, \sigma)^\wedge .
		\end{tikzcd}
	\end{equation}
	We claim that it is right adjointable (cf.\ \cite[7.3.1.2]{HTT}).
	In order to prove this claim, we have to show that the induced natural transformation
	\[ \rL_\cC^\wedge \circ u^p \to (u^s)^\wedge \circ \rL_\cD^\wedge \]
	is an equivalence.
	Since $u$ is cocontinuous, \cref{lem:adjunction_u^s_su} shows that $(u^s)^\wedge$ has a right adjoint, given by $\su^\wedge$.
	In particular, $(u^s)^\wedge$ commutes with colimits.
	It is therefore enough to prove that for every representable presheaf $h_D \in \PSh(\cD)$ the induced transformation
	\[ \rL_\cC^\wedge(u^p(h_D)) \to (u^s)^\wedge(\rL_\cD^\wedge(h_D)) \]
	is an equivalence.
	This follows from the definition of $(u^s)^\wedge$ and the fact that $h_D$ is a hypercomplete sheaf by assumption.
	
	Since the square \eqref{eq:right_adjointable_square} is right adjointable, we see that the image via $\rL_\cC^\wedge$ of the unit of the adjunction $(u_p, u^p)$ coincides with the unit of the adjunction $(u_s^\wedge, (u^s)^\wedge)$.
	Since $u$ is fully faithful, the same goes for $u_p$ (cf.\ \cite[4.3.3.5]{HTT}).
	Thus, the unit of $(u_p, u^p)$ is an equivalence.
	This shows that the unit of $(u_s^\wedge, (u^s)^\wedge)$ is an equivalence as well.
	In other words, we proved that $u_s^\wedge$ is fully faithful.
	
	Next we show that $u_s^\wedge$ is essentially surjective.
	Since $u_s^\wedge$ is a left adjoint and fully faithful, its essential image is closed under colimits.
	Since $\Sh(\cD,\sigma)^\wedge$ is generated by colimits of representable sheaves and $u_s^\wedge$ is fully faithful, it suffices to show that representable sheaves are in the essential image of $u_s^\wedge$.
	Let $V$ be an object in $\cD$.
	By Assumption (\ref{item:cofinality}), there exists a $\sigma$-covering of $V$ of the form $\{u(U_i)\to V\}_{i\in I}$.
	Iterating this application of Assumption (\ref{item:cofinality}) we can construct a hypercover $V^\bullet$ in $\Sh(\cD, \sigma)^\wedge$ of $h_V$ such that for every $n \in \Delta$, $V^n$ belongs to the essential image of $u_s^\wedge$.
	Since $\Sh(\cD, \sigma)^\wedge$ is a hypercomplete $\infty$-topos, we see that
	\[ h_V \simeq | V^\bullet | . \]
		
	Using fully faithfulness of $u_s^\wedge$, we can find a simplicial object $U^\bullet$ in $\Sh(\cC, \tau)^\wedge$ such that $u_s(U^\bullet) \simeq V^\bullet$.
	Finally, from the fact that $u_s^\wedge$ commutes with colimits, we conclude:
	\[ u_s^\wedge(|U^\bullet|) \simeq |u_s^\wedge(U^\bullet)| \simeq |V^\bullet| \simeq h_V . \]
	Thus, $h_V$ is also in the essential image of $u_s^\wedge$.
\end{proof}

\subsection{Change of geometric contexts}

In this section, we study how geometric stacks behave with respect to changes of geometric contexts.

\begin{defin} \label{def:morphism_of_contexts}
Let $(\cC, \tau, \bP)$, $(\cD, \sigma, \bQ)$ be two geometric contexts.
A \emph{morphism of geometric contexts from $(\cC, \tau, \bP)$ to $(\cD, \sigma, \bQ)$} is a continuous functor $u \colon \cC \to \cD$ sending morphisms in $\bP$ to morphisms in $\bQ$.
\end{defin}

\begin{lem}\label{lem:u_s_preserves_geometric_stacks}
Let $u \colon (\cC, \tau, \bP) \to (\cD, \sigma, \bQ)$ be a morphism of geometric contexts.
Assume that $\cC$ and $\cD$ admit pullbacks, and that $u$ commutes with them.
Then the functor $u_s^\wedge \colon\Sh(\cC,\tau)^\wedge \to\Sh(\cD,\sigma)^\wedge$ sends $n$-geometric stacks with respect to $(\cC,\tau,\bP)$ to $n$-geometric stacks with respect to $(\cD,\sigma,\bQ)$ for every $n\ge -1$.
\end{lem}
\begin{proof}
Since $u_s$ is a left adjoint, it commutes with colimits, in particular, with disjoint unions.
It is moreover left exact, so it preserves effective epimorphisms.
By \cref{lem:u_s_preserves_representables}, it sends $(-1)$-geometric stacks with respect to $(\cC,\tau,\bP)$ to $(-1)$-geometric stacks with respect to $(\cD,\sigma,\bQ)$.
Now the same arguments in \cite[Proposition 2.8]{Toen_Algebrisation_2008} imply the statement for every $n\ge -1$.
\end{proof}

\begin{prop}\label{prop:u^s_preserves_geometric_stacks}
Let $(\cC, \tau, \bP)$, $(\cD, \sigma, \bQ)$ be two geometric contexts.
Let $u\colon \cC\to\cD$ be a functor satisfying the assumptions of \cref{prop:equivalence_of_topoi}.
Assume that there exists a non-negative integer $m$ such that for any $X\in\cD$, the sheaf $(u^s)^\wedge(X)$ is an $(m\!-\!1)$-geometric stack with respect to the context $(\mathcal{C},\tau,\bP)$.
Then the functor $(u^s)^\wedge \colon\Sh(\cD,\sigma)^\wedge\to\Sh(\cC,\tau)^\wedge$ sends $n$-geometric stacks with respect to $(\cD,\sigma,\bQ)$ to $(n\!+\!m)$-geometric stacks with respect to $(\cC,\tau,\bP)$ for every $n\ge -1$.
\end{prop}
\begin{proof}
We prove by induction on $n$.
The statement for $n=-1$ is the assumption.
Assume that the statement holds for $n$-geometric stacks with respect to $(\mathcal{D},\sigma,\bQ)$.
Let $X$ be an $(n+1)$-geometric stack with respect to $(\mathcal{D},\sigma,\bQ)$.
By the induction hypothesis, the diagonal morphism of $(u^s)^\wedge(X)$ is $(n+m)$-representable.
Now let $\{U_i\}_{i\in I}$ be an $(n+1)$-atlas of $X$.
For every $i\in I$, let $\{V_{ij}\}_{j\in J_i}$ be an $(m-1)$-atlas of $(u^s)^\wedge(U_i)$ with respect to $(\mathcal{C},\tau,\bP)$.
By the induction hypothesis, the morphisms $(u^s)^\wedge(U_i\to X)$ are $(n+m)$-representable with respect to $(\mathcal{C},\tau,\bP)$.
By assumption, the morphisms $V_{ij}\to (u^s)^\wedge(U_i)$ are $(m-1)$-representable with respect to $(\mathcal{C},\tau,\bP)$.
So the morphisms $V_{ij}\to (u^s)^\wedge(X)$ are $(n+m)$-representable with respect to $(\mathcal{C},\tau,\bP)$ by composition.
Therefore, $\{V_{ij}\}_{i\in I, j\in J_i}$ constitutes an $(n+m+1)$-atlas of $(u^s)^\wedge(X)$ with respect to $(\mathcal{C},\tau,\bP)$.
In other words, the stack $X$ is $(n+m+1)$-geometric with respect to $(\mathcal{C},\tau,\bP)$.
\end{proof}

\begin{cor} \label{cor:equivalence_of_geometric_stacks}
Let $u \colon (\cC, \tau, \bP) \to (\cD, \sigma, \bQ)$ be a morphism of geometric contexts satisfying the assumptions of \cref{prop:u^s_preserves_geometric_stacks}.
Then the functors
\[u_s^\wedge \colon\Sh(\cC,\tau)^\wedge \leftrightarrows\Sh(\cD,\sigma)^\wedge \colon (u^s)^\wedge\]
form an equivalence of $\infty$-categories which preserves the subcategories of geometric stacks.
\end{cor}
\begin{proof}
It follows from \cref{prop:equivalence_of_topoi}, \cref{lem:u_s_preserves_geometric_stacks} and \cref{prop:u^s_preserves_geometric_stacks}.
\end{proof}

\section{Examples of higher geometric stacks} \label{sec:examples}

In this section, we define higher geometric stacks in four concrete geometrical settings,
namely, in algebraic geometry, in complex analytic geometry, in non-archimedean analytic geometry (or rigid analytic geometry), and in a relative algebraic setting.

\subsection{Higher algebraic stacks} \label{sec:algebraic_stacks}

Let $(\mathrm{Aff},\tauet)$ denote the category of affine schemes endowed with the étale topology.
The étale topology $\tauet$ consists of coverings of the form $\{U_i\to X\}$ where every morphism $U_i\to X$ is étale and $\coprod U_i\to X$ is surjective.
Let $\bPsm$ denote the class of smooth morphisms.
The triple $(\mathrm{Aff},\tauet,\bPsm)$ is a geometric context in the sense of \cref{def:context}.
Geometric stacks with respect to the context $(\mathrm{Aff},\tauet,\bPsm)$ are called \emph{higher algebraic stacks}.
We will simply say \emph{algebraic stacks} afterwards.

We note that in the usual treatment of stacks \cite{Deligne_Irreducibility_1969,Artin_Versal_1974,Laumon_Champs_2000,stacks-project},
one uses the category of all schemes instead of just affine schemes.
Let $(\mathrm{Sch}, \tauet)$ denote the category of schemes endowed with the étale topology.
If we consider geometric stacks with respect to the context $(\mathrm{Sch},\tauet,\bPsm)$,
\cref{cor:equivalence_of_geometric_stacks} ensures that we obtain an equivalent definition of algebraic stacks.

\subsection{Higher complex analytic stacks}\label{sec:canal_stacks}

Let $(\Stn_\C, \tauan)$ denote the category of Stein complex analytic spaces endowed with the analytic topology.
The analytic topology $\tauan$ on $\Stn_\C$ consists of coverings of the form $\{U_i\to X\}$ where every $U_i\to X$ is an open immersion and $\coprod U_i\to X$ is surjective.
Let $\bPsm$ denote the class of smooth morphisms.
The triple $(\Stn_\C, \tauan, \bPsm)$ is a geometric context in the sense of \cref{def:context}.
Geometric stacks with respect to the context $(\Stn_\C, \tauan,\bPsm)$ are called \emph{higher \canal stacks}.
We will simply say \emph{\canal stacks}.

Besides the analytic topology $\tauan$, we can consider the étale topology $\tauet$ on $\Stn_\C$.
It consists of coverings of the form $\{U_i\to X\}$ where every $U_i\to X$ is a local biholomorphism and $\coprod U_i\to X$ is surjective.
If we consider geometric stacks with respect to the context $(\Stn_\C,\tauet,\bPsm)$, \cref{cor:equivalence_of_geometric_stacks} ensures that we obtain an equivalent definition of \canal stacks.

Moreover, we can consider the whole category $\An_\C$ of complex analytic spaces.
Similarly, we have geometric contexts $(\An_\C,\tauan,\bPsm)$ and $(\An_\C,\tauet,\bPsm)$.
By \cref{cor:equivalence_of_geometric_stacks}, they all give rise to the same notion of \canal stacks.

\subsection{Higher non-archimedean analytic stacks}\label{sec:kanal_stacks}

Let $k$ be a non-archimedean field with non-trivial valuation.
Let $(\Afd_k, \tauqet)$ denote the category of strictly $k$-affinoid spaces (\cite{Berkovich_Spectral_1990,Berkovich_Etale_1993}) endowed with the quasi-étale topology.
In the affinoid case, the quasi-étale topology consists of coverings of the form $\{U_i\to X\}_{i\in I}$ where every $U_i\to X$ is quasi-étale, $\coprod U_i\to X$ is surjective and $I$ is finite.
We refer to \cite[§3]{Berkovich_Vanishing_1994} for the definition of quasi-étale topology in the general case.
Let $\bPqsm$ denote the class of quasi-smooth morphisms\footnote{We refer to Antoine Ducros' work \cite{Ducros_Families_2011} for the notion of quasi-smooth morphism in non-archimedean analytic geometry.
	In the literature of derived geometry, local complete intersection morphisms are sometimes called quasi-smooth morphisms, which is a coincidence of terminology.
	In order to avoid confusion, we will use the terminology quasi-smooth only in the context of non-archimedean analytic geometry.}.
The triple $(\Afd_k, \tauqet,\bPqsm)$ is a geometric context in the sense of \cref{def:context}.
Geometric stacks with respect to the context $(\Afd_k, \tauqet,\bPqsm)$ are called \emph{higher \kanal stacks}.
We will simply say \emph{\kanal stacks}.

Let $(\An_k, \tauqet)$ denote the category of strictly $k$-analytic spaces endowed with the quasi-étale topology.
\cref{cor:equivalence_of_geometric_stacks} ensures that the triple $(\An_k,\tauqet,\bPqsm)$ induces an equivalent definition of \kanal stacks.

\begin{rem}
The notion of strictly $k$-analytic 1-stacks was defined and applied to study a non-archimedean analog of Gromov's compactness theorem by Tony Yue Yu in \cite{Yu_Gromov_2014}.
\end{rem}

\subsection{Relative higher algebraic stacks} \label{sec:relative_algebraic}

It is useful to consider the following relatively algebraic setting.

Let $A$ be either a Stein algebra (i.e.\ the algebra of functions on a Stein complex analytic space), 
or a $k$-affinoid algebra.
Let $(\Afflfp_A, \tauet)$ denote the category of affine schemes locally finitely presented over $\Spec(A)$ endowed with the étale topology.
Let $\bPsm$ denote the class of smooth morphisms.
The triple $(\Afflfp_A, \tauet, \bPsm)$ is a geometric context in the sense of \cref{def:context}.
Geometric stacks with respect to this context are called \emph{higher algebraic stacks relative to $A$}.
We will simply say \emph{algebraic stacks relative to $A$}.


\section{Proper morphisms of analytic stacks} \label{sec:proper_morphisms}

In this section, we introduce the notion of weakly proper pairs of analytic stacks.
Weakly proper pairs are then used to define proper morphisms of analytic stacks.

We use the geometric context $(\mathrm{Aff},\tauet,\bPsm)$ for algebraic stacks, $(\Stn_\C, \tauan, \bPsm)$ for \canal stacks and $(\Afd_k, \tauqet,\bPqsm)$ for \kanal stacks introduced in \cref{sec:examples}.
We simply say \emph{analytic stacks} whenever a statement applies to both \canal stacks and \kanal stacks.

\begin{defin} \label{def:canal_relatively_compact}
Let $U\to V$ be a morphism of representable \canal stacks over a representable \canal stack $S$.
We say that $U$ is \emph{relatively compact} in $V$ over $S$, and denote $U\Subset_S V$,  if
\begin{enumerate}[(i)]
\item $U\to V$ is an open immersion;
\item the closure of $U$ in $V$ is proper over $S$.
\end{enumerate}
\end{defin}

\begin{defin} \label{def:kanal_relatively_compact}
Let $U\to V$ be a morphism of representable \kanal stacks over a representable \kanal stack $S$.
We say that $U$ is \emph{relatively compact} in $V$ over $S$, and denote $U\Subset_S V$,  if
\begin{enumerate}[(i)]
\item $U\to V$ is an embedding of an affinoid domain;
\item $U$ is contained in the relative interior of $V$ with respect to $S$, i.e.\ $U\subset\Int(V/S)$.
\end{enumerate}
\end{defin}

\begin{defin}\label{def:canal_weakly_proper_pair}
A morphism $X\to Y$ of \canal stacks over a representable \canal stack $S$ is said to define a \emph{weakly proper pair} if
\begin{enumerate}[(i)]
\item it is representable by open immersions;
\item there exists a finite atlas $\{V_i\}_{i\in I}$ of $Y$, an open subset $U_i$ of $V_i\times_Y X$ for every $i\in I$, such that the composition $U_i\to V_i\times_Y X\to V_i$ satisfies $U_i\Subset_S V_i$, and that $\{U_i\}_{i\in I}$ is an atlas of $X$.
\end{enumerate}
\end{defin}

\begin{defin}\label{def:kanal_weakly_proper_pair}
A morphism $X\to Y$ of \kanal stacks over a representable \kanal stack $S$ is said to define a \emph{weakly proper pair} if
\begin{enumerate}[(i)]
\item it is representable by embeddings of affinoid domain;
\item there exists a finite atlas $\{V_i\}_{i\in I}$ of $Y$, an affinoid domain $U_i$ of $V_i\times_Y X$ for every $i\in I$, such that the composition $U_i\to V_i\times_Y X\to V_i$ satisfies $U_i\Subset_S V_i$, and that $\{U_i\}_{i\in I}$ is an atlas of $X$.
\end{enumerate}
\end{defin}

\begin{defin}\label{def:analytic_weakly_proper}
A morphism $f\colon X\to Y$ of analytic stacks is said to be \emph{weakly proper} if there exists an atlas $\{Y_i\}_{i\in I}$ of $Y$ such that the identity $X\times_Y Y_i\to X\times_Y Y_i$ defines a weakly proper pair over $Y_i$ for every $i\in I$.
\end{defin}

\begin{defin}
A morphism $X\to Y$ of algebraic stacks (or analytic stacks) is said to be \emph{surjective} if there exists an atlas $\{Y_i\}_{i\in I}$ of $Y$, an atlas $\{U_{ij}\}_{j\in J_i}$ of $X\times_Y Y_i$ for every $i\in I$, such that the induced morphism $\coprod_{j\in J_i} U_{ij}\to Y_i$ is surjective.
\end{defin}

\begin{defin}\label{def:algebraic_weakly_proper}
A morphism $f\colon X\to Y$ of algebraic stacks is said to be \emph{weakly proper} if there exists an atlas $\{Y_i\}_{i\in I}$ of $Y$ such that for every $i\in I$,
there exists a scheme $P_i$ proper over $Y_i$ and a surjective $Y_i$-morphism from $P_i$ to $X\times_Y Y_i$.
\end{defin}

\begin{defin}
We define by induction on $n\geq 0$.
\begin{enumerate}[(i)]
\item An $n$-representable morphism of analytic stacks (or algebraic stacks) is said to be \emph{separated} if its diagonal being an $(n\!-\!1)$-representable morphism is proper.
\item An $n$-representable morphism of analytic stacks (or algebraic stacks) is said to be \emph{proper} if it is separated and weakly proper.
\end{enumerate}
\end{defin}

\begin{lem} \label{lemma wpp representable case}
Let $X,Y$ be representable analytic stacks over a representable analytic stack $S$.
An $S$-morphism $f \colon X \to Y$ defines a weakly proper pair over $S$ if and only if $X$ is relatively compact in $Y$ over $S$.
\end{lem}
\begin{proof}
The ``if'' part is obvious.
Let us prove the ``only if'' part.

In the \canal case, by definition $f \colon X \to Y$ is an open immersion.
Let $\{U_i\}_{i\in I}$, $\{V_i\}_{i\in I}$ be the atlases in \cref{def:canal_weakly_proper_pair}.
Denote by $u_i \colon U_i \to X$ and by $v_i \colon V_i \to Y$ the given maps.
Let $u\colon\coprod U_i\to X$ and $v\colon\coprod V_i\to Y$.
We claim that $v(\widebar{U}) = \widebar{X}$. Clearly, $v(\widebar{U}) \subset \widebar{v(U)} = \widebar{X}$.
Fix now $y \in \widebar{X}$ and assume by contradiction that $y \not \in v(\widebar{U})$.
Since $\widebar{U}$ is compact, the same goes for $v(\widebar{U})$.
In particular, this is a closed subset of $Y$.
We can therefore find a neighborhood $W$ of $y$ satisfying $W \cap v(\widebar{U}) = \emptyset$.
However, we can also choose $x \in W \cap X$ and $p \in U$ such that $v(p) = u(p) = x$.
Since $p \in U \subset \widebar{U}$, this provides a contradiction.

In the \kanal case, by definition, $X$ is an affinoid domain in $Y$.
Let $\{U_i\}_{i\in I}$, $\{V_i\}_{i\in I}$ be the atlases in \cref{def:kanal_weakly_proper_pair}.
Let $x$ be a point in $X$.
Choose a point $\tilde x$ in $U_i$ for some $i\in I$ such that $\tilde x$ maps to $x$.
By definition, $\tilde x\in\Int(V_i/S)$.
Let $\psi$ denote the map from $V_i$ to $Y$.
By \cite[Proposition 2.5.8(iii)]{Berkovich_Spectral_1990}, we have
\[\Int(V_i/S) = \Int(V_i/Y)\cap\psi\inv(\Int(Y/S)).\]
Therefore, $\tilde x\in\psi\inv(\Int(Y/S))$.
So $x=\psi(\tilde x)\in\Int(Y/S)$ for any $x\in X$.
In other words, we have proved that $X\Subset_S Y$.
\end{proof}

\begin{notation}
From now on, we are allowed to use the same symbol $X\Subset_S Y$ to denote a weakly proper pair of analytic stacks $X$, $Y$ over a representable analytic stack $S$.
We will simply write $X \Subset Y$ in the absolute case.
\end{notation}

\begin{lem} \label{lem:wpp_base_change}
Let $X,Y,S$ be analytic stacks.
Assume that $S$ is representable and that $X\Subset_S Y$.
Then for any representable stack $T$ and any morphism $T\to S$,
we have $X \times_S T \Subset_T Y \times_S T$.
\end{lem}
\begin{proof}
We can assume without loss of generality that $X$ and $Y$ are representable as well.
First we prove the \canal case.
Denote by $f \colon T \to S$ and $g \colon Y \to S$ the given morphisms.
Consider the pullback diagram
\[
\begin{tikzcd}
Y \times_S T \arrow{d}{g'} \arrow{r} & Y \arrow{d}{g} \\
T \arrow{r} & S.
\end{tikzcd}
\]
If $K$ is a subset of $T$, then
\[
(g')^{-1}(K) = K \times_S g^{-1}(f(K)).
\]
If $K$ is a compact subset of $T$, we have
\begin{align*}
(g')^{-1}(K) \cap \big( T \times_S \widebar{X} \big) & = \big(K \times_S g^{-1}(f(K))\big) \cap \big(T \times_S \widebar{X} \big) \\
& = K \times_S \big(g^{-1}(f(K)) \cap \widebar{X} \big).
\end{align*}
Since $f(K)$ is compact, by hypothesis we see that $g^{-1}(f(K)) \cap \widebar{X}$ is compact.
Since $S$ is separated, the natural map
\[
K \times_S (g^{-1}(f(K)) \cap \widebar{X}) \to K \times (g^{-1}(f(K)) \cap \widebar{X})
\]
is a closed immersion.
We conclude that $K \times_S (g^{-1}(f(K)) \cap \widebar{X})$ is compact as well.
Observe finally that $\widebar{T \times_S X} \subset T \times_S \widebar{X}$, so that
\[
(g')^{-1}(K) \cap \widebar{T \times_S X} = (g')^{-1}(K) \cap (T \times_S \widebar{X}) \cap \widebar{T \times_S X}
\]
is closed in a compact, hence it is compact itself, completing the proof in the \canal case.

Now we prove the \kanal case.
The assumption $X\Subset_S Y$ implies that there is a positive real number $\epsilon<1$, a positive integer $n$ and a commutative diagram
\[\begin{tikzcd}
X \arrow[hook]{d} \arrow[hook]{r} & \bD^n_S(\epsilon) \arrow[hook]{d} \\
Y \arrow[hook, closed]{r} & \bD^n_S(1),
\end{tikzcd}\]
where the bottom line denotes a closed immersion.
Taking fiber product with $T$ over $S$, we obtain
\[\begin{tikzcd}
X\times_S T \arrow[hook]{d} \arrow[hook]{r} & \bD^n_T(\epsilon) \arrow[hook]{d} \\
Y\times_S T \arrow[hook, closed]{r} & \bD^n_T(1),
\end{tikzcd}\]
where the bottom line denotes a closed immersion.
So we have proved the lemma in the \kanal case.
\end{proof}

\begin{prop} \label{prop:base_change_separatedness_properness}
\begin{enumerate}[(i)]
\item Separated maps are stable under base change.
\item Proper maps are stable under base change.
\end{enumerate}
\end{prop}
\begin{proof}
Using \cref{lem:wpp_base_change}, the proposition follows from induction on the geometric level of the stacks.
\end{proof}

\begin{lem} \label{lem:shrink_wpp}
Let $S, U, V, W, Y$ be representable analytic stacks such that
$W \Subset_S Y$ and $U \Subset_Y V$.
Then we have $U \times_Y W \Subset_S V$.
\end{lem}

\begin{proof}
We first prove the \canal case.
Denote by $f \colon Y \to S$ and by $g \colon V \to Y$ the given morphisms.
We have open immersions $W \times_Y U \to U$ and $W \times_Y U \to V$.
We observe that the closure of $W \times_Y U$ in $V$ coincides with the closure of $W \times_Y U$ in $\widebar{U}$,
which we denote by $\widebar{W \times_Y U}$.

We claim that $\widebar{W \times_Y U} \subset g^{-1}(\widebar{W})$.
Let $p \in \widebar{W \times_Y U}$. If $\Omega'$ is an open neighborhood of $g(p)$ in $Y$, we can find an open neighborhood $\Omega$ of $p$ in $V$ such that $g(\Omega) \subset \Omega'$.
By hypothesis we can find $q \in \Omega \cap (W \times_Y U)$.
We observe that $g(q) \in W$.
It follows that $g(p) \in \widebar{W}$, completing the proof of the claim.

Let $K$ be a compact subset of $S$.
We have
\begin{align*}
g^{-1}(f^{-1}(K)) \cap \widebar{W \times_Y U} & = g^{-1}(f^{-1}(K)) \cap g^{-1}(\widebar{W}) \cap \widebar{W \times_Y U} \\
& = g^{-1}(f^{-1}(K) \cap \widebar{W}) \cap \widebar{W \times_Y U} \\
& = (g^{-1}(f^{-1}(K) \cap \widebar{W}) \cap \widebar{U}) \cap \widebar{W \times_Y U}.
\end{align*}
We remark that $f^{-1}(K) \cap \widebar{W}$ is a compact subset of $Y$ by hypothesis.
Therefore $K' \coloneqq g^{-1}(f^{-1}(K) \cap \widebar{W}) \cap \widebar{U}$ is a compact subset of $V$. Since $\widebar{W \times_Y U}$ is a closed subset of $V$, we see that $K' \cap \widebar{W \times_Y U}$ is a closed subset of a compact and hence it is itself compact, completing the proof in the \canal case.

Now we turn to the \kanal case.
By the assumptions, we have $W\subset\Int(Y/S)$ and $U\subset\Int(V/Y)$.
Let $\psi$ denote the map from $V$ to $Y$.
We have
\[U\times_Y W \simeq U\cap\psi\inv(W)\subset \Int(V/Y)\cap\psi\inv(\Int(Y/S))=\Int(V/S),\]
the last equality being \cite[Proposition 2.5.8(iii)]{Berkovich_Spectral_1990}.
So we have proved the lemma.
\end{proof}

\begin{lem} \label{lem:variant_wpp_stability}
Let $U,V,W,W',Y$ be representable analytic stacks over a representable analytic stack $S$.
Assume that
$W \Subset_S W' \subset Y$ and $U \Subset_Y V$.
Then we have $U \times_Y W\Subset_S V \times_Y W'$.
\end{lem}
\begin{proof}
By \cref{lem:wpp_base_change}, the assumption that $U\Subset_Y V$ implies that
\[
U \times_{Y} W' \Subset_{W'} V \times_{Y} W'.
\]
Then \cref{lem:shrink_wpp} shows that
\[
U \times_Y W = (U \times_Y W') \times_{W'} W \Subset_S V \times_Y W'.
\]
\end{proof}

\begin{lem}\label{lem:product_of_wpp_representable}
Let $S$ be a representable analytic stack and consider a commutative diagram of representable analytic stacks over $S$
\[
\begin{tikzcd}
U \arrow{d} \arrow{r} &  X \arrow{d} & U' \arrow{d} \arrow{l} \\
V \arrow{r} & Y & V' \arrow{l}.
\end{tikzcd}
\]
Assume that $U \Subset_S V$, $U' \Subset_S V'$ and that $Y$ is separated over $S$.
Moreover, assume that $X\to Y$ is an open immersion in the \canal case and an affinoid domain in the \kanal case.
Then we have $U \times_X U' \Subset_S V \times_Y V'$.
\end{lem}

\begin{proof}
Consider the following diagram
\[
\begin{tikzcd}
U \times_{Y} U' \arrow{d} \arrow{r} & V \times_{Y} V' \arrow{d} \arrow{r} & Y \arrow{d}{\Delta_{Y/S}} \\
U \times_S U' \arrow{r} & V \times_S V' \arrow{r} & Y \times_S Y.
\end{tikzcd}
\]
The right and the outer squares are pullbacks.
It follows that the left square is a pullback as well.
By hypothesis the diagonal morphism $\Delta_{Y/S}\colon Y\to Y\times_S Y$ is a closed immersion.
Therefore, both morphisms $U\times_Y U'\to U\times_S U'$ and $V\times_Y V'\to V\times_S V'$ are closed immersions.

Let us first consider the \canal case.
The assumption that $X\to Y$ is an open immersion implies that it is a monomorphism.
So the map $U \times_X U' \to U \times_Y U'$ is an isomorphism.
Therefore the map $U \times_X U' \to V \times_Y V'$ is an open immersion.
Since the morphisms $U\times_Y U'\to U\times_S U'$ and $V\times_Y V'\to V\times_S V'$ are closed immersions, we have
\[
\widebar{U \times_Y U'} = (V \times_Y V') \cap \widebar{U \times_S U'} = (V \times_Y V') \cap (\widebar{U} \times_S \widebar{U'}).
\]
Let $q \colon V \times_S V' \to S$ denote the natural map and $p \colon V \times_Y V' \to S$ its restriction.
For every subset $K \subset S$ we have
\[
p^{-1}(K) = (V \times_Y V') \cap q^{-1}(K).
\]
Therefore
\[
p^{-1}(K) \cap \widebar{U \times_Y U'} = q^{-1}(K) \cap (\widebar{U} \times_S \widebar{U'}) \cap (V \times_Y V').
\]
Since $V \times_Y V'$ is closed, it suffices show that $q^{-1}(K) \cap (\widebar{U} \times_S \widebar{U'})$ is compact whenever $K$ is compact.
Let $f \colon V \to S$ and $g \colon V' \to S$ denote the given maps, we have $q^{-1}(K) = f^{-1}(K) \times_S g^{-1}(K)$ and therefore
\[
q^{-1}(K) \cap (\widebar{U} \times_S \widebar{U'}) = (f^{-1}(K) \cap \widebar{U}) \times_S (g^{-1}(K) \cap \widebar{U'}),
\]
which is compact (because $S$ is Hausdorff).
So we have proved the lemma in the \canal case.

Now let us turn to the \kanal case.
The assumption that $X\to Y$ is an affinoid domain implies that it is a monomorphism.
So the map $U \times_X U' \to U \times_Y U'$ is an isomorphism.
Therefore the map $U \times_X U' \to V \times_Y V'$ is an embedding of an affinoid domain.

By the assumptions, there exists a positive real number $\epsilon<1$, positive integers $n$, $n'$, and commutative diagrams
\begin{equation*}
\begin{tikzcd}
U \arrow[hook]{d} \arrow[hook]{r} & \bD^n_S(\epsilon) \arrow[hook]{d} \\
V \arrow[hook, closed]{r} & \bD^n_S(1),
\end{tikzcd}
\hspace{50pt}
\begin{tikzcd}
U' \arrow[hook]{d} \arrow[hook]{r} & \bD^{n'}_S(\epsilon) \arrow[hook]{d} \\
V' \arrow[hook, closed]{r} & \bD^{n'}_S(1),
\end{tikzcd}
\end{equation*}
where $\bD^n_S(\epsilon)$ denotes the $n$-dimensional closed polydisc with radius $\epsilon$, similar for the others,
and the two arrows on the bottom denote closed immersions.

Taking fiber product of the two commutative diagrams above over $S$, we obtain
\begin{equation}\label{eq:product_of_wpp_representable}
\begin{tikzcd}
U\times_S U' \arrow[hook]{d} \arrow[hook]{r} & \bD^{n+n'}_S(\epsilon) \arrow[hook]{d} \\
V\times_S V' \arrow[hook, closed]{r} & \bD^{n+n'}_S(1),
\end{tikzcd}
\end{equation}
where the bottom line denotes a closed immersion.
So we have proved that $U\times_S U'\Subset_S V\times_S V'$.

Combining the closed immersions $U\times_Y U'\to U\times_S U'$ and $V\times_Y V'\to V\times_S V'$ with \cref{eq:product_of_wpp_representable}, we obtain
\[\begin{tikzcd}
U\times_Y U' \arrow[hook]{d} \arrow[hook, closed]{r} & U\times_S U' \arrow[hook]{d} \arrow[hook]{r} & \bD^{n+n'}_S(\epsilon) \arrow[hook]{d} \\
V\times_Y V' \arrow[hook, closed]{r} & V\times_S V' \arrow[hook, closed]{r} & \bD^{n+n'}_S(1).
\end{tikzcd}\]
Since the composition of the bottom line is a closed immersion, we have proved that $U\times_X U'\simeq U\times_Y U'\Subset_S V\times_Y V'$, completing the proof.
\end{proof}

\begin{prop} \label{prop:product_of_wpp}
Let $S$ be a representable analytic stack and let $X$ be an analytic stack separated over $S$.
Let $U,U',V,V'$ be representable analytic stacks over $X$, such that $U \Subset_S V$ and $U' \Subset_S V'$.
Then we have $U \times_X U' \Subset_S V \times_X V'$.
\end{prop}
\begin{proof}
As in the beginning of the proof of \cref{lem:product_of_wpp_representable}, we have the following pullback diagram
\[
\begin{tikzcd}
U \times_{X} U' \arrow{d} \arrow{r} & V \times_{X} V' \arrow{d} \arrow{r} & X \arrow{d}{\Delta_{X/S}} \\
U \times_S U' \arrow{r} & V \times_S V' \arrow{r} & X \times_S X.
\end{tikzcd}
\]
Since $X$ is separated over $S$, the diagonal morphism $\Delta_{X / S}$ is proper.
It follows from \cref{prop:base_change_separatedness_properness} that both $V \times_{X} V' \to V \times_S V'$ and $U \times_{X} U' \to U \times_S U'$ are proper morphisms.

Set $U'' \coloneqq U \times_S U'$ and $V'' \coloneqq V \times_S V'$.
\cref{lem:product_of_wpp_representable} implies that $U'' \Subset_S V''$.
Let $\Omega \Subset_{V''} \Omega'$ be a weakly proper pair of representable stacks over $V''$. 
Using \cref{lem:shrink_wpp}, we deduce that $\Omega \times_{V''} U'' \Subset_S \Omega'$.

Let us choose a finite double atlas $\{\Omega_i \Subset_{V''} \Omega_i'\}_{i\in I}$ of $V \times_{X} V'$.
Using \cref{lem:shrink_wpp}, we deduce that $\Omega_i \times_{V''} U'' \Subset_S \Omega'_i$ for every $i\in I$.
Moreover $\{\Omega_i \times_{V''} U''\}_{i\in I}$ gives a finite atlas of $U \times_{X} U'$.
Therefore, $U \times_{X} U' \to V \times_{X} V'$ is a weakly proper pair over $S$.
\end{proof}

\begin{cor}
Let $S$ be a representable analytic stack and consider a commutative diagram of analytic stacks over $S$
\[
\begin{tikzcd}
U \arrow{d} \arrow{r} &  X \arrow{d} & U' \arrow{d} \arrow{l} \\
V \arrow{r} & Y & V' \arrow{l}.
\end{tikzcd}
\]
Assume that $U \Subset_S V$, $U' \Subset_S V'$ and that $Y$ is separated over $S$.
Moreover, assume that $X \to Y$ is representable by open immersions in the \canal case and by embeddings of affinoid domains in the \kanal case.
Then we have $U \times_X U' \Subset_S V \times_Y V'$.
\end{cor}

\begin{proof}
Since $X \to Y$ is in particular representable by monomorphisms, the canonical map
\[
U \times_X U' \to U \times_Y U'
\]
is an isomorphism.
We are therefore reduced to the case where $X \to Y$ is the identity map.
In this case, choose finite atlases $\{U_i\}_{i \in I}$, $\{V_i\}_{i \in I}$ of $U$ and $V$ and $ \{U_j'\}_{j \in J}$, $\{V_j'\}_{j \in J}$ of $U'$ and $V'$ satisfying the relations $U_i \Subset_S V_i$ and $U_j' \Subset V_j'$.
It follows from \cref{prop:product_of_wpp} that
\[
U_i \times_X U_j' \Subset V_i \times_X V_j'.
\]
Let $\{W_{ijk}\}$ and $\{W_{ijk}'\}$ be respectively finite atlases of $U_i \times_X U_j'$ and of $V_i \times_X V_j'$ satisfying $W_{ijk} \Subset_S W_{ijk}'$.
We see that the collection $\{W_{ijk}\}$ forms, as the indices $i$, $j$ and $k$ vary, an atlas for $U \times_X U'$, while $\{W_{ijk}'\}$ forms an atlas for $V_i \times_X V_j'$. This completes the proof.
\end{proof}

\section{Direct images of coherent sheaves} \label{sec:direct_images}

\subsection{Sheaves on geometric stacks} \label{sec:sheaves_on_stacks}

In this section, we study sheaves on geometric stacks and operations on these sheaves.

Let $(\cC,\tau,\bP)$ be a geometric context and let $X$ be a geometric stack with respect to this geometric context.

In order to speak of sheaves on $X$, we need to specify a site associated to $X$ on which the sheaves will live.
Our choice is an analog of the classical \emph{lisse-étale} site \cite{Laumon_Champs_2000,Olsson_Sheaves_2007}.
We remark that the functoriality problem associated with lisse-étale sites disappears in our $\infty$-categorical approach (see Olsson \cite{Olsson_Sheaves_2007} for a description of the problem and a different solution).
Other possible choices are analogs of the big sites as in \cite{stacks-project}.
However, as pointed out in Tag 070A loc.\ cit., the pushforward functor defined via the big sites does not preserve quasi-coherent sheaves.
This drawback would make the theory more complicated.

Let $\CXP$ denote the full subcategory of the overcategory ${\Sh(\cC,\tau)^\wedge}_{/X}$ spanned by $\bP$-morphisms from representable stacks to $X$.
The topology $\tau$ on $\cC$ induces a topology on $\CXP$ such that coverings in $\CX$ are coverings in $\cC$ after forgetting the maps to $X$.
We denote the induced topology again by $\tau$.
So we obtain an \infsite $(\CXP,\tau)$.

For any presentable \infcat $\cT$, we denote $\Sh_\cT(X)\coloneqq\Sh_\cT(\CXP,\tau)$, the \infcat of $\cT$-valued sheaves on $X$.
We denote by $\Sh_\cT(X)^\wedge$ the full subcategory spanned by hypercomplete $\cT$-valued sheaves.

Let $\Ab$ denote the category of abelian groups.
Let $\DAb$ be the unbounded derived \infcat of $\Ab$ (cf.\ \cite[1.3.5.8]{Lurie_Higher_algebra}).
It is a symmetric monoidal \infcat.

The natural t-structure on $\DAb$ induces a t-structure on $\Sh_{\DAb}(X)$ in the following way.
Let $\rH^n\colon\DAb\to\Ab$ denote the $n\textsuperscript{th}$  cohomology functor.
It induces a functor from $\Sh_{\DAb}(X)$ to the category of presheaves of abelian groups on $X$.
By sheafification, we obtain a functor
\[\rH^n\colon\Sh_{\DAb}(X)\to\Sh_\Ab(X).\]
Let $\Sh^{\le 0}_{\DAb} (X)$ (resp.\ $\Sh^{\ge 0}_{\DAb} (X)$) denote the full subcategory of $\Sh_{\DAb}(X)$ spanned by objects $\cF\in\Sh_{\DAb}(X)$ such that $\rH^n(\cF)=0$ for all $n>0$ (resp.\ $n<0$).
We have the following proposition.

\begin{prop}[{\cite[1.7]{DAG-VII}}] \label{prop:t-structure}
The full subcategories $\Sh^{\le 0}_{\DAb} (X)$ and $\Sh^{\ge 0}_{\DAb} (X)$ form a t-structure on $\Sh_{\DAb}(X)$. Moreover, the functor $\mathrm H^0 \colon \Sh_{\DAb}(X) \to \Sh_{\Ab}(X)$ induces an equivalence between the heart of this t-structure and the category $\Sh_{\Ab}(X)$.
\end{prop}

We will denote by $\tau_{\le n}$ and $\tau_{\ge n}$ the truncation functors associated to the t-structure above.

\begin{rem}[\cite{DAG-VII}]
The $\infty$-symmetric monoidal structure on $\DAb$ induces an $\infty$-symmetric monoidal structure on $\Sh_{\DAb}(X)$, which is compatible with the t-structure in \cref{prop:t-structure}.
\end{rem}

For the next proposition we will denote by $t_{\le n}$ and $t_{\ge n}$ the truncation functors for the canonical $t$-structure on $\DAb$.
The dual of \cite[1.2.1.5]{Lurie_Higher_algebra} shows that each functor $t_{\le n}$ is a right adjoint.

\begin{prop} \label{prop:hypercompleteness_DAb}
	Let $(\cC, \tau)$ be an $\infty$-site and let $\cF$ be a $\DAb$-valued sheaf on $(\cC,\tau)$.
	The following conditions are equivalent:
	\begin{enumerate}[(i)]
		\item The sheaf $\cF$ is left $t$-complete for the $t$-structure on $\Sh_{\DAb}(\cC, \tau)$, in the sense that $\cF\simeq\lim_m\tau_{\ge m}\cF$.
		\item The Postnikov tower of the truncated sheaf $t_{\le 0} \circ \cF$ converges.
				\item The sheaf $\cF$ is a hypercomplete $\DAb$-valued sheaf.
		\item The truncated sheaf $t_{\le 0} \circ \cF$ is hypercomplete when seen as an $\cS$-valued sheaf.
			\end{enumerate}
\end{prop}

\begin{proof}
	We first address the equivalence between (i) and (ii).
	Form a fiber sequence
	\[ \cF' \to \cF \to \cF'' \]
	where $\cF' \in \Sh_{\DAb}^{\le 0}(\cC, \tau)$ and $\cF'' \in \Sh_{\DAb}^{\ge 1}(\cC, \tau)$.
	Since the object $\cF''$ is left bounded, it is also left complete.
	It follows that $\cF$ is left $t$-complete if and only if $\cF'$ is.
	On the other side, $t_{\le 0} \circ \cF'' = 0$
		and therefore $t_{\le 0} \circ \cF \simeq t_{\le 0} \circ \cF'$.
	In other words, we can replace $\cF$ by $\cF'$, or, equivalently, we can assume from the very beginning that $\cF$ belongs to $\Sh_{\DAb}^{\le 0}(\cC, \tau)$.
	In this case, we can use \cite[1.2.1.9]{Lurie_Higher_algebra} to see that the Postnikov tower of $\cF$ (computed in $\Sh_{\DAb}^{\le 0}(\cC, \tau)$) coincides with the tower of left $t$-truncations.
	Since $t_{\le 0}$ is a right adjoint, it commutes with limits and therefore we conclude that (i) implies (ii).
	For the other direction, if the Postnikov tower of $t_{\le 0} \circ \cF$ converges, so does the Postnikov tower of $t_{\le n} \circ \cF \simeq t_{\le 0} \circ (\cF[n])$ for every non-negative integer $n$.
	Let us denote by $\cF^\wedge$ the limit of the Postnikov tower of $\cF$.
	Since $t_{\le n}$ commutes with limits for every non-negative integer $n$, it follows that for each $U \in \cC$ the natural map $\cF(U) \to \cF^\wedge(U)$ induces equivalences
	\[ t_{\le n}(\cF(U)) \to t_{\le n}(\cF^\wedge(U)) \]
	for every non-negative integer $n$.
	We conclude that the map $\cF(U) \to \cF^\wedge(U)$ is an equivalence and therefore (ii) implies (i).
	
	Now we turn to the equivalence between (ii) and (iv).
	It suffices to prove that the hypercompleteness for an $\cS$-valued sheaf is equivalent to the convergence of its Postnikov tower.
	Let $G \in \Sh(\cC, \tau)$.
	If $G$ is hypercomplete, then it belongs to the hypercompletion $\Sh(\cC, \tau)^\wedge$, and so does every $\tau_{\le n}(G)$. Since $\Sh(\cC,\tau)^\wedge$ is hypercomplete, we see that
	\[ G \simeq \lim \tau_{\le n}(G) \]
	is in $\Sh(\cC, \tau)^\wedge$.
	Then it is enough to note that $\Sh(\cC, \tau)^\wedge$ is a localization of $\Sh(\cC, \tau)$, and therefore the inclusion functor preserves limits.
		For the other direction, if the Postnikov tower of $G$ converges, then $G$ is the limit of hypercomplete objects and it is therefore hypercomplete itself.
	
	Finally we prove the equivalence between (iii) and (iv).
	Since $t_{\le 0}$ is a right adjoint, we see immediately that (iii) implies (iv).
	Now assume that (iv) holds.
	Then the Postnikov tower of $t_{\le 0} \circ \cF$ converges and therefore the Postnikov tower of $t_{\le 0} \circ (\cF[n])$ converges for every non-negative integer $n$.
	It follows that each $t_{\le 0} \circ (\cF[n])$ takes hypercoverings to limit diagrams.
	Let $U^\bullet \to U$ be a hypercovering in $\cC$.
	We want to prove that the canonical map
	\[ \cF(U) \to \lim \cF(U^\bullet) \]
	is an equivalence.
	Since $t_{\le n}$ is a right adjoint, the hypothesis implies that each induced map
	\[ t_{\le n} \cF(U) \to t_{\le n} \left( \lim \cF(U^\bullet) \right) \simeq \lim t_{\le n} \cF(U^\bullet) \]
	is an equivalence, completing the proof.
\end{proof}

\begin{cor}\label{cor:hypercompleteness_of_bounded_below}
Every sheaf in $\Sh^{\ge 0}_{\DAb}(X)$ is hypercomplete.
\end{cor}
\begin{proof}
It follows from the implication (i)$\Rightarrow$(iii) in \cref{prop:hypercompleteness_DAb}.
\end{proof}

From now on, we restrict to one of the following geometric contexts: $(\mathrm{Aff},\tauet,\bPsm)$, $(\Stn_\C, \tauan, \bPsm)$ and $(\Afd_k, \tauqet,\bPqsm)$.

The site $(\cC,\tau)$ has a structure sheaf of ordinary rings $\cO_\cC$ defined by $\cO_\cC(S)=\Gamma(\cO_S)$ for every $S\in\cC$.

\begin{lem}\label{lem:structure_sheaf}
	Via composition with the inclusion $\Ab\to\DAb$, the sheaf $\cO_\cC$ is a hypercomplete sheaf in the heart $\Sh_{\DAb}^\heartsuit(\cC,\tau)$.
\end{lem}
\begin{proof}
	Since the restriction of $\cO_\cC$ to any objects in the category $\cC$ is acyclic, we see that via composition with the inclusion $\Ab\to\DAb$, $\cO_\cC$ belongs to $\Sh_{\DAb}(\cC,\tau)$.
		It lies obviously in the heart $\Sh_{\DAb}^\heartsuit(\cC,\tau)$.
	Therefore, it is hypercomplete by \cref{prop:descent_vs_hyperdescent}.
\end{proof}

Let $X$ be a geometric stack.
Composing with the forgetful functor $\CXP\to\cC$, we obtain from $\cO_\cC$ a structure sheaf $\cO_X$, which is a hypercomplete sheaf in the heart $\Sh_{\DAb}^\heartsuit(X)$.
Regarding $\cO_X$ as a commutative algebra object in $\Sh_{\DAb}(X)^\wedge$,  we define the derived \infcat $\cO_X\Mod$ as the \infcat of left $\cO_X$-module objects of $\Sh_{\DAb}(X)^\wedge$ (\cite[4.2.1.13]{Lurie_Higher_algebra}). It follows from \cite[3.4.4.2]{Lurie_Higher_algebra} that $\cO_X \Mod$ is a presentable \infcat.

\begin{rem}[{\cite[2.1.3]{DAG-VIII}}]
The \infcat $\cO_X\Mod$ is endowed with a t-structure and a symmetric monoidal structure induced by $\Sh_{\DAb}(X)$, which are compatible with each other.
We denote by $\cO_X\Modh$ the heart.
\end{rem}

\begin{rem}
Assume that $X$ is a representable stack.
Then $\cO_X\Modh$ coincides with the category of $\cO_X$-modules over the lisse-étale site of $X$ (cf.\ \cite[Proposition 2.1.3 and Remark 2.1.5]{DAG-VIII}).
Since $X$ is a representable stack, the lisse-étale site of $X$ is a $1$-category.
In particular, the associated $\infty$-topos is $1$-localic.
Thus, the hypotheses of \cite[Proposition 2.1.8]{DAG-VIII} are satisfied.
This allows us to identify the full subcategory of $\cO_X\Mod$ spanned by left-bounded objects with the derived $\infty$-category $\cD^+(\cO_X\Modh)$.
Since all the objects of $\cO_X\Mod$ are hypercomplete by definition, \cref{prop:hypercompleteness_DAb}  shows that the induced $t$-structure on $\cO_X\Mod$ is left $t$-complete.
Therefore, we obtain an equivalence $\cO_X\Mod \simeq \cD(\cO_X\Modh)$.
This last statement is false for a general geometric stack.
\end{rem}

For the functoriality of sheaves on geometric stacks, it is useful to introduce another \infsite.
Let $\GeomXP$ denote the full subcategory of the overcategory $\Sh(\cC,\tau)^\wedge_{/X}$ spanned by morphisms from geometric stacks to $X$ which are in $n$-$\bP$ for some $n$.
We consider the topology on $\GeomXP$ generated by coverings of the form $\{U_i/X\to U/X\}_{i\in I}$ such that every morphism $U_i\to U$ is in $n$-$\bP$ for some $n$ and that the morphism $\coprod U_i\to U$ is an effective epimorphism.
By an abuse of notation, we denote this topology again by $\bP$.
So we obtain an \infsite $(\GeomXP,\bP)$.

\begin{lem} \label{lem:small_sites}
Let $u\colon \CXP\to\GeomXP$ denote the inclusion functor.
For any presentable \infcat $\cT$, the functors $\Tulowersw$ and $\big(\Tuuppers\big)^\wedge$ introduced in \cref{sec:functoriality}
\[\Tulowersw \colon \Sh_\cT(\CXP,\tau)^\wedge \leftrightarrows\Sh_\cT(\GeomXP,\bP)^\wedge \colon \big(\Tuuppers\big)^\wedge \]
are equivalences of \infcats.
\end{lem}
\begin{proof}
By \cite[1.1.12]{DAG-V}, it suffices to prove the statement for $\cT=\cS$.
We note that for each of the geometric contexts $(\mathrm{Aff},\tauet,\bPsm)$, $(\Stn_\C, \tauan, \bPsm)$ and $(\Afd_k, \tauqet,\bPqsm)$, surjective morphisms in $\bP$ have sections locally with respect to the topology $\tau$.
Therefore, we conclude by \cref{prop:equivalence_of_topoi}.
\end{proof}

Let $f\colon X\to Y$ be a morphism of geometric stacks.
It induces a continuous functor
\[ v\colon \GeomYP\to\GeomXP,\qquad U\mapsto U\times_Y X.\]
The functor $v$ commutes with pullbacks, so by Lemmas \ref{lem:adjunction_u_s_u^s} and \ref{lem:small_sites}, we obtain a pair of adjoint functors $(\tensor*[^{\DAb}]{v}{_s^\wedge},\big(\tensor*[^{\DAb}]{v}{^s}\big)^\wedge)$, which we denote for simplicity as
\[\Dfpull\colon\Sh_{\DAb}(Y)^\wedge \leftrightarrows\Sh_{\DAb}(X)^\wedge \colon\Dfpush.\]

Via the natural map $\cO_Y\to\Dfpush\cO_X$, the functor $\Dfpush$ induces a composite functor
\[\cO_X\Mod\to\Dfpush\cO_X\Mod\to\cO_Y\Mod,\]
which we denote by
\[\rR f_*\colon\cO_X\Mod\longrightarrow\cO_Y\Mod.\]

By base change along the natural map $\Dfpull\cO_Y\to\cO_X$, the functor $\Dfpull$ induces a composite functor
\[\cO_Y\Mod\to\Dfpull\cO_Y\Mod\to\cO_X\Mod,\]
which we denote by
\[\rL f^*\colon\cO_Y\Mod\longrightarrow\cO_X\Mod.\]

The notations $\rL f^*$, $\rR f_*$ are chosen in accordance with the classical terminology.
We denote $\rL^i f^* \coloneqq \rH^i\circ\rL f^*$, $\rR^i f_* \coloneqq \rH^i\circ\rR f_*$ for every $i\in\Z$.

\begin{prop} \label{prop:adjunction_Lfpull_Rfpush}
The functor $\rL f^*$ is left adjoint to the functor $\rR f_*$.
\end{prop}
\begin{proof}
	Observe that the map $\cO_Y \to \Dfpush \cO_X$ factors canonically as
	\[ \cO_Y \to \Dfpush \Dfpull \cO_Y \to \Dfpush \cO_X .  \]
	Therefore, $\rR f_*$ can be factored as the composition
	\[ \cO_X \Mod \to \Dfpush \cO_X \Mod \to \Dfpush \Dfpull \cO_Y \Mod \to \cO_Y \Mod . \]
	Observe now that there is a commutative diagram
	\begin{equation} \label{eq:adjunction_Lfpull_Rfpush}
		\begin{tikzcd}
			{} & \Dfpull \cO_Y \Mod \arrow{dl} \arrow{d} & \cO_X \Mod \arrow{d} \arrow{l} \\
			\cO_Y \Mod & \Dfpush \Dfpull \cO_Y \Mod \arrow{l} & \Dfpush \cO_X \Mod \arrow{l} .
		\end{tikzcd}
	\end{equation}
	The horizontal functors are forgetful functors and therefore they have left adjoints.
	We claim that the functors $\Dfpull \cO_Y \Mod \to \Dfpush \Dfpull \cO_Y \Mod$ and $\cO_X \Mod \to \Dfpush \cO_X \Mod$ also have left adjoints.
	Up to replacing $\cO_X$ with $\Dfpull \cO_Y$, we see that it is enough to prove the existence of left adjoint in the second case.
	
	Using \cite[4.3.3.2, 4.3.3.13]{Lurie_Higher_algebra}, the forgetful functors
	\[
	\cO_X \Mod \to \Sh_{\DAb}(X)^\wedge, \quad \Dfpush \cO_X \Mod \to \Sh_{\DAb}(Y)^\wedge
	\]
	admit left adjoints and are conservative.
	Moreover, \cite[4.2.3.5]{Lurie_Higher_algebra} shows that they commute with arbitrary limits and colimits.
	Consider the diagram
	\[ \begin{tikzcd}
	\cO_X \Mod \arrow{r} \arrow{d} & \Sh_{\DAb}(X)^\wedge \arrow{d}{\Dfpush} \\
	\Dfpush \cO_X \Mod \arrow{r} & \Sh_{\DAb}(Y)^\wedge
	\end{tikzcd} \]
	and observe that it is commutative by definition of the functor $\cO_X \Mod \to \Dfpush \cO_X \Mod$.
	Furthermore, $\Dfpush$ is by definition a morphism in $\RPr$, the $\infty$-category of presentable $\infty$-categories with morphisms given by right adjoints.
	In particular, there exists a regular cardinal $\kappa \gg 0$ such that $\Dfpush$ commutes with $\kappa$-filtered colimits.
	These observations imply that the induced functor $\cO_X \Mod \to \Dfpush \cO_X \Mod$ commutes with limits and $\kappa$-filtered colimits.
	It follows from the adjoint functor theorem that this functor admits a left adjoint.
	
	Passing to left adjoints in the diagram \eqref{eq:adjunction_Lfpull_Rfpush}, we obtain another commutative diagram
	\[ \begin{tikzcd}
	{} & \Dfpull \cO_Y \Mod \arrow{r} & \cO_X \Mod \\
	\cO_Y \Mod \arrow{r} \arrow{ur} & \Dfpush \Dfpull \cO_Y \Mod \arrow{r} \arrow{u} & \Dfpush \cO_X \Mod . \arrow{u}
	\end{tikzcd} \]
	To complete the proof, it is now enough to observe that $\rR f^*$ is by definition the composition
	\[ \cO_Y \Mod \to \Dfpull \cO_Y \Mod \to \cO_X \Mod , \]
	and that the composition
	\[ \cO_Y \Mod \to \Dfpush \Dfpull \cO_Y \Mod \to \Dfpush \cO_X \to \cO_X \Mod \]
	is by construction a left adjoint of $\rR f_*$.
\end{proof}

\begin{rem} \label{rem:t_exactness_Dfpush}
By \cref{prop:adjunction_Lfpull_Rfpush}, the functors $\rL f^*$ and $\rR f_*$ are exact functors between stable \infcats.
Moreover, concerning t-structures, the functor $\rR f_*$ is left t-exact, while the functor $\rL f^*$ is right t-exact.
Indeed, the functor $\Dfpush$ is left t-exact by construction, and therefore the functor $\rR f_*$ is also left t-exact.
It follows by adjunction that the functor $\rL f^*$ is right t-exact.
\end{rem}

For the purpose of cohomological descent, we will consider (augmented) simplicial geometric stacks,
that is, (augmented) simplicial objects in the \infcat of geometric stacks.

Let $f^\bullet \colon X^\bullet \to X$ be an augmented simplicial (analytic or algebraic) stack.
The functors $\rL f^*$ induce a functor
\[
\rL f^{\bullet *} \colon \cO_X \Mod \to \varprojlim \cO_{X^\bullet} \Mod
\]
where the limit is taken in the $\infty$-category $\LPr$ of presentable $\infty$-categories with morphisms given by left adjoints.
It admits a right adjoint which we denote by
\[
\rR f^\bullet_* \colon \varprojlim \cO_{X^\bullet} \Mod \to \cO_X \Mod.
\]
We refer to \cref{sec:spectral_sequence} for a detailed discussion on the functor $\rR f^\bullet_*$.

We denote $\rL^i f^{\bullet*} \coloneqq \rH^i\circ \rL f^{\bullet*}$, $\rR^i f^\bullet_* \coloneqq \rH^i\circ \rR f^\bullet_*$ for every $i\in\Z$.

\begin{defin} \label{def:coherent_sheaf_on_stacks}
Let $X$ be either an algebraic stack or an analytic stack.
We denote by $\Coh(X)$ the full subcategory of $\cO_X\Mod$ spanned by $\cF\in\cO_X\Mod$ for which there exists an atlas $\{\pi_i\colon U_i\to X\}_{i\in I}$ such that for every $i\in I$, $j\in\Z$, the $\cO_{U_i}$-modules $\rL^j\pi_i^*(\cF)$ are coherent sheaves.

We denote by $\Cohh$ (resp.\ $\Cohb(X)$, $\Coh^+(X)$, $\Coh^-(X)$) the full subcategory of $\Coh(X)$ spanned by objects cohomologically concentrated in degree 0 (resp.\ locally cohomologically bounded, bounded below, bounded above).
\end{defin}

\subsection{Coherence of derived direct images for algebraic stacks}

\begin{lem}[devissage] \label{lem:devissage}
Let $\cT$ be a stable $\infty$-category equipped with a $t$-structure.
Let $\cA_0$ be a full subcategory of the heart $\cT^\heartsuit$.
Let $\cT_0^+$ be the full subcategory of $\cT$ spanned by connective objects whose cohomologies are in $\cA_0$.
Let $\cT_0^b$ be the full subcategory of $\cT_0^+$ consisting of objects which are also coconnective.
Let $\cK$ be a full subcategory of $\cT$ containing $\cA_0$ which is closed under equivalences, loops, suspensions and extensions, then $\cK$ contains $\cT_0^b$.
Moreover, assume that for any object $\cF\in\cT$ such that $\tau_{\le n} \cF\in\cK$ for every $n\ge 0$, we have $\cF\in\cK$.
Then $\cK$ contains $\cT_0^+$.
\end{lem}

\begin{proof}
Let $\cF \in \cT_0^b$.
If $\cF$ is concentrated in one degree, then $\cF \in \cK$, because $\cK$ is closed under loops and suspensions.
In general, let $i$ be the biggest index such that $\mathrm H^{i}(\cF) \ne 0$.
We have a fiber sequence
\[
\tau_{\le i-1} \cF \to \cF \to \tau_{\ge i} \cF,
\]
where $\tau_{\ge i} F$ is concentrated in one degree.
Since $\cK$ is closed under extensions, it follows from induction that $\cF\in\cK$.
The last statement of the lemma follows from the definition of $\cT_0^+$.
\end{proof}

\begin{thm} \label{thm:proper_direct_image_alg}
Let $f \colon X \to Y$ be a proper morphism of locally noetherian algebraic stacks.
The derived pushforward functor
\[
\rR f_* \colon \cO_X\Mod \longrightarrow \cO_Y\Mod
\]
sends the full subcategory $\Coh^+(X)$ to the full subcategory $\Coh^+(Y)$.
\end{thm}

\begin{proof}
The question being local on the target, we can assume that $Y$ is representable.
Moreover we can assume that there exists a scheme $P$ proper over $Y$ and a surjective $Y$-morphism $p\colon P\to X$.

We proceed by induction on the geometric level $n$ of the stack $X$.
The case $n=-1$ is classical.
Assume that the statement holds when $X$ is $k$-geometric for $k<n$.
Let us prove the case when $X$ is $n$-geometric.

By noetherian induction, we can assume that the statement holds for any closed substack of $X$ not equal to $X$.
Let $\cK$ denote the full subcategory of $\Coh^+(X)$ spanned by the objects whose image under $\rR f_*$ belongs to $\Coh^+(Y)$.
Since $\rR f_*$ is an exact functor of stable \infcats, the subcategory $\cK$ is closed under equivalences, loops, suspensions and extensions.
Moreover, since $\rR f_*$ is left t-exact, the subcategory $\cK$ verifies the last condition of \cref{lem:devissage}.
Therefore, by \cref{lem:devissage}, it suffices to prove that $\cF\in K$ for any $\cF\in\Cohh(X)$.

Let $\cJ$ be the nilradical ideal sheaf of $X$.
Let $\cF\in\Cohh(X)$.
It is killed by a power $\cJ^m$ for some $m$.
For $1\le l\le m$, we have a short exact sequence
\[0\to \cJ^{l-1}\cF/\cJ^l\cF\to \cF/\cJ^l\cF\to \cF/\cJ^{l-1}\cF\to 0.\]
By induction on $l$, in order to prove that $\cF\in\cK$, it suffices to prove that $\cJ\cF\in\cK$.
In other words, we can assume that $X$ is reduced.

Let $\cF' \coloneqq \rR p^\bullet_* \;\rL^0 p^{\bullet *} \cF$.
\cref{thm:generic_flatness} implies the existence of a dense open substack $U$ of $X$ such that the pullback $P \times_X U \to U$ is flat.
By cohomological descent \cite[1.3.7.2]{HAG-II}, the natural morphism $\cF\to\cF'$ is an equivalence over $U$.
The spectral sequence of \cref{thm:spectral_sequence} reads off as
\[
\rR^t p^s_* \;\rL^0 p^{s,*} \cF \Rightarrow \rR^{t+s} p^\bullet_* \;\rL^0 p^{\bullet *} \cF.
\]
The induction hypothesis on the geometric level $n$ shows that each $\rR^t p^s_* \;\rL^0 p^{s,*} \cF$ is coherent.
It follows that $\cF' \in \Coh^+(X)$.

Let $\cG$ be the fiber of the morphism $\cF\to\cF'$.
We deduce that $\cG \in \Coh^+(X)$.
The noetherian induction hypothesis implies that $\cG \in \cK$.
Let $g^\bullet$ denote the induced morphism $P^\bullet \to Y$.
Once again, the induction hypothesis on the geometric level, together with the spectral sequence of \cref{thm:spectral_sequence}, shows that $\rR g^\bullet_* \;\rL^0 p^{\bullet *} \cF \in \Coh^+(Y)$.
Since $\rR g^\bullet_* = \rR f_* \circ \rR p^\bullet_*$, we deduce that $\cF'=\rR p^\bullet_* \;\rL^0 p^{\bullet *} \cF \in\cK$.
We conclude that $\cF\in\cK$, completing the proof.
\end{proof}

\subsection{Coherence of derived direct images for analytic stacks} \label{sec:coherence_of_direct_images}

\begin{lem} \label{lem:double_hypercover}
Let $f \colon X \to Y$ be a proper morphism of analytic stacks, with $Y$ representable.
Assume moreover that the identity $X \to X$ defines a weakly proper pair over $Y$.
Then there exists two hypercoverings $\mathcal U^\bullet$ and $\mathcal V^\bullet$ of $X$ such that $\mathcal U^n \Subset_Y \mathcal V^n$ for every $n\ge 0$.
\end{lem}
\begin{proof}
We construct the two hypercoverings by successive refinements.
Since the identity of $X$ is a weakly proper pair over $Y$, we can choose a finite double atlas $\{U_i^0 \Subset_Y V_i^0\}_{i \in I_0}$ of $X$.
Denote by $U^0$ (resp.\ $V^0$) the disjoint union of $U_i^0$ (resp.\ $V_i^0$) over $i\in I_0$.
We have $U^0 \Subset_Y V^0$.
Since proper morphisms are separated, \cref{prop:product_of_wpp} shows that $U^0 \times_{X} U^0 \to V^0 \times_{X} V^0$ is a weakly proper pair over $Y$.
Choose finite atlases $\{U_i^1\}_{i \in I_1}$ and $\{V_i^1\}_{i \in I_1}$ of $U^0 \times_{X} U^0$ and $V^0 \times_{X} V^0$ respectively such that $U_i^1 \Subset_Y V_i^1$.
Set
\begin{gather*}
U^1 \coloneqq U^0 \sqcup \coprod_{i \in I_1} U_i^1, \\
V^1 \coloneqq V^0 \sqcup \coprod_{i \in I_1} V_i^1.
\end{gather*}
We have $U^1 \Subset_Y V^1$.
We define a morphism $u^1 \colon U^1 \to U^0 \times_{X} U^0$ (resp.\ $v^1 \colon V^1 \to V^0 \times_{X} V^0$) by taking the disjoint union of the atlas map with the diagonal embedding $U^0 \to U^0 \times_{X} U^0$ (resp.\ $V^0 \to V^0 \times_{X} V^0$).
We define the face maps $U^1 \rightrightarrows U^0$ (resp.\ $V^1 \rightrightarrows V^0$) to be the composition of the face maps $U^0\times_X U^0\rightrightarrows U^0$ (resp.\ $V^0\times_X V^0\rightrightarrows V^0$) with the map $u^1$ (resp.\ $v^1$).

Suppose now that we have already built two $n$-truncated augmented simplicial objects $\mathcal U^\bullet_{\le n}$ and $\mathcal V^\bullet_{\le n}$ such that $\mathcal U^m_{\le n} \Subset_Y \mathcal V^m_{\le n}$ for every $m \le n$.
Set $U^m \coloneqq \mathcal U^m_{\le n}$, $V^m \coloneqq \mathcal V^m_{\le n}$, and
\[
U^{n+1} \coloneqq \mathrm{cosk}_n(\mathcal U^\bullet_{\le n})^{n+1} = U^n \times_{U^{n-1}} U^n \times_{U^{n-1}} \cdots \times_{U^{n-1}} U^n .
\]
We define $V^{n+1}$ in a similar way.
Both are representable stacks.
Moreover, an iterated application of \cref{prop:product_of_wpp} shows that $U^{n+1} \Subset_Y V^{n+1}$. Proceeding by induction, we obtain the hypercovering we need.
\end{proof}

\begin{rem} \label{rem:triple_hypercover}
The same reasoning in the proof of \cref{lem:double_hypercover} shows the existence of three hypercoverings $\cU^\bullet, \cV^\bullet, \cW^\bullet$ of $X$ such that $\cU^n\Subset_Y \cV^n\Subset_Y \cW^n$ for every $n\ge 0$.
\end{rem}

\begin{prop} \label{prop:proper_direct_image_kanal_absolute}
Let $f\colon X\to Y$ be a proper morphism of \kanal stacks.
Assume $Y=\Sp A$ for a $k$-affinoid algebra $A$.
Assume moreover that the identity $X\to X$ defines a weakly proper pair over $Y$.
Let $\cF\in\Cohh(X)$.
Then the $A$-module $H^n(X,\cF)$ is of finite type for any integer $n$.
It is zero for $n<0$.
\end{prop}
\begin{proof}
We follow closely the proof of \cite[Satz 2.6]{Kiehl_Endlichkeitssatz_1967} (see also \cite{Houzel_Espaces_analytiques_rigides_1995}). 
Let $U^\bullet \Subset_Y V^\bullet$ be the double hypercovering of $X$ constructed in \cref{lem:double_hypercover}.
For every $n\ge 0$, let $B^n$, $C^n$ be the $k$-affinoid algebras corresponding to $U^n$ and $V^n$ respectively.
Since $\cF\in\Cohh(X)$, by Tate's acyclicity theorem and \cite[§2.1]{Conrad_Non-archimedean_analytification_2009}, the sections $\cF(U^n)$ and $\cF(V^n)$ are respectively $B^n$-modules and $C^n$-modules of finite type.
Since $U^n\Subset_Y V^n$, there exists a Banach $A$-module $D^n$ and an epimorphism $D^n\to\cF(V^n)$ such that the composition of
\[ D^n\twoheadrightarrow\cF(V^n)\to\cF(U^n)\]
is a nuclear map.

Let $\Cech(U^\bullet,\cF)$ and $\Cech(V^\bullet,\cF)$ denote the Čech complexes of the sheaf $\cF$ with respect to the hypercoverings $U^\bullet$ and $V^\bullet$ respectively.
We deduce that for every $n\ge 0$, there exists a Banach $A$-module $E^n$ and an epimorphism $E^n\to\Cech^n(V^\bullet,\cF)$ such that the composition of
\[E^n\twoheadrightarrow\Cech^n(V^\bullet,\cF)\to\Cech^n(U^\bullet,\cF)\]
is a nuclear map.
Since the sheaf $\cF$ is hypercomplete by definition, the Čech complexes $\Cech(U^\bullet,\cF)$ and $\Cech(V^\bullet,\cF)$ both compute the cohomology of $\cF$.

So we have an isomorphism $H^n(V^\bullet,\cF)\xrightarrow{\sim} H^n(U^\bullet,\cF)$.
Therefore, using \cite[Korollar 1.5]{Kiehl_Endlichkeitssatz_1967}, both are $A$-modules of finite type.
Since the Čech complexes vanish in negative degrees, the $A$-module $H^n(X,\cF)$ vanishes in negative degrees as well.
\end{proof}

\begin{prop} \label{prop:proper_direct_image_kanal}
Let $f\colon X\to Y$ be a proper morphism of \kanal stacks.
Let $\cF\in\Cohh(X)$.
Then $\rR f_*(\cF)\in\Coh^+(Y)$.
\end{prop}
\begin{proof}
The statement being local on $Y$, we can assume that $Y=\Sp A$ for some $k$-affinoid algebra $A$ and that the identity $X\to X$ defines a weakly proper pair over $Y$.
By induction on the dimension of $Y$ as in the proof of \cite[Satz 3.5]{Kiehl_Endlichkeitssatz_1967},
for any affinoid domain $U = \Sp B$ in $Y$, any integer $n$, we have
\[ H^n(X\times_Y U,\cF)\simeq H^n(X,\cF)\otimes_A B.\]
Combining with \cref{prop:proper_direct_image_kanal_absolute}, we have proved the statement.
\end{proof}

\begin{lem} \label{lem:absolute_double_hypercover}
Let $f \colon X \to Y$ be a proper morphism of $\mathbb C$-analytic stacks, with $Y$ representable.
Assume moreover that the identity $X\to X$ defines a weakly proper pair over $Y$.
For every point $y_0 \in Y$ and every triple of open neighborhoods $W_2 \Subset W_1 \Subset W_0$ of $y_0$, there exist hypercoverings $\mathcal U_i^\bullet$ of $X \times_Y W_i$ for $i = 0,1,2$ such that $\mathcal U_2^\bullet \Subset \mathcal U_1^\bullet \Subset \mathcal U_0^\bullet$.
\end{lem}
\begin{proof}
In virtue of \cref{lem:double_hypercover} and \cref{rem:triple_hypercover} we can choose a triple hypercovering $\mathcal V_2^\bullet \Subset_Y \mathcal V_1^\bullet \Subset_Y \mathcal V_0^\bullet$ of $X$.
Let $W_2 \Subset W_1 \Subset W_0$ be arbitrary open neighborhoods of $y_0$ in $Y$, \cref{lem:variant_wpp_stability} shows that
\begin{gather*}
\mathcal V_2^\bullet \times_Y W_2 \Subset \mathcal V_1^\bullet \times_Y W_1, \\
\mathcal V_1^\bullet \times_Y W_1 \Subset \mathcal V_0^\bullet \times_Y W_0.
\end{gather*}
Setting $\mathcal U_i^\bullet \coloneqq \mathcal V_i^\bullet \times_Y W_i$, the lemma is proven.
\end{proof}

\begin{prop} \label{prop:proper_direct_image_canal}
Let $f\colon X\to Y$ be a proper morphism of \canal stacks.
Let $\cF\in\Cohh(X)$.
Then $\rR f_*(\cF)\in\Coh^+(Y)$.
\end{prop}
\begin{proof}
The statement being local on $Y$, we can assume that $Y=S$ is representable and that the identity $X\to X$ defines a weakly proper pair over $S$.
Fix a point $s_0 \in S$, it suffices to prove that $\rR^n f_* \cF$ is coherent in a neighborhood of $s_0$.
We can assume that $S_0 \coloneqq S$ admits a closed embedding in an open subset $\Omega$ of $\mathbb C^m$ for some $m$.
Let $\rB(s_0, R_2) \Subset \rB(s_0, R_1) \Subset \Omega$ be polydisks and set $S_i \coloneqq S \cap \rB(s_0, R_i)$ for $i = 1,2$.
We have $S_1 \Subset S_0$.
Invoking \cref{lem:absolute_double_hypercover}, we obtain hypercoverings $\mathcal U_i^\bullet$ of $S_i \times_Y X$ for $i = 0,1,2$ which satisfy $\mathcal U_2^\bullet \Subset \mathcal U_1^\bullet \Subset \mathcal U_0^\bullet$.

Now fix a degree $r$ and let $U_\alpha'' \Subset U_\alpha' \Subset U_\alpha$ be connected components of $\mathcal U_2^r$, $\mathcal U_1^r$ and $\mathcal U_0^r$ respectively.
Choose a closed embedding $i_\alpha \colon U_\alpha \to \Omega_\alpha$, where $\Omega_\alpha$ is an Stein open subset of $\C^n$ for some $n$. Since $U_\alpha' \Subset U_\alpha$, we can find an open subset $\Omega_\alpha' \Subset \Omega_\alpha$ such that $U_\alpha' = \Omega_\alpha' \cap U_\alpha$.
Observe that in this way $U_\alpha'$ becomes a closed subspace of $\Omega_\alpha'$.
We can therefore find a Stein neighborhood $W$ of $U_\alpha'$ contained in $\Omega_\alpha'$.
Since the closure of such a neighborhood in $\Omega_\alpha$ is closed inside $\widebar{\Omega_\alpha'}$, we see that $W \Subset \Omega_\alpha$.
In other words, we can assume that $\Omega_\alpha'$ is a Stein neighborhood of $U_\alpha'$.
Reasoning in the same way, we find a third Stein open subset $\Omega_\alpha'' \Subset \Omega_\alpha'$ such that $U_\alpha'' = \Omega_\alpha'' \cap U_\alpha$.

Denote by $j \colon S_2 \to B(s_0, R_2)$ the given embeddings and introduce the sheaf
\[
\cF_\alpha \coloneqq (i_\alpha \times (j \circ f))_*(\cF) .
\]
Observe that $\cF_\alpha$ is coherent because $i_\alpha \times (j \circ f)$ is a closed immersion of \emph{representable} \canal stacks.
Let $D \Subset B(s_0,R_2)$ be a concentric polydisk.
Using \cite[Proposition 2]{Douady_Proper_1973} we conclude that
\[
\cF(U_\alpha'' \times_X f^{-1}(D)) = \cF_\alpha(\Omega''_\alpha \times D)
\]
is a fully transverse $\mathcal O(D)$-module,  and the same goes for
\[\cF(U_\alpha' \times_X f^{-1}(D)) = \cF_\alpha(\Omega'_\alpha \times D).\]
Moreover, the restriction map
\[
\cF(U_\alpha' \times_X f^{-1}(D)) \to \cF(U_\alpha'' \times_X f^{-1}(D))
\]
is $\mathcal O(D)$-subnuclear by \cite[Proposition 4]{Douady_Proper_1973}. For every Stein open subset $V \subset D$, we deduce from the proof of \cite[Proposition 2]{Douady_Proper_1973} that \[
\mathcal O(V) \hat{\otimes}_{\mathcal O(D)} \cF(U_\alpha \times_X f^{-1}(D)) = \cF(U_\alpha \times_X f^{-1}(V)) .
\]
It follows that the \v{C}ech complex $\check{\mathcal C}(\mathcal U_2^\bullet \times_X f^{-1}(D), \cF)$ is a fully transverse $\mathcal O(D)$-module, and the same goes for $\check{\mathcal C}(\mathcal U_1^\bullet \times_X f^{-1}(D), \cF)$.

Let $V \subset S_2$ be a Stein open subset.
Since both $V \times_S \mathcal U_2^\bullet$ and $V \times_S \mathcal U_1^\bullet$ are acyclic hypercoverings, we obtain
isomorphisms
\[
\mathrm H^\bullet (\check{\mathcal C}(V \times_S \mathcal U_1^\bullet, \cF)) \simeq \mathrm H^\bullet(V \times_S X, \cF) \simeq \mathrm H^\bullet(\check{\mathcal C}(V \times_S \mathcal U_2^\bullet, \cF)).
\]
It follows that the restriction map
\[
\cF((D \cap S) \times_S \mathcal V^k) \to \cF((D \cap S) \times_S \mathcal U^k)
\]
is a quasi-isomorphism which is $\mathcal O(D)$-subnuclear in every degree.

Set $D = D(s_0,R)$ and $D_t = D(s_0, tR)$ for $0 < t < 1$.
\cite[Théorème 2]{Douady_Proper_1973}  shows that for every integer $N$ and every $t < 1$, there is a complex of finitely generated free $\mathcal O$-modules $\mathcal L_N^\bullet$ and a $\mathcal O(D_t)$-linear quasi-isomorphism of complexes
\[
\mathcal L_N^\bullet(D_t) \to \tau_{\le N} \check{\mathcal C}((D_t \cap W) \times_Y \mathcal V^\bullet, \cF).
\]
When $V \subset D_t$ is an arbitrary Stein open subset, it follows from \cite[Proposition 3]{Douady_Proper_1973}  that
\[
\mathcal L_N^\bullet(V) \to \tau_{\le N} \check{\mathcal C}((V \cap W) \times_Y \mathcal U_2^\bullet, \cF)
\]
is a quasi-isomorphism.
At this point, we conclude along the same lines as \cite[§7]{Douady_Proper_1973}.

\end{proof}

%

\begin{thm} \label{thm:proper_direct_image_anal}
Let $f \colon X \to Y$ be a proper morphism of analytic stacks.
The derived pushforward functor
\[
\rR f_* \colon \cO_X\Mod \longrightarrow \cO_Y\Mod
\]
sends the full subcategory $\Coh^+(X)$ to the full subcategory $\Coh^+(Y)$.
\end{thm}
\begin{proof}
The statement being local on the target, we can assume that $Y$ is representable and that the identity morphism $X\to X$ defines a weakly proper pair over $Y$.
Then the theorem follows from \cref{lem:devissage}, \cref{prop:proper_direct_image_canal} and \cref{prop:proper_direct_image_kanal}.
\end{proof}

\section{Analytification functors} \label{sec:analytification}

\subsection{Analytification of algebraic stacks}

In this section, we define the analytification of algebraic stacks locally finitely presented over $A$, where $A$ is either a Stein algebra or a $k$-affinoid algebra.
We use the various geometric contexts introduced in \cref{sec:examples}.

Let $A$ be a Stein algebra, that is, the algebra of functions on a Stein \canal space $S$.
The analytification functor in \cite[§ VIII]{Hakim_Topos_1972} induces a morphism of geometric contexts
\[(-)^\mathrm{an} \colon (\Afflfp_A, \tauet,\bPsm) \to (\Stn_S, \tauet,\bPsm),\]
where $\Stn_S$ denotes the category of Stein \canal spaces over $S$.

By \cref{lem:u_s_preserves_geometric_stacks}, we obtain a complex analytification functor for sheaves
\[
(-)^\mathrm{an} \colon \Sh(\Afflfp_A, \tauet)^\wedge \to \Sh(\Stn_S, \tauet)^\wedge
\]
which preserves geometric stacks.
We obtain the absolute case by setting $A=\C$.

In the \kanal case, we let $A$ be a $k$-affinoid algebra and $S=\Sp A$.
Similarly, the non-archimedean analytification functor in \cite{Berkovich_Spectral_1990} induces a non-archimedean analytification functor for sheaves
\[
(-)^\mathrm{an} \colon \Sh(\Afflfp_A, \tauet)^\wedge \to \Sh(\An_S, \tauqet)^\wedge
\]
which preserves geometric stacks,
where $\An_S$ denotes the category of \kanal spaces over $S$.
We obtain the absolute case by setting $A=k$.

\begin{lem} \label{lem:smooth_surjective_effective}
Let $f \colon X \to Y$ be a smooth (resp.\ quasi-smooth) and surjective morphism of \canal (resp.\ \kanal) stacks.
Then $f$ is an effective epimorphism.
\end{lem}
\begin{proof}
It suffices to show that for any representable stack $S$ and any morphism $S \to Y$, there exists an \'etale (resp.\ quasi-étale) covering $\{S_i\}_{i \in I}$ of $S$ and factorizations of $S_i \to Y$ through $X$.
This follows from the existence of étale (resp.\ quasi-étale) sections of smooth (resp.\ quasi-smooth) and surjective morphisms of \canal spaces (resp.\ strictly \kanal spaces).
\end{proof}

\begin{lem} \label{lem:characterising_proper_morphisms}
Let $f \colon X \to Y$ be a morphism of \canal stacks, with $Y$ representable.
Then $f$ is weakly proper if and only if for every Stein open subset $W \Subset Y$ and every atlas $\{U_i\}_{i \in I}$ of $X$ there exists a finite subset $I' \subset I$ such that $\{W \times_Y U_i\}_{i \in I'}$ is an atlas for $W \times_Y X$.
\end{lem}

\begin{proof}
First we assume that $f$ is weakly proper. Since $W$ is relatively compact in $Y$ we can find a finite family $\{Y_j \to Y\}_{j \in J}$ of smooth morphisms such that $W$ is contained in the union of the images of $Y_j$ in $Y$ and that every identity map $X \times_Y Y_j \to X \times_Y Y_j$ defines a weakly proper pair over $Y_j$.
Let $W'$ be the union of the images of $Y_j$ in $Y$, then $W \Subset W'$.
We claim that the identity map $X \times_Y W' \to X \times_Y W'$ also defines a weakly proper pair over $W'$.
Indeed, let us denote by $\{V_{jk} \Subset_{Y_j} V_{jk}'\}_{k \in K}$ the finite double atlas of $X \times_Y Y_j$ in the definition of weakly proper pair. Then the families $\{V_{jk}\}_{j,k}$ and $\{V_{jk}'\}_{j,k}$ form atlases of $X \times_Y W'$ and they are finite, so that the claim is proved.

We can therefore assume that $Y = W'$ and that the identity map $X \to X$ is a weakly proper pair from the very beginning. Let $\{V_j \Subset_Y V_j'\}_{j \in J}$ be the finite double atlas in the definition of weakly proper pair.
\cref{lem:shrink_wpp} shows that
\[
V_j \times_Y W \Subset V_j'.
\]
If $\{U_i\}_{i \in I}$ is any atlas of $X$, up to refining it we can assume that every $U_i \to X$ factors through $V_j'$ for some $j$.
Let us denote by $I_j$ the subset of $I$ consisting of indices $i$ for which $U_i$ factors precisely through $V_j'$.
We conclude that there always exists a finite subset $I_j'$ of $I_j$ such that $V_j \times_Y W$ is contained in the union of the images of $U_i \to V_j'$. Taking $I'$ to be the union of all the $I_j$, we conclude the proof of this implication.

For the converse, it suffices to exhibit a smooth atlas $\{Y_i\}_{i \in I}$ of $Y$ such that each identity map $X \times_Y Y_i \to X \times_Y Y_i$ is a weakly proper pair.
We can build in fact an analytic atlas of $Y$ with the desired property.
Indeed, let $y \in Y$ be any point and let $W \Subset Y$ be a Stein open neighborhood of $y$.
Let $\{V_j \Subset V_j'\}_{j \in J}$ be a double atlas of $X$.
By hypothesis, we can find a finite subset $J' \subset J$ such that $\{V_j \times_Y W\}_{j \in J'}$ is an atlas for $X \times_Y W$. Then the same will be true for $\{V_j' \times_Y W\}_{j \in J'}$. Moreover, \cref{lem:wpp_base_change} shows that $V_j \times_Y W \Subset_W V_j' \times_Y W$, thus completing the proof.
\end{proof}

\begin{prop}
Let $S$ be a Stein \canal space and let $A = \Gamma(\mathcal O_S)$.
Let $f \colon X \to Y$ be a proper morphism of algebraic stacks relative to $A$ in $\Sh(\Afflfp_A,\tauet)^\wedge$.
Then the analytification $f^\mathrm{an}\colon X\an \to Y\an$ is a proper morphism of \canal stacks over $S$.
\end{prop}

\begin{proof}
Arguing by induction on the geometric level, it suffices to prove that $f\an$ is weakly proper.
Since the question is local on $Y\an$, we can assume that there exists a scheme $P$ proper over $Y$ and a surjective $Y$-morphism $p\colon P\to X$.
Let $T$ be any Stein space and let $T\to Y\an$ be any smooth morphism.
Let us prove that the induced morphism $X\an \times_{Y\an} T \to T$ is weakly proper. Let $B \coloneqq \Gamma(\mathcal O_T)$, and consider the proper morphism of algebraic stacks relative to $B$
\[
X \times_Y \Spec(B) \to \Spec(B)
\]
The analytification relative to $B$ of this morphisms coincides with the analytification relative to $A$.
So we can reduce to the case $Y = \Spec(A)$.

Since $f\circ p\colon P\to Y$ is proper, \cite[Proposition 2.6.(iii)]{Hakim_Topos_1972} shows that $(f \circ p)^\mathrm{an}$ is a proper morphism of complex analytic spaces.
Now we use the equivalent formulation given in \cref{lem:characterising_proper_morphisms} to show that $f^\mathrm{an}$ is a weakly proper morphism. Let $W \Subset Y$ be an open subset and let $\{U_i\}_{i \in I}$ be any smooth atlas for $X^{\mathrm{an}}$. If we base change to $P$ we obtain an effective epimorphism
\[
\coprod_{i \in I} P^\mathrm{an} \times_X U_i \to P^\mathrm{an}.
\]
If $\{V_{ij}\}_{j \in J_i}$ is a smooth atlas for $P^\mathrm{an} \times_X U_i$, using the properness of $(f \circ p)^{\mathrm{an}}$, we can find a finite subset $J'$ of $\bigcup_{i \in I} J_i$ such that $\{V_{ij}\}_{j \in J'}$ is an atlas for $W \times_{Y^\mathrm{an}} P^\mathrm{an}$.
This finite subset $J'$ induces a finite subset $I'$ of $I$ having the property that
\[
\coprod_{i \in I'} W \times_{Y^\mathrm{an}} P^\mathrm{an} \times_X U_i \to W \times_{Y^\mathrm{an}} P^\mathrm{an}
\]
is an effective epimorphism.
It follows that the morphism
\[
\mathcal U \coloneqq \coprod_{i \in I'} W \times_{Y^\mathrm{an}} U_i \to W \times_{Y^\mathrm{an}} X^\mathrm{an}
\]
is a smooth and surjective.
So we conclude by Lemmas \ref{lem:smooth_surjective_effective} and \ref{lem:characterising_proper_morphisms}.
\end{proof}

\begin{prop}\label{prop:properness_in_kanal_analytification}
Let $A$ be a $k$-affinoid algebra.
Let $f\colon X\to Y$ be a proper morphism of algebraic stacks relative to $A$ in $\Sh(\Afflfp_A,\tauet)^\wedge$.
Then the analytification $f\an\colon X\an\to Y\an$ is a proper morphism of \kanal stacks over $\Sp A$.
\end{prop}
\begin{proof}
We denote $S=\Spec A$ and $S\an=\Sp A$.
By induction on the geometric level, it suffices to prove that $f\an$ is weakly proper.
The statement being local on $Y$, without loss of generality, we can assume that $Y=S$ and that there exists a scheme $P$ proper over $S$ and a surjective $S$-morphism $p\colon P\to X$.
Then $P\an\to X\an$ is surjective and $P\an\to S\an$ is proper.
In particular, $P\an$ is compact as a topological space.

Let $\{U_i\}_{i\in I}$ be an atlas for $X$.
Put $U \coloneqq \coprod U_i$.
Since the morphism $U\to X$ is smooth, the analytification $U\an\to X\an$ is also smooth, in particular boundaryless.
Thus for any point $u\in U\an$, there exists two affinoid neighborhoods $V_u$ and $W_u$ of $u$ in $U\an$ such that $V_u\Subset_{S\an} W_u$.

Let $\{U'_i\}_{i\in I'}$ be an atlas for $U\times_P X$.
Put $U' \coloneqq \coprod U'_i$.
For every point $j\in P\an$, choose a point $x(j)\in (U')\an$ which projects to $j$.
Let $\widebar x (j)$ denote the image of $x(j)$ under the composition $(U')\an\to U\an\times_{P\an} X\an\to U\an$.
Let $V_{x(j)} \coloneqq V_{\widebar x(j)}\times_{U\an} (U')\an$ and let $V'_j$ be the image of $V_{x(j)}$ under the morphism $(U')\an\to P\an$.
Since $V_{\widebar x(j)}$ is a neighborhood of $\widebar x(j)$ in $U\an$, $V_{x(j)}$ is a neighborhood of $x(j)$ in $(U')\an$.
Since $(U')\an\to P\an$ is smooth, $V'_j$ is a neighborhood of $j$ in $P\an$.
By the compactness of $P\an$, there exists a finite set of points $J\subset P\an$ such that $\coprod_{j\in J} V'_j$ covers $P\an$.
By \cref{lem:smooth_surjective_effective}, $\{V_{\widebar x(j)}\}_{j\in J}$ and $\{W_{\widebar x(j)}\}_{j\in J}$ are two atlases for the stack $X\an$.
So we have proved that $X\an$ is weakly proper over $S\an$.
\end{proof}

\subsection{Analytification of coherent sheaves} \label{sec:analytification_of_sheaves}

Let $A$ be a Stein algebra or a $k$-affinoid algebra.
Let $X$ be an algebraic stack locally finitely presented over $A$.

Let $(\Afflfp_A,\tauet,\bPsm)$, $(\An_\C,\tauet,\bPsm)$ and $(\An_k,\tauqet,\bPqsm)$ be as in \cref{sec:examples}.
Let $\big({((\Afflfp_A)_{/X})}_{\bPsm},\tauet\big)$ and $((\Geom_{/X})_{\bPsm},\bPsm)$  be the \infsites introduced in \cref{sec:sheaves_on_stacks} with respect to the geometric context $(\Afflfp_A,\tauet,\bPsm)$.

For the analytification $X\an$ of $X$, we denote by $((\An_{/X\an})_\bP,\tauet)$ and $(\Geom_{/X\an})_{\bP},\bP)$ the corresponding \infsites similar as above.

The analytification functor induces a continuous functor
\[
u_X \colon {((\Afflfp_A)_{/X})}_{\bPsm} \rightarrow (\An_{/X\an})_\bP,
\]
which satisfies the assumption of \cref{lem:u^p_preserves_sheaves}.
By \cref{lem:adjunction_u_s_u^s}, we obtain a pair of adjoint functors $\big(\tensor*[^{\DAb}]{u}{_{X,s}^\wedge}, \big(\tensor*[^{\DAb}]{u}{_X^s}\big)^{\wedge} \big)$, which we denote for simplicity as
\[ \tensor*[^\cD]{u}{_s^\wedge} \colon\Sh_{\DAb}(X)^\wedge\leftrightarrows\Sh_{\DAb}(X\an)^\wedge\colon \big(\Duuppers\big)^{\wedge}.\]

Let $\cO_X$ and $\cO_{X\an}$ denote respectively the structure sheaves of $X$ and $X\an$.
We have a morphism
\[
\cO_X \to \big( \Duuppers \big)^{\wedge} \cO_{X\an}
\]
defined by the morphism
\[
\cO_X(U) = \cO_U(U) \to \cO_{U\an}(U\an)= \cO_{X\an}(U\an) = \big( \Duuppers \big)^{\wedge} \cO_{X\an}(U)
\]
for every representable stack $U$ and every morphism $U\to X$.
By adjunction, this corresponds to a morphism $\tensor*[^\cD]{u}{_s^\wedge} \cO_X \to \cO_{X\an}$, and therefore defines via base change a functor
\[
(-)\an \colon \cO_X \Mod \longrightarrow \cO_{X\an} \Mod.
\]
We remark that the functor above preserves coherent sheaves.

\section{GAGA theorems} \label{sec:GAGA}

\subsection{Comparison of derived direct images}

In this section, we compare algebraic derived direct images with analytic derived direct images.
We prove the analog of \cite[Theorem 1]{Serre_GAGA} for higher stacks.

Let $A$ be either the field of complex numbers or a $k$-affinoid algebra.
Let $f\colon X\to Y$ be a morphism of algebraic stacks locally finitely presented over $A$.

Using the notations in \cref{sec:analytification_of_sheaves}, the following commutative diagram
\[
\begin{tikzcd}
{((\Afflfp_A)_{/X})}_{\bPsm} \arrow{r}{u_X} & (\An_{/X\an})_\bP \\
{((\Afflfp_A)_{/Y})}_{\bPsm} \arrow{u}{v}\arrow{r}{u_Y} & (\An_{/Y\an})_\bP \arrow{u}{v\an}
\end{tikzcd}
\]
induces a canonical comparison morphism
\[(\rR f_*\cF)\an\longrightarrow \rR f\an_*\cF\an\]
for any $\cF\in\cO_X\Mod$.

\begin{thm}\label{thm:GAGA1}
Let $f\colon X\to Y$ be a proper morphism of algebraic stacks locally finitely presented over $\Spec A$, where $A$ is either the field of complex numbers or a $k$-affinoid algebra.
The canonical comparison morphism
\[ (\rR f_*\cF)\an\longrightarrow \rR f\an_*\cF\an\]
in $\Coh^+(Y\an)$ is an isomorphism for all $\cF \in\Coh^+(X)$.
\end{thm}

\begin{proof}
The question being local on the target, we can assume that $Y$ is representable.
Moreover we can assume that there exists a scheme $P$ proper over $Y$ and a surjective $Y$-morphism $p\colon P\to X$.

We proceed by induction on the geometric level $n$ of the stack $X$.
The case $n=-1$ is classical \cite{SGA1,Kopf_Eigentliche_1974}.
Assume that the statement holds when $X$ is $k$-geometric for $k<n$.
Let us prove the case when $X$ is $n$-geometric.

We use noetherian induction as in the proof of \cref{thm:proper_direct_image_alg}.
So we can assume that $X$ is reduced.
By devissage (\cref{lem:devissage}), it suffices to prove for $\cF\in\Cohh(X)$.

By \cref{thm:generic_flatness}, there exists a dense open substack $U\subset X$ over which the morphism $p\colon P\to X$ is flat.
Let $p^\bullet\colon P^\bullet\to X$ be the simplicial nerve of $p$ and put $\cF' \coloneqq \rR p^\bullet_*\;\rL^0 p^{\bullet *}(\cF)$.
By \cref{thm:proper_direct_image_alg}, we have $\cF'\in\Coh^+(X)$.

Cohomological descent implies that the canonical morphism $\cF\to\cF'$ restricts to an isomorphism over $U$.
Therefore, by the noetherian induction hypothesis, it suffices to prove that the theorem holds for $\cF'$.

By the induction hypothesis on the geometric level and the spectral sequence of \cref{thm:spectral_sequence}, we have an isomorphism
\[(\rR p^\bullet_*\;\rL^0 p^{\bullet *} \cF)\an\xrightarrow{\ \sim\ }\rR p^{\bullet \text{an}}_* (\rL^0 p^{\bullet *}\cF)\an.\]
So we have isomorphisms
\begin{multline}\label{eq:GAGA1_eq1}
\rR(f\circ p^\bullet)\an_* (\rL^0 p^{\bullet *}\cF)\an\simeq \rR (f\an\circ p^{\bullet \text{an}})_* (\rL^0 p^{\bullet *}\cF)\an \\
\simeq \rR f\an_* \;\rR p^{\bullet \text{an}}_*(\rL^0 p^{\bullet *}\cF)\an\simeq \rR f\an_* (\rR p^\bullet_* \;\rL^0 p^{\bullet *}\cF)\an.
\end{multline}
By the induction hypothesis on the geometric level and \cref{thm:spectral_sequence} again, we have an isomorphism
\[\big(\rR(f\circ p^\bullet)_*\; \rL^0 p^{\bullet *}\cF\big)\an\simeq  \rR(f\circ p^\bullet)\an_* (\rL^0 p^{\bullet *}\cF)\an.\]
Combining with \cref{eq:GAGA1_eq1}, we obtain isomorphisms
\begin{multline*}
\big(\rR f_*(\rR p^\bullet_* \rL^0 p^{\bullet *}\cF)\big)\an\simeq \big(\rR(f\circ p^\bullet)_* \;\rL^0 p^{\bullet *}\cF\big)\an \\
\simeq \rR(f\circ p^\bullet)\an_* (\rL^0 p^{\bullet *}\cF)\an \simeq \rR f\an_* (\rR p^\bullet_* \;\rL^0 p^{\bullet *}\cF)\an.
\end{multline*}
In other words, the theorem holds for $\cF'=\rR p^\bullet_* \;\rL^0 p^{\bullet *}(\cF)\in\Coh^+(X)$.
So we have completed the proof.
\end{proof}

\subsection{The existence theorem}

In this section, we compare algebraic coherent sheaves with analytic coherent sheaves.
We prove the analog of \cite[Theorems 2 and 3]{Serre_GAGA} for higher stacks.

\begin{lem} \label{lem:analytification_inverse_limits}
	Let $X$ be an algebraic stack locally of finite presentation over $\Spec(A)$, where $A$ is either the field of complex numbers or a $k$-affinoid algebra.
	Let $\cF \in \cO_\cX \Mod$ and assume we have
	\[ \cF \simeq \lim \tau_{\ge -n} \cF. \]
	Then the analytification functor $(-)\an$ commutes with this limit.
\end{lem}

\begin{proof}
	Since the analytification functor is $t$-exact, we see that $(\tau_{\ge n} \cF)\an \simeq \tau_{\ge n} \cF\an$.
	Therefore we have
	\[ \cF\an \simeq \lim \tau_{\ge -n} \cF\an \simeq \lim (\tau_{\ge -n} \cF)\an , \]
	completing the proof.
\end{proof}

\begin{prop}\label{prop:GAGA_fully_faithful}
Let $X$ be an algebraic stack proper over $\Spec A$, where $A$ is either the field of complex numbers or a $k$-affinoid algebra.
Let $\cF, \cG\in\Coh(X)$.
Then the natural map
\[ \Map_{\Coh(X)}(\cF,\cG) \longrightarrow \Map_{\Coh(X\an)}(\cF\an,\cG\an)\]
is an equivalence.
\end{prop}

\begin{proof}
Let us first prove the result when both $\cF$ and $\cG$ belong to $\Cohh(X)$.
In this case, it follows from \cite[Chap.\ 0, 12.3.3]{EGA3-1} that the internal Hom $\RcHom_X(\cF,\cG)$ belongs to $\Coh^+(X)$.
Moreover, we have an isomorphism for internal Hom's
\[\big(\RcHom_X(\cF,\cG)\big)\an\xrightarrow{\ \sim\ }\RcHom_{X\an}(\cF\an,\cG\an)\]
by Proposition 12.3.4 loc.\ cit.
Taking global sections and using \cref{thm:GAGA1},
we have proved the statement for the case $\cF,\cG\in\Cohh(X)$.

Recall that a stable $\mathbb Z$-linear $\infty$-category $\cC$ is canonically enriched in $\DAb$ (\cite[Examples 7.4.14, 7.4.15]{Gepner_Enriched_2013}).
For every $X \in \cC$, we denote by $\Map_{\cC}^{\DAb}(X, -) \colon \cC \to \DAb$ the induced exact functor of stable $\infty$-categories.
We have an equivalence
\[
\Map_{\cC}(X,Y) \simeq \tau_{\ge 0} \Map_{\cC}^{\DAb}(X,Y)
\]
in $\cS$.
Given $\cF, \cG \in \Coh(X)$, we denote by $\psi_{\cF, \cG}$ the natural map
\[ \Map_{\Coh(X)}^{\DAb}(\cF, \cG) \to \Map_{\Coh(X\an)}^{\DAb}(\cF\an, \cG\an) . \]

We now deal with the bounded case.
Fix $\cF \in \Cohb(X)$ and let $\mathcal K_\cF$ be the full subcategory of $\Cohb(X)$ spanned by those $\cG \in \Cohb(X)$ such that $\psi_{\cF, \cG}$ is an equivalence in $\DAb$.
It suffices to show that $\mathcal K_\cF = \Cohb(X)$.
Observe that $\mathcal K_\cF$ is closed under extensions because $\Map_{\Cohb(X)}^{\DAb}(\cF,-) \colon\Cohb(X)\to \DAb$ preserves fiber sequences.
Since $\mathcal K_\cF$ is also closed under loops and suspensions, by \cref{lem:devissage}, it suffices to show that $\cK_\cF$ contains $\Cohh(X)$.
Let $\cG \in \Cohh(X)$ and let $\mathcal K_{\cG}$ be the full subcategory of $\Cohb(X)$ spanned by those $\mathcal H \in \Cohb(X)$ such that the natural map
\[
\Map_{\Cohb(X)}^{\DAb}(\mathcal H, \cG) \to \Map_{\Cohb(X\an)}^{\DAb}(\cF\an, \cG\an)
\]
is an equivalence in $\DAb$.
By \cref{lem:devissage} again, it suffices to show that $\cK_\cG$ contains $\Cohh(X)$, which is the result of the first paragraph of the proof.

We now turn to the general case. Since $\DAb$ is both left and right $t$-complete, it is enough to prove that for every integer $n \in \mathbb Z$, $\pi_{-n} \psi_{\cF, \cG}$ is an isomorphism of abelian groups.
Since
\begin{gather*}
	\pi_{-n} \Map^{\DAb}_{\cO_X}(\cF, \cG) = \pi_0 \Map^{\DAb}_{\cO_X}(\cF, \cG[n]) , \\
	\pi_{-n} \Map^{\DAb}_{\cO_{X\an}}(\cF\an, \cG\an) = \pi_0 \Map^{\DAb}_{\cO_{X\an}}(\cF\an, \cG\an[n]) ,
\end{gather*}
it suffices to treat the case $n = 0$.
Observe that $\pi_0 \Map^{\DAb}_{\cO_X}(\cF, \cG)$ can be identified with the global sections of the cohomology sheaf $\mathcal H^0( \RcHom_{\cO_X}(\cF, \cG))$.
We claim that there are canonical equivalences
\[ \cF \simeq \colim_n \tau_{\le n} \cF, \qquad \cG \simeq \lim_m \tau_{\ge m} \cG , \]
where the limits and colimits are computed in $\cO_X \Mod$.
Indeed, the first one is a consequence of the right completeness of the $t$-structure on $\cO_X \Mod$ (cf.\ \cite[Proposition 2.1.3]{DAG-VIII}).
The second one is a consequence of the hypercompleteness of $\cG$ and of \cref{prop:hypercompleteness_DAb}.
Using the fact that $(-)\an$ commutes with any colimits and invoking \cref{lem:analytification_inverse_limits}, we conclude that
\[ \cF\an \simeq \colim_n (\tau_{\le n} \cF)\an, \qquad \cG\an \simeq \lim_m (\tau_{\ge m} \cG)\an . \]
Moreover, the $t$-exactness of the analytification shows that $(\tau_{\le n} \cF)\an \simeq \tau_{\le n} \cF\an$ and $(\tau_{\ge m} \cG)\an \simeq \tau_{\ge m} \cG\an$.
We are therefore reduced to the case where $\cF \in \Coh^-(X)$ and $\cG \in \Coh^+(X)$.
Assume that $\mathcal H^i(\cF) = 0$ for $i \ge n_0$ and $\mathcal H^j(\cG) = 0$ for $j \le m_0$.
Then the same bounds hold for $\cF\an$ and $\cG\an$.
So we obtain:
\begin{gather*}
	\pi_0 \Map^{\DAb}_{\cO_X}(\cF, \cG) \simeq \pi_0 \Map^{\DAb}_{\cO_X}(\tau_{\ge m_0 - 1} \cF, \tau_{\le n_0 + 1} \cG) \\
	\pi_0 \Map^{\DAb}_{\cO_{X\an}}(\cF\an, \cG\an) \simeq \pi_0 \Map^{\DAb}_{\cO_X}(\tau_{\ge m_0 - 1} \cF\an, \tau_{\le n_0 + 1} \cG\an).
\end{gather*}
Since both $\tau_{\ge m_0 - 1} \cF$ and $\tau_{\le n_0 + 1} \cG$ belong to $\Cohb(X)$, we know that the canonical map
\[ \Map^{\DAb}_{\cO_X}(\tau_{\ge m_0 - 1} \cF, \tau_{\le n_0 + 1} \cG) \to \Map^{\DAb}_{\cO_X}(\tau_{\ge m_0 - 1} \cF\an, \tau_{\le n_0 + 1} \cG\an) \]
is an equivalence.
In conclusion, $\psi_{\cF, \cG}$ is an equivalence for every $\cF, \cG \in \Coh(X)$, completing the proof.
\end{proof}

\begin{thm} \label{thm:GAGA2}
Let $X$ be a proper algebraic stack over $\Spec A$, where $A$ is either the field of complex numbers or a $k$-affinoid algebra.
The analytification functor on coherent sheaves induces an equivalence of 1-categories
\begin{equation} \label{eq:GAGA2}
\Cohh(X)\longrightarrow\Cohh(X\an).
\end{equation}
\end{thm}

\begin{proof}
The full faithfulness follows from \cref{prop:GAGA_fully_faithful}.

Let us prove the essential surjectivity.
By descent of coherent sheaves, the statement is local on $\Spec A$.
So we can assume that there exists a scheme $P$ proper over $\Spec A$ and a surjective $A$-morphism $p\colon P\to X$.

We proceed by induction on the geometric level $n$ of $X$.
The case $n=-1$ is classical \cite{SGA1,Kopf_Eigentliche_1974,Conrad_Relative_2006}.
Now assume that the functor \eqref{eq:GAGA2} is an equivalence when the geometric level of $X$ is less than $n$.

Then we use noetherian induction as in the proof of \cref{thm:proper_direct_image_alg}.
So we can assume that $X$ is reduced.

By \cref{thm:generic_flatness}, there exists a dense open substack $U_0 \subset X$ over which the map $p$ is flat.
Let $\cF$ be a coherent sheaf on $X\an$.
Let $p^\bullet\colon P^\bullet\to X$ be the simplicial nerve of $p$ and put $\cF' \coloneqq \rR^0 p^{\bullet\text{an}}_* \;\rL^0 p^{\bullet\text{an} *}(\cF)$.
\cref{thm:proper_direct_image_anal} plus the spectral sequence of \cref{thm:spectral_sequence} shows that $\cF'$ is coherent.
By the induction hypothesis on the geometric level and \cref{thm:GAGA1}, we see that $\cF'$ is algebraizable.
By the noetherian induction hypothesis, it suffices to prove that the canonical morphism $\cF\to\cF'$ restricts to an isomorphism over the dense open substack $U_0$.

Let us consider the \kanal case first.
Since the question is local on $U_0$, it suffices to prove that for any $k$-affinoid algebra $B$ and any quasi-smooth morphism $V\coloneqq\Sp B\to U\an_0$, the pullback of the morphism $\cF\to\cF'$ to $V$ is an isomorphism.
Let $V\alg \coloneqq \Spec B$.
Let $\cF_V \coloneqq \cF\times_{X\an} V$ and $p_V \coloneqq p\times_X V\alg$.
Since $V$ is affinoid, the analytic coherent sheaf $\cF_V$ over $V$ can be regarded as an algebraic coherent sheaf over $V\alg$, which we denote by $\cF_V\alg$.

Since $p_V$ is proper and faithfully flat, the canonical morphism
\[\cF\alg_V\to \rR^0 p^\bullet_{V*} \;\rL^0 p^{\bullet *}_V (\cF\alg_V)\]
is an isomorphism by fppf descent.
By \cref{thm:GAGA1} and the spectral sequence of \cref{thm:spectral_sequence}, the same holds for the canonical morphism $\cF_V\to \rR^0 p^{\bullet\text{an}}_{V*} \; \rL^0 p^{\bullet\text{an}*}_V(\cF_V)$.
So we have proved the \kanal case.

Let us turn to the \canal case.
Since the question is local on $U_0$, it suffices prove that for any smooth morphism $U_1 \to U_0$ with $U_1$ representable, the pullback of $\cF \to \cF'$ to $U_1$ is an isomorphism.
For this, we only need to show that for every relatively compact Stein open $V \Subset U_1\an$, the pullback of $\cF \to \cF'$ to $V$ is an isomorphism.

Set $\cG^n \coloneqq \rL^0 (p^n)^{\mathrm{an} *}(\cF)$ and $\cG^\bullet \coloneqq \rL^0 p^{\bullet \mathrm{an} *}(\cF)$.
We have $\cF' = \rR^0 p^{\bullet \mathrm{an}}_* (\cG^\bullet)$.
Let $\cF_V \coloneqq \cF\times_{X\an} V$ and $\cG^\bullet_V \coloneqq \cG^\bullet \times_{X\an} V$.
By \cref{lem:canal_relatively_compact_globally_presented}, $\cF_V$ is of global finite presentation (cf. \cref{def:globally_presented}).
By \cref{lem:canal_global_sections_equivalence}, it determines a coherent sheaf on $\Spec(B)$, which we denote by $\cF_V^\mathrm{alg}$.
The induction hypothesis on the geometric level shows the existence of
\[
\cE^\bullet \in \varprojlim \Coh^+(P^\bullet)
\]
such that $\cE^{\bullet \mathrm{an}} = \cG^\bullet$.

Set $A \coloneqq \Gamma(\cO_{U_1\an})$ and $B \coloneqq \Gamma(\cO_V)$.
\cref{lem:open_immersions_are_flat} shows that $B$ is a flat $A$-algebra.
Form the pullback diagram of simplicial \emph{algebraic} stacks
\[
\begin{tikzcd}
P_V^\bullet \arrow{r}{p_V^\bullet} \arrow{d}[swap]{j^\bullet} & \Spec(B) \arrow{d}{i} \\
P_{U_1}^\bullet \arrow{r} & U_1
\end{tikzcd}
\]

For every $n \in \Z_{\ge 0}$, we consider the following diagram:
\[
\begin{tikzcd}[column sep=small, row sep=small]
{} & \Coh^+(P_{U_1}^n) \arrow{rr} \arrow{dd} \arrow{dl} & & \Coh^+((P_{U_1}^n)\an) \arrow{dd} \arrow{dl} \\
\Coh^+(P_V^n) \arrow[crossing over]{rr} \arrow{dd} & & \Coh^+((P_V^n)\an) \\
& \Coh^+(U_1) \arrow{rr} \arrow{dl} & & \Coh^+(U_1\an) \arrow{dl} \\
\Coh^+(\Spec(B)) \arrow{rr} & & \Coh^+(V) \arrow[crossing over, leftarrow]{uu}
\end{tikzcd}
\]

The square on the left side commutes by flat base change.
The square on the right side commutes because $V \to U_1\an$ is an open immersion.
The square in the back commutes by \cref{thm:GAGA1}.
The top and bottom squares commute by construction.
As a result, for every $\mathcal H \in \Coh^+(P^n_V)$ which is the pullback of an element in $\Coh^+(P^n_{U_1})$, one has
\[
(\rR^0 p^n_* (\mathcal H))\an \simeq \rR^0 (p^n)\an_* (\mathcal H\an).
\]
Using the spectral sequence of \cref{thm:spectral_sequence}, we conclude that for every $\mathcal H^\bullet \in \varprojlim \Coh^+(P_V^\bullet)$ which is the pullback of an element in $\varprojlim \Coh^+(P_{U_1}^\bullet)$, one has
\begin{equation}\label{eq:GAGA2_isom}
(\rR^0 p^\bullet_* (\mathcal H^\bullet))\an \simeq \rR^0 p^{\bullet \mathrm{an}}_* (\mathcal H^{\bullet \mathrm{an}}).
\end{equation}

Observe that $\cG_V^\bullet\coloneqq \rL^0 p^{\bullet *}_V (\cF_V^\mathrm{alg})$ is the pullback of $\cE^\bullet$.
Therefore the isomorphism \eqref{eq:GAGA2_isom} holds for $\mathcal H^\bullet = \cG^\bullet_V$.
Now the proof proceeds as in the \kanal case.
\end{proof}

\begin{cor} \label{cor:unbounded_GAGA2}
	Let $X$ be a proper algebraic stack over $\Spec A$, where $A$ is either the field of complex numbers or a $k$-affinoid algebra.
	The analytification functor on coherent sheaves induces an equivalence of \infcats
	\[ \Coh(X)\xrightarrow{\ \sim\ }\Coh(X\an). \]
\end{cor}

\begin{proof}
We know from \cref{prop:GAGA_fully_faithful} that the functor is fully faithful.
Let us prove the essential surjectivity.
By devissage (\cref{lem:devissage}), we see that \cref{thm:GAGA2} implies that
\begin{equation}
\Cohb(X)\longrightarrow\Cohb(X\an)
\end{equation}
is an equivalence.
Let $\cF \in \Coh^-(X\an)$.
By the hypercompleteness of $\cF$ and by \cref{prop:hypercompleteness_DAb}, we can write
\[ \cF \simeq \lim \tau_{\ge n} \cF . \]
Since the analytification functor $(-)\an$ is fully faithful and essentially surjective on $\Cohb(X\an)$, the diagram $\{\tau_{\ge n} \cF\}$ is the analytification of a tower $\{\cG_n\}$ in $\Cohb(X)$.
Moreover, using the fact that the functor $(-)\an$ is conservative and $t$-exact, we deduce that $\cG_n \in \Coh^{\ge n}(X) \cap \Cohb(X)$.
Set $\cG \coloneqq \lim \cG_n$.
We claim that $\cG \in \Coh^-(X)$ and that $\tau_{\le n} \cG \simeq \cG_n$.
In order to prove this we consider the canonical map $\cG_n \to \cG_{n+1}$.
Since $\cG_{n+1}$ belongs to $\Coh^{\ge n +1}(X)$, we can canonically factor this map as
\[ \begin{tikzcd}
	\cG_n \arrow{r} & \tau_{\ge n +1} \cG_n \arrow{r}{\alpha_n} & \cG_{n+1}
\end{tikzcd} \]
Using once more the fact that the functor $(-)\an$ is conservative and $t$-exact, we deduce that $\alpha_n$ is an equivalence.
Since the $t$-structure on $\Coh(X)$ is left complete, it follows from \cite[1.2.1.18]{Lurie_Higher_algebra} that $\tau_{\le n} \cG \simeq \cG_n$.
In particular, $\cG \in \Coh^-(X)$.
Now using \cref{lem:analytification_inverse_limits}, we obtain
\[ \cG\an \simeq \lim \cG_n\an \simeq \lim \tau_{\ge n} \cF \simeq \cF . \]
This proves that $(-)\an$ is essentially surjective on $\Coh^-(X\an)$.
	
Finally, let $\cF \in \Coh(X\an)$.
The same argument as above, applying to the truncations $\tau_{\le n} \cF$ and using the fact that $(-)\an$ commutes with colimits, shows that there exists $\cG \in \Coh(X)$ such that $\cG\an \simeq \cF$.
The proof is therefore complete.
\end{proof}

\section{Appendices}

\subsection{Generic flatness for higher algebraic stacks}

The goal of this section is to generalize the generic flatness theorem to higher algebraic stacks.

We use the geometric context $(\Aff,\tauet,\bPsm)$ for algebraic stacks introduced in \cref{sec:algebraic_stacks}.
Let $f \colon X \to Y$ be a morphism of algebraic stacks.
We define an $\cS$-valued presheaf $\widebar{Y}_{\mathrm{flat},f}$ on the category $\Aff$ as follows.
For any $S\in\Aff$, let
\[
\widebar{Y}^{\mathrm{disc}}_{\mathrm{flat},f}(S) \coloneqq \{g \in \pi_0 \mathrm{Map}_{\Sh(\Aff,\tauet)}(S,Y) \text{ such that } S \times_Y X \to S \textrm{ is flat}\}.
\]
We define $\widebar{Y}_{\mathrm{flat},f}(S)$ to be the pullback
\[
\begin{tikzcd}
\widebar{Y}_{\mathrm{flat},f}(S) \arrow{r} \arrow{d} & \mathrm{Map}(S,Y) \arrow{d} \\
\widebar{Y}^{\mathrm{disc}}_{\mathrm{flat},f}(S) \arrow{r} & \pi_0 \mathrm{Map}(S,Y).
\end{tikzcd}
\]
Let $Y_{\mathrm{flat},f}$ denote the sheafification of $\widebar{Y}_{\mathrm{flat},f}$.

\begin{lem} \label{lem:generic_flatness_representable}
Assume that $Y$ is a reduced noetherian scheme.
For every morphism of finite presentation $f \colon X \to Y$ with $X$ a geometric stack, the presheaf $\widebar{Y}_{\mathrm{flat},f}$ is a stack.
Moreover, it is representable by a dense open subscheme of $Y$.
\end{lem}
\begin{proof}
Let $\{U_i\to X\}$ be a smooth atlas for $X$, let $U \coloneqq \coprod U_i$ and let $p\colon U\to X$.
By the generic flatness theorem for schemes, the flat locus of $f \circ p$ is a dense open subscheme of $Y$, which we denote by $W$.
A map from a scheme $S$ to $Y$ factors through $W$ if and only if the pullback $U\times_Y S \to S$ is flat.
This pullback is flat if and only if $X\times_Y S\to S$ is flat, because $U$ is an atlas for $X$.
Therefore, we have proved that $\widebar{Y}_{\mathrm{flat},f} = W$.
\end{proof}

\begin{lem} \label{lem:generic_flatness}
Let $f \colon X \to Y$ be a finitely presented morphism of geometric stacks and let $V \to Y$ be a morphism from a scheme to $Y$. Let $g \colon V \times_Y X \to V$ be the morphism induced by base change. Then the natural diagram
\[
\begin{tikzcd}
\widebar{V}_{\mathrm{flat},g} \arrow{r} \arrow{d} & \widebar{Y}_{\mathrm{flat},f} \arrow{d} \\
V \arrow{r} & Y
\end{tikzcd}
\]
is a pullback diagram in $\PSh(\Aff)$.
\end{lem}

\begin{proof}
By the Yoneda lemma, it suffices to prove that for every affine scheme $S$, the canonical map
\[
\mathrm{Map}(S, \widebar{V}_{\mathrm{flat},g}) \to \mathrm{Map}(S, V) \times_{\mathrm{Map}(S, Y)} \mathrm{Map}(S, \widebar{Y}_{\mathrm{flat},f})
\]
is an equivalence.
By the definition of $\widebar{V}_{\mathrm{flat},g}$ and $\widebar{Y}_{\mathrm{flat},f}$, both the source and the target of the above map can be embedded in $\mathrm{Map}(S,V)$.
Therefore, it suffices to check that this map is an isomorphism on $\pi_0$.
Indeed, an element in
\[
\pi_0 \mathrm{Map}(S,V) \times_{\pi_0 \mathrm{Map}(S,Y)} \pi_0 \mathrm{Map}(S, \widebar{Y}_{\mathrm{flat},f})
\]
is just a morphism $\varphi \colon S \to V$ with the property that $S \times_Y X \to S$ is flat.
We note that a morphism $\varphi\colon S\to V$ factors through $\pi_0 \mathrm{Map}(S, \widebar{V}_{\mathrm{flat},g})$ if and only if $S \times_V (V \times_Y X)\to S$ is flat.
So we have completed the proof.
\end{proof}

\begin{thm} \label{thm:generic_flatness}
Let $f \colon X \to Y$ be a morphism of finite type between noetherian algebraic stacks, with $Y$ being reduced.
Then $Y_{\mathrm{flat},f}$ is a geometric stack and the natural morphism $Y_{\mathrm{flat},f} \to Y$ is a dense open immersion.
\end{thm}
\begin{proof}
Let $\{V_i\to Y\}$ be an atlas for $Y$ and let $V \coloneqq \coprod V_i$.
Since the sheafification functor is a left exact localization, applying the sheafification functor to the pullback square of \cref{lem:generic_flatness}, we get a pullback diagram in the \infcat of sheaves $\Sh(\Aff,\tauet)$.
\cref{lem:generic_flatness_representable} implies that the sheafification of $\widebar{V}_{\mathrm{flat},g}$ coincides with itself, and that the map $V_{\mathrm{flat},g} \to V$ is representable by a dense open immersion.

Moreover, we note that the morphism $V_{\mathrm{flat},g} \to Y_{\mathrm{flat},f}$ is a smooth effective epimorphism, so it defines a smooth atlas for $Y_{\mathrm{flat},f}$.
It follows that $Y_{\mathrm{flat},f}$ is a geometric stack and that the natural morphism $Y_{\mathrm{flat},f} \to Y$ is a dense open immersion.
\end{proof}

\subsection{A spectral sequence for descent} \label{sec:spectral_sequence}

Suppose we are given a coaugmented cosimplicial diagram $\cC_+^\bullet$ in the $\infty$-category of presentable stable $\infty$-categories.
Denote by $\cD$ the object $\cC^{-1}$ and by $\cC^\bullet$ the underlying cosimplicial object of $\cC_+^\bullet$.
There is a natural morphism
\[
f^\bullet_* \colon \varprojlim \cC^\bullet \to \cD.
\]
Suppose that the categories $\cC^n$ for $n \ge -1$ are equipped with $t$-structures.
The main goal of this section is to construct a spectral sequence converging to the homotopy groups of $f^\bullet_*(\cF)$ for every $\cF \in \varprojlim \cC^\bullet$.

Let $\LPr$ (resp.\ $\RPr$) denote the $\infty$-category of presentable $\infty$-categories with morphisms given by left adjoints (resp.\ right adjoints) (cf.\ \cite[5.5.3.1]{HTT}).
The first part of the construction can be performed in a greater generality for diagrams $K \to \LPr$, where $K$ is a (small) simplicial set.
Since $\RPr \simeq (\LPr)^\mathrm{op}$, any diagram $Z \colon K \to \LPr$ gives rise to another diagram $Z' \colon K^\mathrm{op} \to \RPr$, informally by passing to right adjoints.
For this reason, any functor $Z \colon K \to \LPr$ determines a presentable fibration $\mathcal Z \to K$ via Grothendieck construction (cf.\ \cite[5.5.3.3]{HTT}).
This fibration gives rise to two different objects in the category $\sSet^+_{/K}$ of marked simplicial sets over $K$:
$\mathcal Z^\natural_{\mathrm{cocart}} \coloneqq (\mathcal Z, \mathcal E_{\mathrm{cocart}})$, 
and $\mathcal Z^\natural_{\mathrm{cart}} \coloneqq (\mathcal Z, \mathcal E_{\mathrm{cart}})$,
where $\mathcal E_{\mathrm{cocart}}$ (resp.\ $\mathcal E_{\mathrm{cart}}$) denotes the collection of cocartesian (resp.\ cartesian) edges of $\mathcal Z \to K$.
We refer to \cite[§3]{HTT} for the theory of marked simplicial sets.
By \cite[3.3.3.2]{HTT}, we have
\[
\varprojlim_K Z \simeq \Map_{K}(K^\sharp, \mathcal Z^\natural_{\mathrm{cocart}}).
\]

Let $X, Y \colon K \to \LPr$ be two $K$-diagrams in $\LPr$ and let $\varphi \colon Y \to X$ be a natural transformation.
Passing to right adjoints, we obtain an induced natural transformation $\psi \colon X' \to Y'$ in $\RPr$, with the property that for every $s \in K$, $\psi_s$ is a right adjoint of $\varphi_s$.
Let $p \colon \cX \to K$ and $q \colon \cY \to K$ be the presentable fibrations determined by $X$ and $Y$.
The Grothendieck construction converts $\varphi$ into a functor $F_\mathrm{cocart} \colon \cY^\natural_{\mathrm{cocart}} \to \cX^\natural_{\mathrm{cocart}}$.
Dually, $\psi$ determines a functor $G_\mathrm{cart} \colon \cX^\natural_{\mathrm{cart}} \to \cY^\natural_{\mathrm{cart}}$.
Observe that $F_\mathrm{cocart}$ does not respect the cartesian structure on $\cY$ and $G_\mathrm{cart}$ does not respect the cocartesian structure on $\cX$.
Nevertheless, after forgetting the markings, the two morphisms $F \colon \cY \leftrightarrows \cX \colon G$ are adjoint to each other. 
To see this, it is convenient to interpret $\varphi \colon Y \to X$ as a functor $Z \colon K \times \Delta^1 \to \LPr$. Let $\pi \colon \cZ \to K \times \Delta^1$ be the presentable fibration associated to $Z$.
Since $Z |_{K \times \{0\}} \simeq X$ and $Z |_{K \times \{1\}} \simeq Y$ by definition, we obtain canonical equivalences
\[ \cX \simeq \cZ \times_{K \times \{0\}} (K \times \Delta^1), \qquad \cY \simeq \cZ \times_{K \times \{1\}} (K \times \Delta^1) \]
Let $\mathrm{pr} \colon K \times \Delta^1 \to \Delta^1$ be the canonical projection.
We remark that for every object $s \in K$ there is a canonical morphism $(s,0) \to (s,1)$ in $K \times \Delta^1$ which is both cartesian and cocartesian.
We claim that the composition
\[ \begin{tikzcd}
	\cZ \arrow{r}{\pi} & K \times \Delta^1 \arrow{r}{\mathrm{pr}} & \Delta^1
\end{tikzcd} \]
is both a cartesian and cocartesian fibration.
By symmetry it will be sufficient to prove that it is a cocartesian one.
Let therefore $x \in \cZ$ and suppose that $\mathrm{pr}(\pi(x)) = 0 \in \Delta^1$.
Let $s$ be the image of $\pi(x)$ in $K$ under the projection $K \times \Delta^1 \to K$ and denote by $e \colon (s,0) \to (s,1)$ the canonical morphism in $K \times \Delta^1$.
Since $\pi$ is a cocartesian fibration, we can choose a $\pi$-cocartesian lift $f \colon x \to y$ in $\cZ$ lying over $e$.
As we already remarked that $e$ is $(\mathrm{pr})$-cocartesian, we can now invoke \cite[2.4.1.3.(3)]{HTT} to conclude that $f$ is $(\mathrm{pr} \circ \pi)$-cocartesian.
At this point, since the functor $\mathrm{pr} \circ \pi \colon \cZ \to \Delta^1$ is corepresented by $F \colon \cY \to \cX$ and represented by $G \colon \cX \to \cY$ by construction,
we can conclude that $F$ and $G$ are adjoint to each other.

By composing with $F_\mathrm{cocart}$, we obtain a functor
\[
f \colon \varprojlim_K Y \simeq \Map^\flat_K(K^\sharp, \cY^\natural_{\mathrm{cocart}}) \to \Map^\flat_K(K^\sharp, \cX^\natural_{\mathrm{cocart}}) \simeq \varprojlim_K X.
\]
Since the limits are computed in $\LPr$, the functor $f$ admits a right adjoint
\[
g \colon \varprojlim_K X \to \varprojlim_K Y.
\]
The first goal of this section is to provide a useful factorization of the functor $g$.

The inclusion $K^\flat \to K^\sharp$ in $\sSet^+$ induces a natural transformation of functors $\Map^\flat_K(K^\sharp, -) \to \Map^\flat_K(K^\flat, -)$.
Evaluating this natural transformation on the morphism $F$ produces the following commutative diagram
\[
\begin{tikzcd}
\Map^\flat_K(K^\flat, \cX^\natural_{\mathrm{cocart}}) & \Map^\flat_K(K^\flat, \cY^\natural_{\mathrm{cocart}}) \arrow{l}[swap]{\widetilde{F}} \\
\Map^\flat_K(K^\sharp, \cX^\natural_{\mathrm{cocart}}) \arrow{u} & \Map^\flat_K(K^\sharp, \cY^\natural_{\mathrm{cocart}}). \arrow{u} \arrow{l}[swap]{f}
\end{tikzcd}
\]
Since
\[
\Map^\flat_K(K^\flat, \cX^\natural_{\mathrm{cocart}}) \simeq \Fun_K(K, \cX),
\]
the above diagram can be rewritten as
\begin{equation}
\begin{tikzcd} \label{diagram:spectral_sequence}
\Fun_K(K, \cX) & \Fun_K(K, \cY) \arrow{l}[swap]{\widetilde{F}} \\
\Map^\flat_K(K^\sharp, \cX^\natural_{\mathrm{cocart}}) \arrow{u} & \Map^\flat_K(K^\sharp, \cY^\natural_{\mathrm{cocart}}). \arrow{u} \arrow{l}[swap]{f}
\end{tikzcd}
\end{equation}
We see that $\widetilde F$ has a right adjoint
\[
\widetilde{G} \colon \Fun_K(K, \cX) \to \Fun_K(K,\cY)
\]
induced by composition with $G\colon\cX\to\cY$.

\begin{lem} \label{lem:limit_lax_limit}
Let $K$ be a small simplicial set and let $p \colon \cX \to K$ be a presentable fibration.
The functors of $\infty$-categories
\begin{gather*}
\Map^\flat_K(K^\sharp, \cX^\natural_{\mathrm{cart}}) \to \Map^\flat_K(K^\flat, \cX^\natural_{\mathrm{cart}})\\
\Map^\flat_K(K^\sharp, \cX^\natural_{\mathrm{cocart}}) \to \Map^\flat_K(K^\flat, \cX^\natural_{\mathrm{cocart}})
\end{gather*}
are fully faithful and admit right adjoints.
\end{lem}

\begin{proof}
Since a cartesian fibration over $K$ is the same as a cocartesian fibration over $K^\mathrm{op}$, it suffices to prove the statement for $\cX^\natural_{\mathrm{cart}}$.
Since $\Map^\flat_K(K^\sharp, \cX^\natural)$ is a sub-simplicial set of $\Map^\flat_K(K^\flat, \cX^\natural)$, and both are \infcats, this inclusion is fully faithful.
The existence of right adjoint is the content of \cite[5.5.3.17]{HTT}.
\end{proof}

From now on we assume that $Y$ is the constant diagram $K \to \LPr$ associated to a presentable $\infty$-category $\cD$.
In this case, the presentable fibration $\cY\to K$ associated to $Y$ is simply the projection $K \times \cD \to K$.

\begin{lem} \label{lem:lax_limit_cosimplicial_objects}
There exists a commutative diagram in $\mathrm h(\mathrm{Cat}_\infty)$\[
\begin{tikzcd}
\Map^\flat_K(K^\flat, \cY^\natural_{\mathrm{cocart}}) \arrow{r} & \Fun(K, \cD) \\
\Map^\flat_K(K^\sharp, \cY^\natural_{\mathrm{cocart}}) \arrow{r} \arrow{u} & \cD, \arrow{u}[swap]{c}
\end{tikzcd}
\]
where $c \colon \cD \to \Fun(K, \cD)$ is the functor induced by composition with $K \to \Delta^0$ and the horizontal morphisms are isomorphisms in $\mathrm h(\mathrm{Cat}_\infty)$.
\end{lem}

\begin{proof}
Since $\cD$ is an $\infty$-category, the functor $\cD \to \Delta^0$ is both a cartesian and a cocartesian fibration.
Moreover, cocartesian edges in $\cD$ are precisely the equivalences in $\cD$.
Base change induces a presentable fibration $K \times \cD \to K$ whose cocartesian edges are precisely those morphisms that project to equivalences in $\cD$.
To simplify notation, we will write $\cY^\natural$ instead of $\cY^\natural_{\mathrm{cocart}}$.
Let $K^\triangleright$ denote the right cone and $\{v\}$ the vertex of the cone.
Consider the commutative diagram in $\sSet$
\[
\begin{tikzcd}
\Map^\flat_K(K^\flat, \cY^\natural) & \Map^\flat_{K^\triangleright}((K^\triangleright)^\flat, \cY^\natural) \arrow{l} \arrow{r} & \Map^\flat_{K^\triangleright}(\{v\}^\flat, \cY^\natural) \\
\Map^\flat_K(K^\sharp, \cY^\natural) \arrow{u} & \Map^\flat_{K^\triangleright}((K^\triangleright)^\sharp, \cY^\natural) \arrow{l} \arrow{u} \arrow{r} & \Map^\flat_{K^\triangleright}(\{v\}^\sharp, \cY^\natural). \arrow{u}
\end{tikzcd}
\]
Since $\{v\}^\sharp = \{v\}^\flat$, the vertical morphism on the right is the identity.
The proof of \cite[3.3.3.2]{HTT} shows that the morphisms on the bottom of the diagram are categorical equivalences.
Moreover, $\Map^\flat_{K^\triangleright}(\{v\}^\flat, \cY^\natural) \simeq \cD$.
Observe that
\[
\Map^\flat_K(K^\flat, \cY^\natural) \simeq \Map_K(K, \cD \times K) \simeq \Fun(K, \cD).
\]
Similarly, we have the identification
\[
\Map^\flat_{K^\triangleright}((K^\triangleright)^\flat, \cY^\natural) \simeq \Fun(K^\triangleright, \cD).
\]
To conclude the proof, it suffices to note that the image of
\[
\Map^\flat_{K^\triangleright}((K^\triangleright)^\sharp, \cY^\natural) \to \Map^\flat_{K^\triangleright}((K^\triangleright)^\sharp, \cY^\natural) \simeq \Fun(K^\triangleright, \cD)
\]
consists precisely of the constant diagrams from $K$ to $\cD$. \end{proof}

\begin{cor}\label{cor:limit}
Let $Y \colon K \to \LPr$ be the constant diagram associated to the presentable $\infty$-category $\cD$ and let $\cY = K \times \cD \to K$ be the presentable fibration classified by $Y$. The right adjoint to the inclusion
\[
\Map^\flat_K(K^\sharp, \cY^\natural_{\mathrm{cocart}}) \to \Map^\flat_K(K^\flat, \cY^\natural_{\mathrm{cocart}})
\]
can be identified with the limit functor
\[
\lim \colon \Fun(K, \cD) \to \cD.
\]
\end{cor}

Combining \cref{lem:limit_lax_limit} and \cref{lem:lax_limit_cosimplicial_objects}, we can pass to the right adjoints in Diagram \eqref{diagram:spectral_sequence} and obtain the following commutative diagram of $\infty$-categories
\[
\begin{tikzcd}
\Fun_K(K, \cX) \arrow{r}{\widetilde{G}} \arrow{d} & \Fun(K, \cD) \arrow{d}{\lim}  \\
\Map^\flat_K(K^\sharp, \cX^\natural_{\mathrm{cocart}}) \arrow{r}{g} & \cD.
\end{tikzcd}
\]
Since the left vertical map is right adjoint to a fully faithful inclusion functor, we obtain the following factorization of the functor $g$
\begin{equation} \label{eq:factorization_descending_spectral_sequence}
\begin{tikzcd}
{} & \Fun(K, \cD) \arrow{dr}{\lim} \\
\varprojlim X \arrow{ur}{\theta} \arrow{rr}{g} & & \cD.
\end{tikzcd}
\end{equation}

The case of main interest for us is when $K = \mathrm{N}(\bDelta)$.
In this case, we can combine the Dold-Kan correspondence together with the spectral sequence of \cite[§1.2.2]{Lurie_Higher_algebra} to produce a spectral sequence for the functor $g$.

To fix notations, let $\cC^\bullet_+ \colon \bDelta_+ \to \LPr_{\mathrm{stab}}$ be a coaugmented cosimplicial presentable stable $\infty$-category. Let $\cD \coloneqq \cC^{-1}$ and let $\cC^\bullet$ be the underlying cosimplicial object of $\cC^\bullet_+$.
Let $\cD^\bullet$ denote the constant cosimplicial object associated to $\cD$.
We obtain a canonical map $f^\bullet \colon \cD^\bullet \to \cC^\bullet$.
Passing to limits, we obtain a functor
\[
f \colon \cD \to \varprojlim \cC^\bullet .
\]
Let $g$ denote its right adjoint as in the previous discussion.

Let $\cX \to \bDelta$ (resp.\ $\cY \to \bDelta$) be the presentable fibration associated to $\cC^\bullet$ (resp.\ $\cD^\bullet$) as before.
Recall that the morphism $f^\bullet$ induces a functor $F \colon \cX \to \cY$ whose right adjoint relative to $\bDelta$ is denoted by $G \colon \cY \to \cX$.
Moreover, we can canonically identify the functor induced by $G$ between the fibers over $[n]$ with the right adjoint of $f^n$, which we denote by $g^n$. We observe that the inclusion $[n] \to \bDelta$ induces canonical projection maps for every $n$
\[
p^n \colon \varprojlim \cC^\bullet \simeq \Map^\flat_{\bDelta}(\bDelta^\sharp, \cX^\natural_{\mathrm{cocart}}) \to \Map^\flat_{\bDelta}([n]^\sharp, \cX^\natural_{\mathrm{cocart}})  \simeq \cC^n.
\]
Unravelling the definitions, we obtain the following commutative diagram
\[
\begin{tikzcd}
\varprojlim \cC^\bullet \arrow{d}[swap]{p^n} \arrow{r}{\theta} & \Fun(\bDelta, \cD) \arrow{d} \\
\cC^n \arrow{r}{g^n} & \Fun ([n], \cD) \simeq \cD .
\end{tikzcd}
\]

\begin{rem}
	The factorization of \eqref{eq:factorization_descending_spectral_sequence} is stated in \cite[Lemma 1.3.13]{Drinfeld_On_some_finiteness_2013} in the special case where $\cC^\bullet$ is the cosimplicial diagram of the categories of quasi-coherent sheaves associated to a simplicial derived algebraic stack.
	Even in that case, the construction of the functor $\theta$ is nontrivial.
	Our considerations in this section provide a detailed uniform explanation.
\end{rem}

\begin{thm}\label{thm:spectral_sequence}
With the above notations, assume that $\cD$ admits a t-structure compatible with sequential limits (cf.\ \cite[1.2.2.12]{Lurie_Higher_algebra}).
For every $\cF \in \varprojlim \cC^\bullet$ there exists a converging spectral sequence
\[
E^{s,t}_1 = \pi_t(g^s(p^s(\cF))) \Rightarrow \pi_{s + t}(g(\cF)).
\]
\end{thm}

\begin{proof}
Let $\cF \in \varprojlim \cC^\bullet$.
Then $\theta(\cF)$ is a cosimplicial object in the stable $\infty$-category $\cD$. Using the $\infty$-categorical Dold-Kan correspondence, we can associate to $\theta(\cF)$ a filtered object in $\cD$.
Its associated spectral sequence $\{E^{s,t}_r\}$ converges in virtue of \cite[1.2.2.14]{Lurie_Higher_algebra}.
By \cite[1.2.4.4]{Lurie_Higher_algebra}, we can identify the complex $\{E^{*,t}_1, d_1\}$ with the normalized chain complex associated to the cosimplicial object $\pi_t(\theta(\cF))$ of the abelian category $\cD^\heartsuit$. Moreover, in degree $s$ the cosimplicial object $\theta(\cF)$ coincides simply with $g^s p^s(\cF)$.
It follows that we have a canonical identification
\[
E^{s,t}_1 \simeq \pi_t(g^s(p^s(\cF))),
\]
completing the proof.
\end{proof}

\begin{rem}
The constructions performed in this section seem to heavily depend on the model of $\infty$-categories via quasi-categories.
In fact, the models were mainly used to introduce the adjunctions in \cref{lem:limit_lax_limit}.
We remark that the adjunction
\[\Map^\flat_K(K^\sharp, \cX^\natural_{\mathrm{cart}}) \leftrightarrows \Map^\flat_K(K^\flat, \cX^\natural_{\mathrm{cart}})\]
can be understood in a model-independent as an adjunction between limits and lax limits
\[
\lim_K X \leftrightarrows \mathop{\mathrm{lax.lim}}_K X.
\]
The same holds for the other adjunction concerning cocartesian fibrations.
\end{rem}

\subsection{Complements on Stein complex analytic spaces}

In this section, we collect several results concerning Stein complex analytic spaces which we were not able to find in the literature.

\begin{defin} \label{def:globally_presented}
Let $X$ be a Stein complex analytic space.
A coherent sheaf $\cF \in \Cohh(X)$ is said to be of \emph{global finite presentation} if there exists an exact sequence in $\Cohh(X)$ of the form
\begin{equation} \label{eq:global_presentation}
\cO_X^m \to \cO_X^n \to \cF \to 0.
\end{equation}
We denote by $\Cohh_{\mathrm{gfp}}(X)$ the full subcategory of $\Cohh(X)$ spanned by coherent sheaves of global finite presentation.
\end{defin}

\begin{lem} \label{lem:canal_global_sections_equivalence}
Let $X$ be a Stein space and let $A \coloneqq \Gamma(\cO_X)$ be the algebra of global functions.
The global section functor $\Gamma \colon \Cohh(X) \to A \Mod$ restricts to an equivalence of categories
\[
\Cohh_{\mathrm{gfp}}(X) \simeq A \Mod^{\mathrm{fp}},
\]
where $A \Mod^{\mathrm{fp}}$ denotes the full subcategory of $A \Mod$ spanned by modules of finite presentation.
\end{lem}

\begin{proof}
Let $\cF \in \Cohh_{\mathrm{gfp}}(X)$.
Taking global sections in \cref{eq:global_presentation} and applying Cartan's theorem B, we see that $\Gamma(\cF)$ is an $A$-module of finite presentation.
It follows that $\Gamma$ restricts to a functor $\Phi \colon \Cohh_{\mathrm{gfp}}(X) \to A \Mod^{\mathrm{fp}}$.

To construct a quasi-inverse for $\Phi$, let us denote temporarily by $\mathrm{Op}'(X)$ the category of relatively compact Stein open subsets $U \Subset X$.
We introduce a Grothendieck topology on $\mathrm{Op}'(X)$ generated by coverings of the form $\{U_i\to U\}_{i\in I}$ where $U_i \Subset U$ for every $i \in I$ and $U=\bigcup U_i$.
Since the inclusion $\mathrm{Op}'(X) \to \mathrm{Op}(X)$ satisfies the assumptions of \cref{prop:equivalence_of_topoi}, we can identify sheaves on $\mathrm{Op}(X)$ with sheaves on $\mathrm{Op}'(X)$.

Now let $M \in A \Mod^{\mathrm{fp}}$.
We define the following presheaf of sets on $\mathrm{Op}'(X)$
\[\cF \colon \mathrm{Op}'(X)^\mathrm{op} \to \mathrm{Set}, \qquad U\mapsto M \widehat{\otimes}_A \cO_X(U).\]
We claim that $\cF$ is a sheaf on $\mathrm{Op}'(X)$.
Let
\[
A^m \to A^n \to M \to 0
\]
be a presentation for $M$.
The map $A^m \to A^n$ uniquely determines a morphism of sheaves $\cO_X^m \to \cO_X^n$.
Let $\cF'$ denote the cokernel.
It follows from \cite[Proposition 2]{Douady_Proper_1973} that $\cF'$ satisfies
\[
\cF'(U) \simeq \Gamma(\cF') \widehat{\otimes}_A \cO_X(U),
\]
for every $U\in\mathrm{Op}'(X)$.
Since $X$ is Stein, Cartan's theorem B shows that $\Gamma(\cF') \simeq M$. Therefore, the restriction of $\cF'$ to $\mathrm{Op}'(X)$ coincides with $\cF$.
This proves the claim.

We note that the construction of $\cF'$ is functorial on $M$.
So we obtain a functor
\[
\Psi \colon A \Mod^{\mathrm{fp}} \to \Cohh_{\mathrm{gfp}}(X) ,
\]
which is a quasi-inverse for $\Phi$.
\end{proof}

\begin{lem} \label{lem:canal_relatively_compact_globally_presented}
Let $X$ be a Stein space and $U \Subset X$ a relatively compact Stein open subset of $X$.
Then the restriction functor
\[
\Cohh(X) \to \Cohh(U)
\]
factors through the full subcategory $\Cohh_{\mathrm{gfp}}(U)$.
\end{lem}

\begin{proof}
Let $\cF \in \Cohh(X)$.
It is generated by global sections, by Cartan's theorem A.
Therefore, for every point $x\in\widebar{U}$, there exists a neighborhood $V_x$ and a finite subset of $\Gamma(\cF)$ which generates $\cF_{|V_x}$.
By the compactness of $\widebar{U}$, we obtain a finite subset of $\Gamma(\cF)$ which generates $\cF_{|U}$.
In other words we obtain a morphism of sheaves $\varphi \colon \cO_X^n \to \cF$ which restricts to an epimorphism over $U$.
Let $\cG$ be the kernel of $\varphi$.
Since $\cG$ is also a coherent sheaf, we can repeat the same argument to find a morphism $\cO_X^m \to \cG$ which restricts to an epimorphism over $U$.
So we obtain an exact sequence in $\Cohh(V)$
\[\cO_V^m \to \cO_V^n \to \cF_{|V} \to 0,\]
completing the proof.
\end{proof}

\begin{lem} \label{lem:open_immersions_are_flat}
Let $X$ be a Stein space and $U \Subset X$ a relatively compact Stein open subset of $X$.
Then $\Gamma(\cO_U)$ is flat as $\Gamma(\cO_X)$-algebra.
\end{lem}

\begin{proof}
Set $A \coloneqq \Gamma(\cO_X)$ and $B \coloneqq \Gamma(\cO_U)$.
Recall that $A \Mod$ is an $\omega$-presentable category and that $\omega$-presentable objects in $A \Mod$ are precisely the $A$-modules of finite presentation.
Therefore, in order to show that $B$ is flat as $A$-algebra, it suffices show that for every monomorphism $N\to M$ of finitely presented $A$-modules, the induced morphism
\[
M \otimes_A B \to N \otimes_A B
\]
is again a monomorphism.
Since both $M$ and $N$ are finitely presented, we have $M \otimes_A B = M \widehat{\otimes}_A B$ and $N \otimes_A B = N \widehat{\otimes}_A B$.
Let $\cF$ and $\cG$ be the globally presented coherent sheaves on $X$ associated to $M$ and $N$ respectively under the equivalence of \cref{lem:canal_global_sections_equivalence}.
Let $\cE$ be the cokernel of $\cF \to \cG$.
It is a coherent sheaf on $X$.
Since $\widebar{U}$ is a compact Stein subset, using Siu's theorem we can find a Stein open subset $V$ satisfying $U \subset V \Subset X$.
Invoking \cite[Proposition 2]{Douady_Proper_1973} we deduce that $\cE$ is transverse to $B$ over $\Gamma(\cO_V)$.
In particular, the morphism
\[
\cF(V) \widehat{\otimes}_{\Gamma(\cO_V)} B \to \cG(V) \widehat{\otimes}_{\Gamma(\cO_V)} B
\]
is a monomorphism. However, since $V \Subset X$ is Stein, \cite[Proposition 2]{Douady_Proper_1973} implies that
\[
\cF(V) = M \widehat{\otimes}_A \Gamma(\cO_V), \qquad \cG(V) = N \widehat{\otimes}_A \Gamma(\cO_V).
\]
We conclude that $M \otimes_A B \to N \otimes_A B$ is a monomorphism, completing the proof.
\end{proof}

\begin{lem} \label{lem:open_cover_faithfully_flat_cover}
Let $X$ be a Stein space and let $\{U_i\to X\}_{i \in I}$ be an open covering by relatively compact Stein open subsets.
Let $A \coloneqq \Gamma(\cO_X)$ and $B_i \coloneqq \Gamma(\cO_{U_i})$.
Then the family $\{A \to B_i\}_{i \in I}$ defines a faithfully flat covering of $A$.
\end{lem}

\begin{proof}
\cref{lem:open_immersions_are_flat} shows that each morphism $A \to B_i$ is flat.
It suffices show that the functors
\[
- \otimes_A B_i \colon A \Mod \to B_i \Mod
\]
are jointly surjective. Since $A \Mod$ is generated under filtered colimits by $A \Mod^{\mathrm{fp}}$, it suffices show that if $M \in A \Mod^\mathrm{fp}$ becomes zero after tensoring with each $B_i$, then $M = 0$.
Let $\cF$ be the coherent sheaf corresponding to $M$ under the equivalence of \cref{lem:canal_global_sections_equivalence}.
If $M \otimes_A B_i = 0$, \cite[Proposition 2]{Douady_Proper_1973} shows that $\cF(U_i) = 0$.
Since $\cF_{|U_i}$ is globally presented, \cref{lem:canal_global_sections_equivalence} implies that $\cF_{|U_i} = 0$.
Since the $\{U_i\}$ is a covering of $X$, we deduce that $\cF = 0$, thus $M = 0$, completing the proof.
\end{proof}

\bibliographystyle{plain}
\bibliography{dahema}

\end{document}